\documentclass[reqno]{amsart}
\usepackage[margin=1in]{geometry}

\usepackage{eurosym}
\usepackage{amsmath}
\usepackage{amssymb,amsthm}
\usepackage{amsfonts,amsrefs}
\usepackage[pdfencoding=unicode,psdextra]{hyperref}
\usepackage{scalerel}
\usepackage{graphicx,mathtools,caption,subcaption}
\usepackage{float}

\usepackage{xcolor}

\usepackage[export]{adjustbox}

\theoremstyle{plain} 
\newtheorem{theorem}{Theorem}[section]

\newtheorem{lemma}[theorem]{Lemma}

\theoremstyle{definition}
\newtheorem{definition}[theorem]{Definition}
\newtheorem{remark}[theorem]{Remark}

\usepackage{chngcntr}
\counterwithin{figure}{section}

\usepackage[T3,OT1]{fontenc}
\DeclareSymbolFont{tipa}{T3}{cmr}{m}{n}
\DeclareMathAccent{\invbreve}{\mathalpha}{tipa}{16}

\makeatletter
\newsavebox{\@brx}
\newcommand{\llangle}[1][]{\savebox{\@brx}{\(\m@th{#1\langle}\)}%
  \mathopen{\copy\@brx\mkern2mu\kern-0.9\wd\@brx\usebox{\@brx}}}
\newcommand{\rrangle}[1][]{\savebox{\@brx}{\(\m@th{#1\rangle}\)}%
  \mathclose{\copy\@brx\mkern2mu\kern-0.9\wd\@brx\usebox{\@brx}}}
\makeatother

\usepackage{enumerate}

\title{Basis for KBSM of fibered torus with multiplicity two exceptional fiber}

\author{Mieczyslaw K. Dabkowski}
\address{Department of Mathematical Sciences, The University of Texas at Dallas, Richardson, TX 75080}
\email{mdab@utdallas.edu}

\author{Cheyu Wu}
\address{Department of Mathematical Sciences, The University of Texas at Dallas, Richardson, TX 75080}
\email{cheyu.wu@utdallas.edu}

\begin{document}

\begin{abstract}
We construct a family of bases for the Kauffman bracket skein module (KBSM) of the product of an annulus and a circle. Using these bases, we find a new basis for the KBSM of $(\beta,2)$-fibered torus as a first step toward developing techniques for computing KBSM of a family of small Seifert fibered $3$-manifolds.
\end{abstract}

\maketitle

\section{Introduction}
\label{s:intro}

Skein modules were first introduced by J.H.~Przytycki in 1987 \cite{Prz1991} and later independently by V.G.~Turaev in 1988 \cite{Tur1990}. Among these, the skein module based on the Kauffman bracket skein relation in $S^{3}$ (see \cite{KLH1987}) has been the most extensively studied and we call it \emph{Kauffman bracket skein module} (KBSM). The main objective for this paper is to construct a basis for KBSM of $(\beta,2)$-fibered torus\footnote{A fibered torus is homeomorphic to a solid torus $V^{3}$ so, in this paper, we constructed a new basis for KBSM of $V^{3}$ (see also Remark~\ref{rem:Final_Rem_Fib_Torus}).} and, in the follow-up work, we plan to use this basis to compute KBSM for families of lens spaces and a family of small Seifert fibered $3$-manifolds.

A framed link in an oriented $3$-manifold $M$ is a disjoint union of smoothly embedded circles equipped with a non-zero normal vector field. Let $R$ be a commutative ring with identity. We fix an invertible element $A \in R$ and denote by $R\mathcal{L}^{fr}$ the free $R$ module with basis $\mathcal{L}^{fr}$ consisting of ambient isotopy classes of framed links in $M$ (also including the empty set as a framed link). Let $S_{2,\infty}$ be the submodule of $R\mathcal{L}^{fr}$ generated by
\begin{equation*}
L_{+} - AL_{0} - A^{-1}L_{\infty} \quad \text{and} \quad L \sqcup T_{1} + (A^{-2}+A^{2})L,
\end{equation*}
where framed links $L_{+},\, L_{0},\, L_{\infty}$ are identical outside of a $3$-ball and differ inside of it as shown on the left of Figure~\ref{fig:skeinrelation}; $L\sqcup T_{1}$ on the right of Figure~\ref{fig:skeinrelation} is the disjoint union of $L$ and the framed trivial knot $T_{1}$ (i.e., $T_{1}$ is contained in a $3$-ball disjoint from $L$). The \emph{Kauffman bracket skein module} (KBSM) of $M$ is defined as the quotient module of $R\mathcal{L}^{fr}$ by $S_{2,\infty}$, i.e.,
\begin{equation*}
\mathcal{S}_{2,\infty}(M;R,A) = R\mathcal{L}^{fr}/S_{2,\infty}.
\end{equation*}

\begin{figure}[ht]
\centering
\includegraphics[scale=0.6]{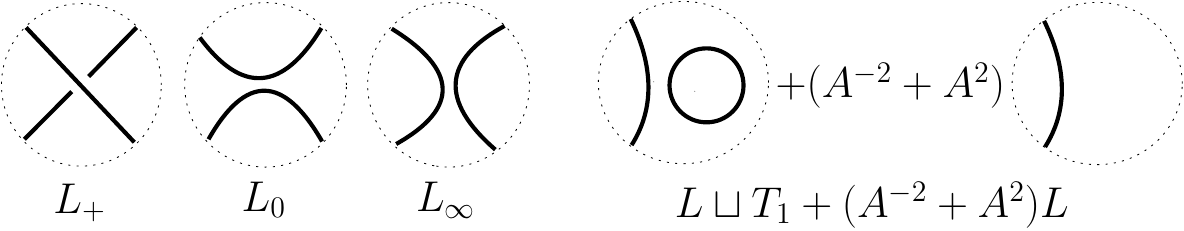}
\caption{Skein triple $L_{+}$, $L_{0}$, $L_{\infty}$ and relation $L\sqcup T_{1} + (A^{-2}+A^{2})L$}
\label{fig:skeinrelation}
\end{figure}

We organize this paper as follows. In Section~\ref{s:ambient_isotopy_In_Fibered_torus}, we discuss a model for a $(\beta,2)$-fibered torus. Using this model, we represent framed links and describe their isotopy by arrow diagrams and arrow moves on $2$-disk with a mark point (see Theorem~\ref{thm:AmbientIsotopiesInVBeta2}). In Section~\ref{s:basis_A2TimesS1}, we construct a family of bases for KBSM of the product of an annulus and a circle. In Section~\ref{s:basis_V} we find a basis for the $(\beta,2)$-fibered torus in terms of arrow diagrams in a $2$-disk with a marked point.

\section{Isotopy of framed links in \texorpdfstring{$(\beta,2)$}{(\unichar{0946},2)}-fibered torus}

\label{s:ambient_isotopy_In_Fibered_torus}

By a $(\beta,2)$-fibered torus $V(\beta,2)$ we mean a $3$-manifold obtained by $(2,\beta)$-Dehn filling of the inner boundary torus of ${\bf A}^{2} \times S^{1}$ (i.e., the product of an annulus ${\bf A}^{2}$ and a circle $S^{1}$). It follows from Corollary~6.3 of \cite{Hud1969} that ambient isotopy of links (framed links) in $V(\beta,2)$ is a composition of a finite number of \emph{moves}. We will always assume that a link (also a framed link) $L$ in $V(\beta,2)$ is inside ${\bf A}^{2}\times S^{1}$. We say that a link $L$ is \emph{generic} if (i) $L$ is transversal to fibers of the natural projection $p: {\bf A}^{2}\times S^{1} \to {\bf A}^{2}$; (ii) $L$ intersects ${\bf A}^{2}$ transversally at finitely many points, and (iii) immersed curves $p(L_{i})$ in ${\bf A}^{2}$ intersect transversally, where $L_{i}$ are components of $L$. Using general position arguments, one shows that every link in $V(\beta,2)$ is isotopic to a generic link. Such a link can be assigned a standard framing by choosing the unit vector field tangent to fibers of $p$ and co-oriented with the fibers. A framed link with standard framing is called a \emph{generic framed link} and we note that each framed link in $V(\beta,2)$ is ambient isotopic to such a link. Generic framed links can be described on ${\bf A}^{2}$ by \emph{arrow diagrams} (see Figure~\ref{fig:projectiondiagram_2}) introduced in \cite{MD2009}. 

\begin{figure}[ht]
\centering
\includegraphics[scale=0.5]{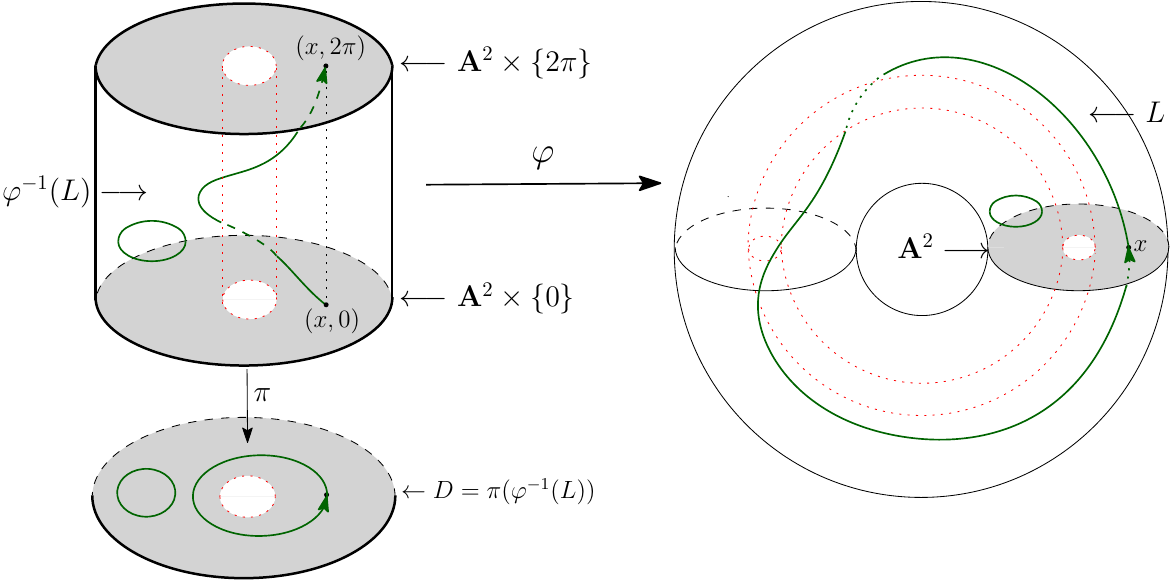}
\caption{Arrow diagram of a link in ${\bf A}^{2}\times S^{1}$}
\label{fig:projectiondiagram_2}
\end{figure}

The ambient isotopies of the generic framed links in ${\bf A}^{2} \times S^{1}$ can be described by \emph{arrow moves} on their arrow diagrams in ${\bf A}^{2}$ (see Figure~\ref{fig:ArrowMoves}). 
\begin{figure}[ht]
\centering
\includegraphics[scale=0.8]{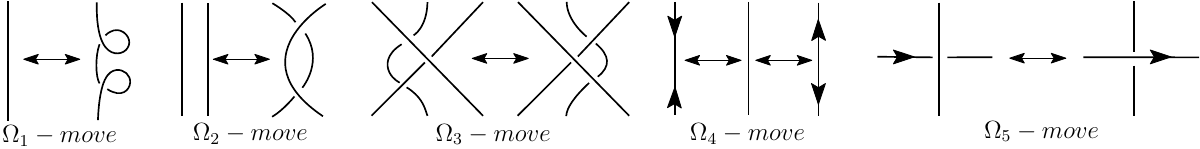}
\caption{Arrow moves $\Omega_{1}-\Omega_{5}$}
\label{fig:ArrowMoves}
\end{figure}
As observed in \cite{MD2009}, the arrow diagrams of generic framed links up to arrow moves $\Omega_{1}-\Omega_{5}$ on ${\bf A}^{2}$ are in bijection with ambient isotopy classes of framed links\footnote{In fact, this is true for any oriented surface ${\bf F}$ (see Theorem~1 of \cite{GM2017}). For $M = F \times S^{1}$, where $F$ is a planar surface, Theorem~1 is a consequence of Theorem~3.2 of \cite{Tur1992}, i.e., descriptions of generic framed links and their isotopies in $M$ by shadow diagrams with shadow moves and by arrow diagrams with arrow moves are equivalent. For a non-planar oriented surface ${\bf F}\neq S^{2}$, Theorem~3.2 of \cite{Tur1992} does not imply Theorem~1.} in ${\bf A}^{2}\times S^{1}$.

We simplify arrow diagrams by applying the $\Omega_{4}$-move to each arc, ensuring that all arrows on a given arc point in the same direction. Following this convention, an arrow labeled with an integer $n$ on an arc $\alpha$ (see Figure~\ref{fig:Reduced_Arrow_Diagrams}) represents $|n|$ arrows on $\alpha$, pointing in the same direction if $n \geq 0$ and in the opposite direction if $n < 0$.

\begin{figure}[H]
\centering
\includegraphics[scale=0.7]{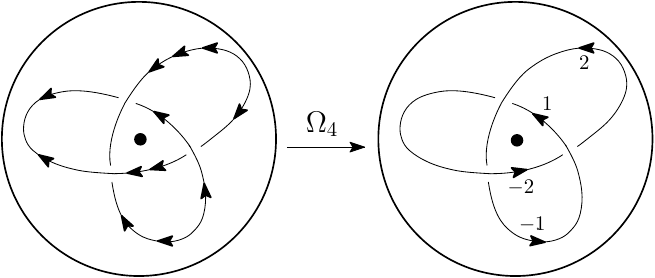}
\caption{Simplified arrow diagram}
\label{fig:Reduced_Arrow_Diagrams}
\end{figure}

There are two types of moves in $V(\beta,2)$ that generate all the ambient isotopes of generic framed links. Moves of the first type occur in normal cylinders ${\bf D}^{2}\times I$ inside ${\bf A}^{2}\times S^{1}$ and of the second type, called \emph{$2$-handle slide}, occur in a normal cylinder attached along the $(2,\beta)$-curve to the inner boundary torus ${\bf T}^{2}$ of ${\bf A}^{2}\times S^{1}$ (see Figure~\ref{fig:HandleSlidingS5}). 

\begin{figure}[H]
\centering
\includegraphics[scale=0.6]{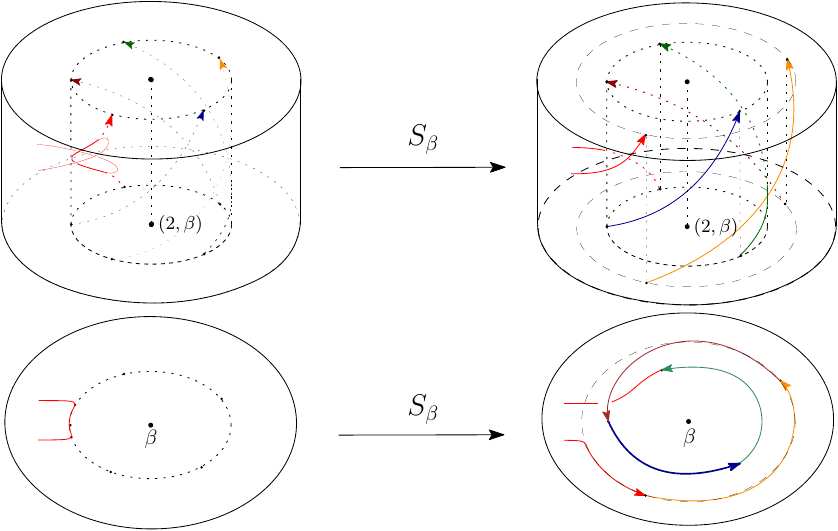}
\caption{$S_{\beta}$-move on ${\bf D}^{2}_{\beta}$ with $\beta = 5$}
\label{fig:HandleSlidingS5}
\end{figure}

We will represent the $2$-handle slide as a move $S_{\beta}$ between arrow diagrams in a closed $2$-disk ${\bf D}^{2}_{\beta}$ with a marked point $\beta$ (see Figure~\ref{fig:DiskWithMarkedPointHatD2}).

\begin{figure}[H]
\centering
\includegraphics[scale=0.6]{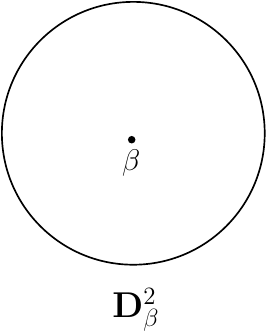}
\caption{${\bf D}^{2}$ with a marked point $\beta$}
\label{fig:DiskWithMarkedPointHatD2}
\end{figure}

\begin{lemma}
\label{lem:SbetaMoveOnD2Beta}
Generic framed links in $V(\beta,2)$ are related by a $2$-handle slide if and only if their arrow diagrams in ${\bf D}^{2}_{\beta}$ differ by a $S_{\beta}$-move shown in Figure~\ref{fig:ArrowDiagramForHandleSlidingSBeta}.
\end{lemma}

\begin{figure}[H]
\centering
\includegraphics[scale=0.8]{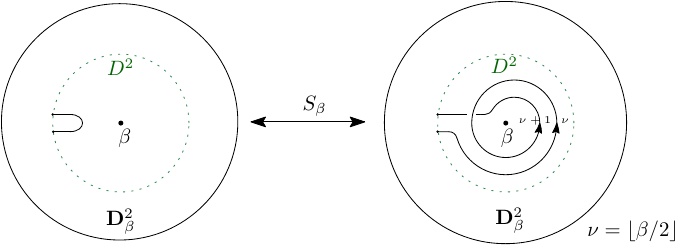}
\caption{$S_{\beta}$-move on ${\bf D}^{2}_{\beta}$}
\label{fig:ArrowDiagramForHandleSlidingSBeta}
\end{figure}

Consequently, an ambient isotopy of framed links in $V(\beta,2)$ can be described as follows.

\begin{theorem}
\label{thm:AmbientIsotopiesInVBeta2}
Generic framed links in $V(\beta,2)$ are ambient isotopic if and only if their arrow diagrams in ${\bf D}^{2}_{\beta}$ are related by a finite sequence of $\Omega_{1}-\Omega_{5}$ (see Figure~\ref{fig:ArrowMoves}) and $S_{\beta}$-moves (see Figure~\ref{fig:ArrowDiagramForHandleSlidingSBeta}).
\end{theorem}

\section{Basis for KBSM of \texorpdfstring{${\bf A}^{2}\times S^{1}$}{A\unichar{0178}\unichar{0215}S\unichar{0185}}}
\label{s:basis_A2TimesS1}

Let ${\bf F}$ be an oriented surface and let $R = \mathbb{Z}[A^{\pm 1}]$. Denote by $\mathcal{D}({\bf F})$ the set of equivalence classes of arrow diagrams (including the empty diagram) in ${\bf F}$ modulo $\Omega_{1}-\Omega_{5}$ moves and let $R\mathcal{D}({\bf F})$ be the free $R$-module with basis $\mathcal{D}({\bf F})$. Let $S_{2,\infty}({\bf F})$ be submodule of $R\mathcal{D}({\bf F})$ generated by
\begin{equation*}
D_{+} - AD_{0}-A^{-1}D_{\infty}\,\,\text{and}\,\,D\sqcup T_{1} +(A^{2}+A^{-2})D,
\end{equation*}
for all skein triples $D_{+}$, $D_{0}$, and $D_{\infty}$ of arrow diagrams in ${\bf F}$ and disjoint unions $D\sqcup T_{1}$ of arrow diagram $D$ and the trivial circle $T_{1}$ (see Figure~\ref{fig:SkeinTripleOfDiagrams}).

\begin{figure}[ht]
\centering
\includegraphics[scale=0.6]{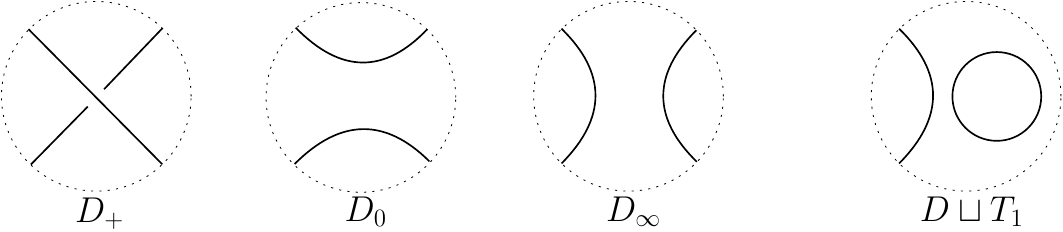}
\caption{Skein triple $D_{+}$, $D_{0}$, $D_{\infty}$, and disjoint union $D\sqcup T_{1}$}
\label{fig:SkeinTripleOfDiagrams}
\end{figure}

Since isotopy classes of framed links in ${\bf F}\times S^{1}$ are in one-to-one correspondence with equivalence classes of arrow diagrams in ${\bf F}$ modulo $\Omega_{1}-\Omega_{5}$ moves, the quotient module
\begin{equation*}
S\mathcal{D}({\bf F}) = R\mathcal{D}({\bf F})/S_{2,\infty}({\bf F}) 
\end{equation*}
is the KBSM of ${\bf F}\times S^{1}$. As shown in \cite{Prz1999} (see Theorem~2.3), KBSM of ${\bf F}\times I$ is a free $R$-module with a basis consisting of all simple closed curves on ${\bf F}$ (including the empty curve) and their parallel copies. Furthermore, new bases for $S\mathcal{D}({\bf D}^{2})$, $S\mathcal{D}({\bf A}^{2})$, and $S\mathcal{D}({\bf F}_{0,3})$, where ${\bf F}_{0,3}$ is a compact planar surface with three boundary components, were also found in \cite{MD2009}. However, the bases introduced in \cite{Prz1999} and \cite{MD2009} do not appear to be well suited for deriving the main result in Section~\ref{s:basis_V} (see Theorem~\ref{thm:basis_for_V(beta,2)}). In this section, we construct a family of bases for the KBSM of ${\bf A}^{2}\times S^{1}$, parameterized by an integer $c$. 

Let $\lambda$ represent the curve shown to the left in Figure~\ref{fig:LambdaCurveOnD2}, and let $\lambda^{k}$ denote the disjoint union of $k\geq 0$ copies of $\lambda$, with the convention that $\lambda^{0} = 1$ represents the empty arrow diagram.

\begin{figure}[H]
\centering
\includegraphics[scale=0.6]{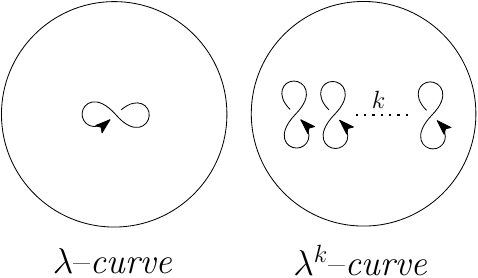}
\caption{Arrow diagrams of $\lambda$ and $\lambda^{k}$ on ${\bf D}^{2}$}
\label{fig:LambdaCurveOnD2}
\end{figure}

As shown in \cite{MD2009} (see Proposition~3.7), $S\mathcal{D}({\bf D}^{2})$ is a free $R$-module with basis of $(t_{1})^{k}$-curves\footnote{In \cite{MD2009}, $t_{1}$-curve is denoted by $x$.} in ${\bf D}^{2}$ (see Figure~\ref{fig:Relation_For_T}). To represent an arrow diagram $D$ in this basis, we simply compute $\langle D \rangle_{r}$ according to Definition~3.5 in \cite{MD2009}. However, in this paper, we will use its modified version which first computes the original bracket and then replaces $t_{1}$-curves with $-A^{3}\lambda$-curves. As shown in \cite{MD2009} the bracket $\langle\cdot \rangle_{r}$ (thus also its modified version) is invariant under $\Omega_{1}-\Omega_{5}$-moves. 

\begin{figure}[H]
\centering
\includegraphics[scale=0.8]{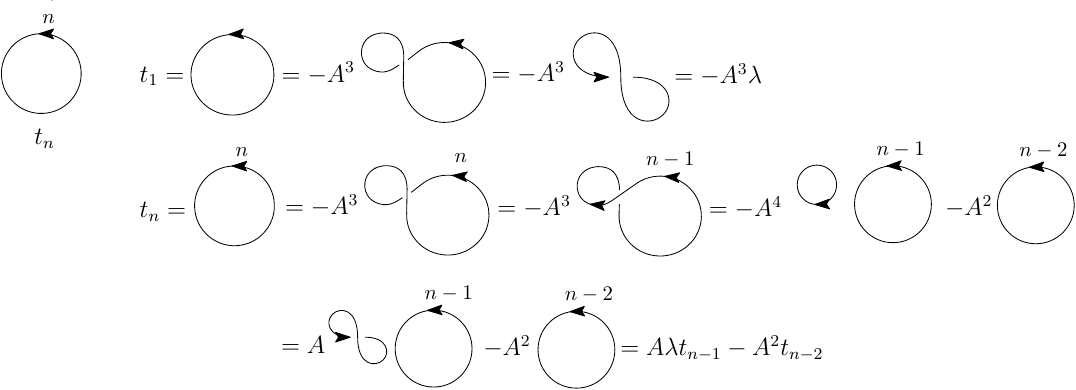}
\caption{$t_{n}$-curve and relation $t_{n} = A\lambda t_{n-1}-A^{2}t_{n-2}$}
\label{fig:Relation_For_T}
\end{figure}

We will use later in this paper the following identities in $S\mathcal{D}({\bf D}^{2})$: For the curve $t_{n}$ on the left in Figure~\ref{fig:Relation_For_T},
\begin{equation*}
t_{1} = -A^{3}\lambda,\,\, t_{0} = -A^{2}-A^{-2},\,\,\text{and}\,\, t_{n} - A \lambda t_{n-1} + A^{2}t_{n-2} = 0.
\end{equation*}
Therefore, if we let 
\begin{equation*}
P_{n} = \langle t_{n} \rangle_{r},
\end{equation*}
then clearly $P_{n}$ is in $R[\lambda]$, $P_{0} = -A^{2}-A^{-2}$, $P_{1} = -A^{3} \lambda$, and
\begin{equation}
\label{eqn:Pn}
P_{n} - A \lambda P_{n-1} + A^{2} P_{n-2} = 0.
\end{equation}

\begin{figure}[H]
\centering
\includegraphics[scale=0.7]{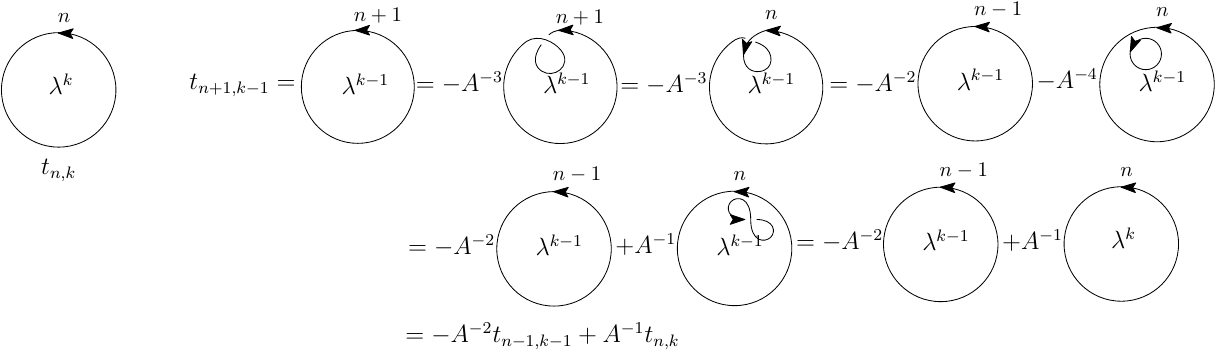}
\caption{$t_{n,k}$-curve and relation $t_{n+1,k-1} = -A^{-2}t_{n-1,k-1} + A^{-1}t_{n,k}$}
\label{fig:Relation_For_Tnk}
\end{figure}

Analogously, for a curve $t_{n,k}$ on the left of Figure~\ref{fig:Relation_For_Tnk}, the following identity holds in $S\mathcal{D}({\bf D}^{2})$:
\begin{equation*}
t_{n,0} = t_{n} \,\, \text{and} \,\, t_{n,k} = At_{n+1,k-1} + A^{-1}t_{n-1,k-1}.
\end{equation*}
Therefore, if we let 
\begin{equation*}
P_{n,k} = \langle t_{n,k} \rangle_{r},
\end{equation*}
then clearly $P_{n,k}$ is in $R[\lambda]$, $P_{n,0} = P_{n}$, and
\begin{equation}
\label{eqn:Pnk}
P_{n,k} = A P_{n+1,k-1} + A^{-1} P_{n-1,k-1}.
\end{equation}

Moreover, as it is demonstrated in Figure~\ref{fig:RelationForTnkInDisk}, the following identity also holds
\begin{equation}
\label{eqn:Pnk_1}
P_{n,k} = A\lambda P_{n-1,k} - A^{2}P_{n-2,k}.
\end{equation} 

\begin{figure}[H]
\centering
\includegraphics[scale=0.8]{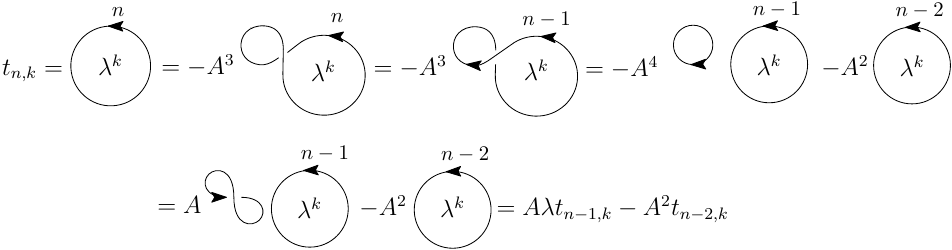}
\caption{Relation $t_{n,k} = A\lambda t_{n-1,k} - A^{2} t_{n-2,k}$} 
\label{fig:RelationForTnkInDisk}
\end{figure}

Let $x_{n}$ represent the curve on ${\bf A}^{2}$ shown in Figure~\ref{fig:CurveXnOnAnnulus}.

\begin{figure}[H]
\centering
\includegraphics[scale=0.7]{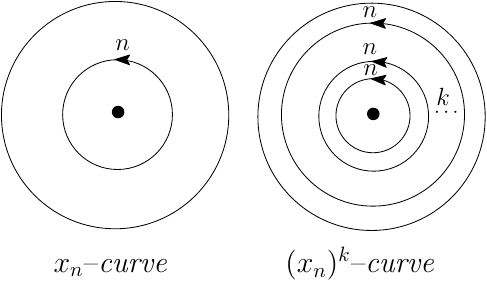}
\caption{Curve $x_{n}$ and $(x_{n})^k$ on ${\bf{A}}^{2}$}
\label{fig:CurveXnOnAnnulus}
\end{figure}

We denote by $(x_{n})^{k}$ the $k$ parallel copies of $x_{n}$, with the convention that $(x_{n})^{0} = 1$ represents the empty arrow diagram. Fix $c\in \mathbb{Z}$, and let
\begin{equation*}
\Sigma_{c} = \{\lambda^{n}, x_{c}\lambda^{n}(x_{c})^{k}, x_{c+1}\lambda^{n}(x_{c})^{k} \mid \ n,k \geq 0 \}\subset \mathcal{D}({\bf A}^{2}),
\end{equation*}
where curves appearing in the word $x_{c+\varepsilon}\lambda^{n}(x_{c})^{k}$ ($\varepsilon = 0,\,1$) are ordered from the inner to the outer circle of ${\bf{A}}^{2}$. We show in this section that
\begin{equation*}
S\mathcal{D}({\bf A}^{2}) \cong R\Sigma_{c}.
\end{equation*}
It suffices to construct a surjective homomorphism $\psi_{c}$ from $R\mathcal{D}({\bf A}^{2})$ to $R\Sigma_{c}$ that descends to an isomorphism of $S\mathcal{D}({\bf A}^{2})$ and $R\Sigma_{c}$. To construct $\psi_{c}$, we start from an arrow diagram $D$ in ${\bf A}^{2}$, we assign to each of its crossings either a positive marker or a negative marker, as shown in Figure~\ref{fig:markers}.

\begin{figure}[H]
\centering
\includegraphics[scale=0.9]{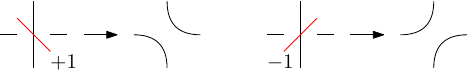}
\caption{Positive and negative markers}
\label{fig:markers}
\end{figure}

Such an assignment $s$ is called a \emph{Kauffman state} of $D$, and we denote by $\mathcal{K}(D)$ the set of all Kauffman states of $D$. For $s\in \mathcal{K}(D)$, let $D_{s}$ denote the arrow diagram (e.g., Figure~\ref{fig:TrefoilOnA2withMarkers}) obtained from $D$ by smoothing all its crossings according to the rule shown in Figure~\ref{fig:markers}.

\begin{figure}[H]
\centering
\includegraphics[scale=0.6]{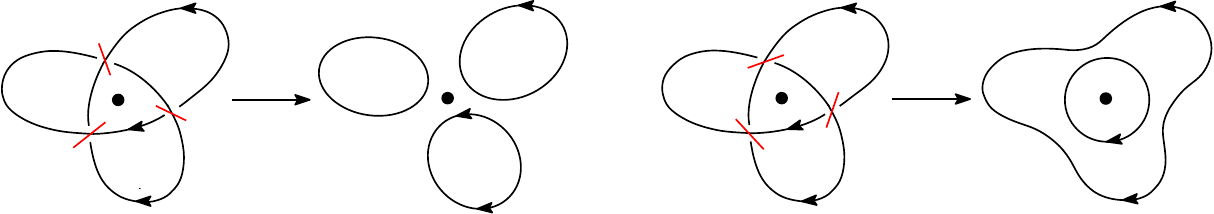}
\caption{Kaufman states for an arrow diagram on ${\bf{A}}^{2}$}
\label{fig:TrefoilOnA2withMarkers}
\end{figure}

Let $D$ be an arrow diagram in ${\bf{A}}^{2}$ without crossings, starting from the inner circle of ${\bf{A}}^{2}$, we denote by $D_{0}$ the first curve appearing before the $x$-curve, by $D_{1}$ the curve between the first and second copies of the $x$-curve, and so on (see Figure~\ref{fig:ArrowDiagramOnA2witoutCrossings}). Thus, $D$ can be regarded as a word:
\begin{equation*}
D = D_{0}x_{m_{1}}D_{1}\ldots D_{k-1}x_{m_{k}}D_{k},   
\end{equation*}
where each $D_{i}\subset {\bf D}^{2}$ has no crossings.

\begin{figure}[H]
\centering
\includegraphics[scale=0.7]{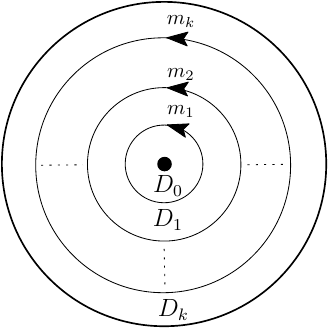}
\caption{Arrow diagram $D$ with no crossings in ${\bf A}^{2}$}
\label{fig:ArrowDiagramOnA2witoutCrossings}
\end{figure}

Let
\begin{equation*}
\langle D_{s} \rangle = (-A^{-2}-A^{2})^{|D_{s}|}D(s), 
\end{equation*}
where $|D_{s}|$ is the number of trivial circles of $D_{s}$, and $D(s)$ is the arrow diagram obtained from $D_{s}$ after removing trivial circles. Define
\begin{equation*}
\llangle D \rrangle = \sum_{s\in \mathcal{K}(D)}A^{p(s)-n(s)}\langle D_{s}\rangle,
\end{equation*}
where $p(s)$ and $n(s)$ is the number of positive and negative markers of $s$. 

It follows from the definition of $\llangle \cdot \rrangle$ that if arrow diagrams $D$ and $D'$ related by $\Omega_{1}-\Omega_{4}$-moves, then
\begin{equation*}
\llangle D \rrangle = \llangle D' \rrangle
\end{equation*}
and for each skein triple $D_{+}$, $D_{0}$, $D_{\infty}$, and disjoint unions $D\sqcup T_{1}$,
\begin{equation*}
\llangle D_{+} \rrangle - A \llangle D_{0} \rrangle - A^{-1}\llangle D_{\infty} \rrangle = 0\,\,\text{and}\,\,\llangle D\sqcup T_{1} \rrangle + (A^{2}+A^{-2})\llangle D \rrangle = 0.
\end{equation*}

\begin{figure}[H]
\centering
\includegraphics[scale=0.7]{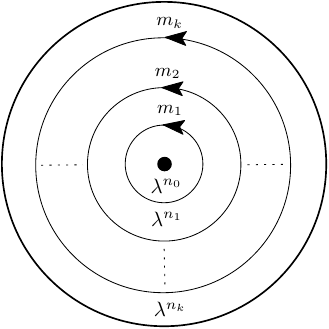}
\caption{Arrow diagrams $\lambda^{n_{0}}x_{m_{1}}\lambda^{n_{1}} \cdots \lambda^{n_{k-1}}x_{m_{k}}\lambda^{n_{k}}$}
\label{fig:Generators_Annulus}
\end{figure}

Let $\Gamma$ be the set of arrow diagrams\footnote{We will not make a distinction between these arrow diagrams with their equivalence classes modulo $\Omega_{1}-\Omega_{4}$-moves.} on ${\bf A}^{2}$ defined by
\begin{equation*}
\Gamma = \{\lambda^{n_{0}}x_{m_{1}}\lambda^{n_{1}} \cdots \lambda^{n_{k-1}}x_{m_{k}}\lambda^{n_{k}} \mid n_{i} \geq 0, \ m_{i} \in \mathbb{Z} \ \text{and} \ k \geq 0 \},
\end{equation*}
where $\lambda^{n_{0}}x_{m_{1}}\lambda^{n_{1}} \cdots \lambda^{n_{k-1}}x_{m_{k}}\lambda^{n_{k}}$ is shown in Figure~\ref{fig:Generators_Annulus}. For an arrow diagram 
\begin{equation*}
w = D_{0}x_{m_{1}}D_{1}\ldots D_{k-1}x_{m_{k}}D_{k}    
\end{equation*}
in ${\bf A}^{2}$ with no crossing, we apply $\langle \cdot \rangle_{r}$ to each $D_{i}$ and express $w$ as an $R$-linear combination of elements of $\Gamma$, i.e., we may define $\llangle w \rrangle_{\Gamma}$ by
\begin{equation*}
\llangle w \rrangle_{\Gamma} = \langle D_{0}\rangle_{r} \, x_{m_{1}}\langle D_{1}\rangle_{r} \ldots \langle D_{k-1}\rangle_{r} \, x_{m_{k}} \langle D_{k}\rangle_{r}.
\end{equation*}
Let $D = w_{1}Ww_{2}$ and $D'= w_{1}W'w_{2}$ be arrow diagrams in Figure~\ref{fig:Omega5FirstTwoTypes}, where $w_{1},w_{2}$ are arrow diagrams with no crossings and components $W,W' \subset {\bf D}^{2}$ differ by an $\Omega_{5}$-move. We denote by $W_{+},W_{-}, W'_{+}$, and $W'_{-}$ arrow diagrams obtained from $W$ and $W'$ after smoothing their corresponding crossing according to positive and negative markers.

\begin{figure}[H]
\centering
\includegraphics[scale=0.7]{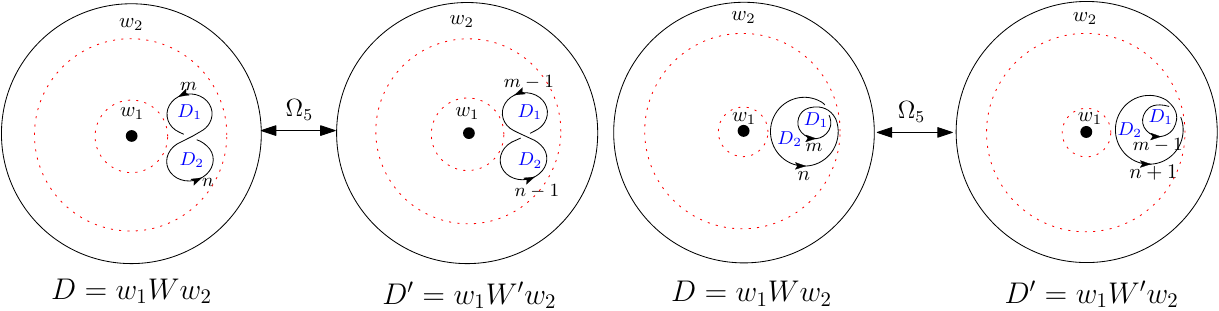}
\caption{Diagrams $D$ and $D'$ in ${\bf A}^{2}$ related by $\Omega_{5}$-moves}
\label{fig:Omega5FirstTwoTypes}
\end{figure}

Since $\langle W \rangle_{r} = \langle W'\rangle_{r}$, it follows that
\begin{eqnarray*}
&&\llangle w_{1}(A W_{+} + A^{-1} W_{-} - A W'_{+} - A^{-1} W'_{-}) w_{2} \rrangle_{\Gamma}\\
&=&\llangle w_{1} \rrangle_{\Gamma} \langle A W_{+} + A^{-1} W_{-} - A W'_{+} - A^{-1}  W'_{-} \rangle_{r} \llangle w_{2} \rrangle_{\Gamma}\\
&=&\llangle w_{1} \rrangle_{\Gamma} (A \langle W_{+} \rangle_{r} + A^{-1} \langle W_{-} \rangle_{r} - A \langle W'_{+} \rangle_{r} - A^{-1} \langle W'_{-}\rangle_{r} ) \llangle w_{2} \rrangle_{\Gamma} = \llangle w_{1}\rrangle_{\Gamma} (\langle W \rangle_{r} - \langle W' \rangle_{r})\llangle w_{2} \rrangle_{\Gamma} = 0.
\end{eqnarray*}
Consequently,
\begin{equation}
\label{eqn:InvarianceOnOmega5_1_2}
\llangle w_{1}(A W_{+} + A^{-1} W_{-} - A W'_{+} - A^{-1} W'_{-}) w_{2} \rrangle_{\Gamma} = 0.
\end{equation}

\medskip

\begin{lemma}
\label{lem:Gamma_c} 
Let $c \in \mathbb{Z}$ and $w\in \Gamma$, then in $S\mathcal{D}({\bf A}^{2})$,
\begin{enumerate}
\item for $w = w'x_{m}\lambda^{n}(x_{c})^{k}$ with $w'$ containing at least one $x$ curve and $n \geq 1$,
\begin{equation*}
w = Aw'x_{m-1}\lambda^{n-1}(x_{c})^{k} + A^{-1}w'x_{m+1}\lambda^{n-1}(x_{c})^{k};
\end{equation*}
\item for $w = w'x_{m}(x_{c})^{k}$ with $w'$ containing at least one $x$ curve and $m > c+1$, 
\begin{equation*}
w = A^{-1}w'\lambda x_{m-1}(x_{c})^{k} - A^{-2}w'x_{m-2}(x_{c})^{k};
\end{equation*}
\item for $w = w'x_{m}(x_{c})^{k}$ with $w'$ containing at least one $x$ curve and $m < c$, 
\begin{equation*}
w = Aw'\lambda x_{m+1}(x_{c})^{k} - A^{2}w'x_{m+2}(x_{c})^{k};
\end{equation*}
\item for $w = w'x_{m}\lambda^{n}x_{c+1}(x_{c})^{k}$, 
\begin{equation*}
w = -A^{-1}w'\lambda t_{c-m,n}(x_{c})^{k} + 2w't_{c-1-m,n}(x_{c})^{k} + A^{-2}w'x_{m+1}\lambda^{n}(x_{c})^{k+1};
\end{equation*}
\item for $w = \lambda^{n}x_{m}\lambda^{n'}(x_{c})^{k}$ with $n \geq 1$,
\begin{equation*}
w = A\lambda^{n-1}x_{m+1}\lambda^{n'}(x_{c})^{k} + A^{-1}\lambda^{n-1}x_{m-1}\lambda^{n'}(x_{c})^{k};
\end{equation*}
\item for $w = x_{m}\lambda^{n}(x_{c})^{k}$ with $m > c+1$,
\begin{equation*}
w = Ax_{m-1}\lambda^{n+1}(x_{c})^{k} - A^{2}x_{m-2}\lambda^{n}(x_{c})^{k};
\end{equation*}
\item for $w = x_{m}\lambda^{n}(x_{c})^{k}$ with $m < c$,
\begin{equation*}
w = A^{-1}x_{m+1}\lambda^{n+1}(x_{c})^{k} - A^{-2}x_{m+2}\lambda^{n}(x_{c})^{k}.
\end{equation*}
\end{enumerate}
\end{lemma}

\begin{figure}[ht]
\centering
\includegraphics[scale=0.7]{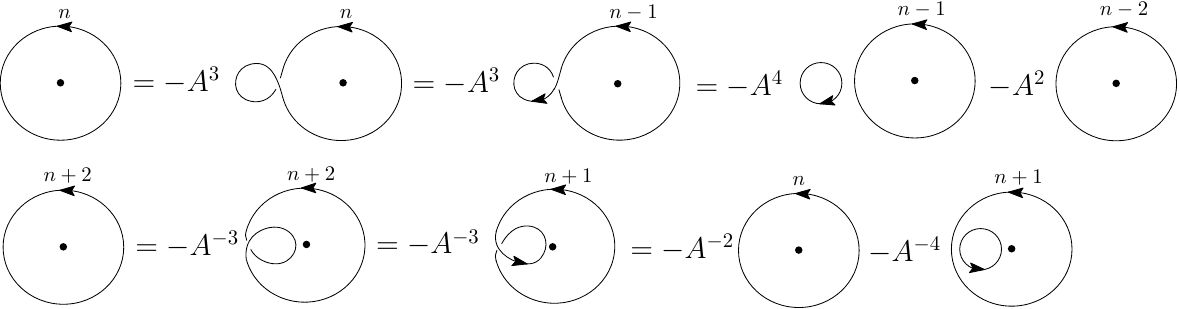}
\caption{Relation for $x_{n}$-curve in ${\bf A}^{2}$}
\label{fig:xn_rel}
\end{figure}

\begin{proof}
As it is shown in Figure~\ref{fig:xn_rel},
\begin{equation*}
x_{n} = -A^{4}x_{n-1}t_{-1} - A^{2}x_{n-2} = -A^{-2}t_{1}x_{n+1} - A^{2}x_{n+2}.
\end{equation*}
Since $t_{-1} = -A^{-3}\lambda$ and $t_{1} = -A^{3}\lambda$, the above equation becomes
\begin{equation*}
x_{n} = Ax_{n-1} \lambda - A^{2}x_{n-2} = A\lambda x_{n+1} - A^{2}x_{n+2}.
\end{equation*}
Identities 1)--3) and 5)--7) follow directly by applying the recursion relations for $x_{n}$ from the above. For relation 4), we notice that links on the left and right of Figure~\ref{fig:Omega2and5ForBracketC} differ by $\Omega_{2}$ and $\Omega_{5}$ moves. Hence,
\begin{equation*}
w'x_{m_{1}}\lambda^{n}x_{m_{2}}(x_{c})^{k} = A^{2}w't_{-1}t_{m_{2}-1-m_{1},n}(x_{c})^{k} + 2w't_{m_{2}-2-m_{1},n}(x_{c})^{k} + A^{-2}w'x_{m_{1}+1}\lambda^{n}x_{m_{2}-1}(x_{c})^{k}.  
\end{equation*}
Therefore, 4) follows by taking $m_{1} = m$ and $m_{2} = c+1$ in the above equation.
\end{proof}

\begin{figure}[ht]
\centering
\includegraphics[scale=0.8]{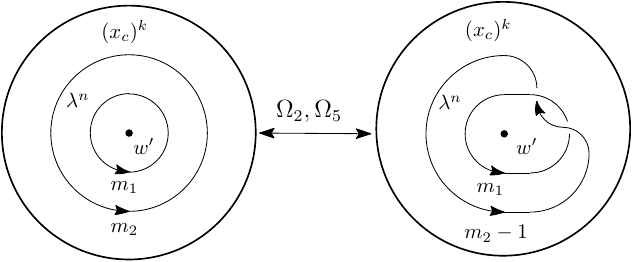}
\caption{$\Omega_{2}$ and $\Omega_{5}$-moves between $x_{m_{1}}$ and $x_{m_{2}}$-curves on ${\bf A}^{2}$}
\label{fig:Omega2and5ForBracketC}
\end{figure}

For a fixed $c \in \mathbb{Z}$, we consider a subset $\Sigma_{c}$ of $\Gamma$ defined by
\begin{equation*}
\Sigma_{c} = \{\lambda^{n}, x_{c}\lambda^{n}(x_{c})^{k}, x_{c+1}\lambda^{n}(x_{c})^{k} \mid \ n,k \geq 0 \}.
\end{equation*}
After applying identities from Lemma~\ref{lem:Gamma_c} and $\langle \cdot \rangle_{r}$ sufficiently many times, we are able to assign to each $w \in \Gamma$ an element $\llangle w \rrangle_{c} \in R\Sigma_{c}$. More precisely, for $w\in R\Gamma$, we define $\llangle w \rrangle_{c}\in R\Sigma_{c}$ as follows:
\begin{enumerate}[(a)]
\item if $w = \sum_{w'\in S} r_{w'}w'$, $S$ is a finite subset of $\Gamma$ with at least two elements, let 
\begin{equation*}
\llangle w \rrangle_{c} = \sum_{w'\in S} r_{w'}\llangle w' \rrangle_{c};
\end{equation*}
\item for $w \in \Sigma_{c}$,
\begin{equation*}
\llangle w \rrangle_{c} = w;
\end{equation*}
\item for $w = w'x_{m}\lambda^{n}(x_{c})^{k}$, where $w'$ contains at least one $x$ curve and $n \geq 1$,
\begin{equation*}
\llangle w \rrangle_{c} = A \llangle w'x_{m-1}\lambda^{n-1}(x_{c})^{k} \rrangle_{c} + A^{-1} \llangle w'x_{m+1}\lambda^{n-1}(x_{c})^{k} \rrangle_{c};
\end{equation*}
\item for $w = w'x_{m}(x_{c})^{k}$, where $w'$ contains at least one $x$ curve and $m > c+1$,
\begin{equation*}
\llangle w \rrangle_{c} = A^{-1} \llangle w'\lambda x_{m-1}(x_{c})^{k} \rrangle_{c} - A^{-2} \llangle w'x_{m-2}(x_{c})^{k} \rrangle_{c};
\end{equation*}
\item for $w = w'x_{m}(x_{c})^{k}$, where $w'$ contains at least one $x$ curve and $m < c$,
\begin{equation*}
\llangle w \rrangle_{c} = A \llangle w'\lambda x_{m+1}(x_{c})^{k} \rrangle_{c} - A^{2} \llangle w'x_{m+2}(x_{c})^{k} \rrangle_{c};
\end{equation*}
\item for $w = w'x_{m}\lambda^{n}x_{c+1}(x_{c})^{k}$,
\begin{equation*}
\llangle w \rrangle_{c} = -A^{-1} \llangle w'\lambda P_{c-m,n}(x_{c})^{k} \rrangle_{c} + 2 \llangle w'P_{c-1-m,n}(x_{c})^{k} \rrangle_{c} + A^{-2} \llangle w'x_{m+1}\lambda^{n}(x_{c})^{k+1} \rrangle_{c};
\end{equation*}
\item for $w = \lambda^{n}x_{m}\lambda^{n'}(x_{c})^{k}$ with $n \geq 1$,
\begin{equation*}
\llangle w \rrangle_{c} = A \llangle \lambda^{n-1}x_{m+1}\lambda^{n'}(x_{c})^{k} \rrangle_{c} + A^{-1} \llangle \lambda^{n-1}x_{m-1}\lambda^{n'}(x_{c})^{k} \rrangle_{c};
\end{equation*}
\item for $w = x_{m}\lambda^{n}(x_{c})^{k}$ with $m > c+1$,
\begin{equation*}
\llangle w \rrangle_{c} = A \llangle x_{m-1}\lambda^{n+1}(x_{c})^{k} \rrangle_{c} - A^{2} \llangle x_{m-2}\lambda^{n}(x_{c})^{k} \rrangle_{c};
\end{equation*}
\item for $w = x_{m}\lambda^{n}(x_{c})^{k}$ with $m < c$,
\begin{equation*}
\llangle w \rrangle_{c} = A^{-1} \llangle x_{m+1}\lambda^{n+1}(x_{c})^{k} \rrangle_{c} - A^{-2} \llangle x_{m+2}\lambda^{n}(x_{c})^{k} \rrangle_{c}.
\end{equation*}
\end{enumerate}

\begin{lemma}
\label{lem:loc_prop_bracket_c}    
For any $w = w_{1}x_{m}w_{2} \in \Gamma$,
\begin{equation}
\label{eqn:xm_rel_1}
\llangle w \rrangle_{c} = A\llangle w_{1}\lambda x_{m+1}w_{2} \rrangle_{c} - A^{2}\llangle w_{1}x_{m+2}w_{2} \rrangle_{c}
\end{equation}
and
\begin{equation}
\label{eqn:xm_rel_2}
\llangle w \rrangle_{c} = A\llangle w_{1}x_{m-1} \lambda w_{2} \rrangle_{c} - A^{2}\llangle w_{1}x_{m-2}w_{2} \rrangle_{c}.
\end{equation}
\end{lemma}

\begin{proof}
It suffices to show that \eqref{eqn:xm_rel_1} and \eqref{eqn:xm_rel_2} hold when
$w_{2} = \lambda^{n}(x_{c+1})^{\varepsilon}(x_{c})^{k}$, where $\varepsilon \in \{0,1\}$. Indeed, after using parts c), d), e), and f) in the definition of $\llangle \cdot \rrangle_{c}$ sufficiently many times, we can express $\llangle w \rrangle_{c}$, $\llangle w_{1}\lambda x_{m+1}w_{2} \rrangle_{c}$, and $\llangle w_{1}x_{m+2}w_{2} \rrangle_{c}$ (or $\llangle w \rrangle_{c}$, $\llangle w_{1}x_{m-1} \lambda w_{2} \rrangle_{c}$, and $\llangle w_{1}x_{m-2}w_{2} \rrangle_{c}$) as linear combinations of $\llangle w_{1}x_{m}w_{2}' \rrangle_{c}$, $\llangle w_{1}\lambda x_{m+1}w_{2}' \rrangle_{c}$, and $\llangle w_{1}x_{m+2}w_{2}' \rrangle_{c}$ (or $\llangle w_{1}x_{m}w_{2}' \rrangle_{c}$, $\llangle w_{1}x_{m-1} \lambda w_{2}' \rrangle_{c}$, and $\llangle w_{1}x_{m-2}w_{2}' \rrangle_{c}$) with $w_{2}' = \lambda^{n}(x_{c+1})^{\varepsilon}(x_{c})^{k}$. For the first equality, we consider the following cases: (i) $w_{1} = \lambda^{n_{1}}$ and (ii) $w_{1} = w_{1}'x_{m_{1}}\lambda^{n_{1}}$, where $w_{1}' \in \Gamma$.

For case i), if $\varepsilon = 0$, then $w = \lambda^{n_{1}}x_{m}\lambda^{n}(x_{c})^{k}$ and \eqref{eqn:xm_rel_1} follows immediately by part g) in the definition of $\llangle \cdot \rrangle_{c}$. For \eqref{eqn:xm_rel_2}, using inductively g), one shows that 
\begin{eqnarray*}
\llangle \lambda^{n_{1}}x_{m}\lambda^{n}(x_{c})^{k} \rrangle_{c} &=& \sum_{i=0}^{n_{1}} A^{n_{1}-2i} \binom{n_{1}}{i} \llangle x_{m+n_{1}-2i} \lambda^{n}(x_{c})^{k} \rrangle_{c}, \\
\llangle \lambda^{n_{1}}x_{m-1}\lambda^{n+1}(x_{c})^{k} \rrangle_{c} &=& \sum_{i=0}^{n_{1}} A^{n_{1}-2i} \binom{n_{1}}{i} \llangle x_{m-1+n_{1}-2i} \lambda^{n+1}(x_{c})^{k} \rrangle_{c}, \\
\llangle \lambda^{n_{1}}x_{m-2}\lambda^{n}(x_{c})^{k} \rrangle_{c} &=& \sum_{i=0}^{n_{1}} A^{n_{1}-2i} \binom{n_{1}}{i} \llangle x_{m-2+n_{1}-2i} \lambda^{n}(x_{c})^{k} \rrangle_{c}.
\end{eqnarray*}
Since by either part h) or i) in the definition of $\llangle \cdot \rrangle_{c}$,
\begin{equation*}
\llangle x_{m-1+n_{1}-2i} \lambda^{n+1}(x_{c})^{k} \rrangle_{c} = A\llangle x_{m-2+n_{1}-2i} \lambda^{n}(x_{c})^{k} \rrangle_{c} + A^{-1}\llangle x_{m+n_{1}-2i} \lambda^{n}(x_{c})^{k} \rrangle_{c}
\end{equation*}
for $i = 0, 1, \ldots, n_{1}$, we see that
\begin{equation*}
\llangle \lambda^{n_{1}}x_{m}\lambda^{n}(x_{c})^{k} \rrangle_{c} = A\llangle \lambda^{n_{1}}x_{m-1}\lambda^{n+1}(x_{c})^{k} \rrangle_{c} - A^{2}\llangle \lambda^{n_{1}}x_{m-2}\lambda^{n}(x_{c})^{k} \rrangle_{c}.
\end{equation*}
Therefore, \eqref{eqn:xm_rel_2} follows.

If $\varepsilon = 1$, then $w = \lambda^{n_{1}}x_{m}\lambda^{n}x_{c+1}(x_{c})^{k}$, and to prove \eqref{eqn:xm_rel_1} we see that by part f) of the definition of $\llangle \cdot \rrangle_{c}$,
\begin{eqnarray*}
\llangle \lambda^{n_{1}}x_{m}\lambda^{n}x_{c+1}(x_{c})^{k} \rrangle_{c} &=& -A^{-1} \llangle \lambda^{n_{1}+1} P_{c-m,n}(x_{c})^{k} \rrangle_{c} + 2 \llangle \lambda^{n_{1}}P_{c-m-1,n}(x_{c})^{k} \rrangle_{c} \\
&+&  A^{-2} \llangle \lambda^{n_{1}}x_{m+1}\lambda^{n}(x_{c})^{k+1} \rrangle_{c}, \\
\llangle \lambda^{n_{1}+1}x_{m+1}\lambda^{n}x_{c+1}(x_{c})^{k} \rrangle_{c} &=& -A^{-1} \llangle \lambda^{n_{1}+2} P_{c-m-1,n}(x_{c})^{k} \rrangle_{c} + 2 \llangle \lambda^{n_{1}+1}P_{c-m-2,n}(x_{c})^{k} \rrangle_{c} \\
&+& A^{-2} \llangle \lambda^{n_{1}+1}x_{m+2}\lambda^{n}(x_{c})^{k+1} \rrangle_{c}, \\
\llangle \lambda^{n_{1}}x_{m+2}\lambda^{n}x_{c+1}(x_{c})^{k} \rrangle_{c} &=& -A^{-1} \llangle \lambda^{n_{1}+1} P_{c-m-2,n}(x_{c})^{k} \rrangle_{c} + 2 \llangle \lambda^{n_{1}}P_{c-m-3,n}(x_{c})^{k} \rrangle_{c} \\
&+& A^{-2} \llangle \lambda^{n_{1}}x_{m+3}\lambda^{n}(x_{c})^{k+1} \rrangle_{c}.
\end{eqnarray*}
Therefore, using \eqref{eqn:Pnk_1} and the previous case, it follows that
\begin{equation*}
\llangle \lambda^{n_{1}}x_{m}\lambda^{n}x_{c+1}(x_{c})^{k} \rrangle_{c} = A\llangle \lambda^{n_{1}+1}x_{m+1}\lambda^{n}x_{c+1}(x_{c})^{k} \rrangle_{c} - A^{2}\llangle \lambda^{n_{1}}x_{m+2}\lambda^{n}x_{c+1}(x_{c})^{k} \rrangle_{c}.
\end{equation*}
Using analogous arguments and \eqref{eqn:Pnk} instead \eqref{eqn:Pnk_1}, one shows that \eqref{eqn:xm_rel_2} also holds.

For case ii), when $\varepsilon = 0$, then $w = w_{1}x_{m}\lambda^{n}(x_{c})^{k}$ and \eqref{eqn:xm_rel_2} follows from part c) in the definition of $\llangle \cdot \rrangle_{c}$. For \eqref{eqn:xm_rel_1}, we can use similar reasoning as for case i) and use part c) in the definition of $\llangle \cdot \rrangle_{c}$ recursively,  
\begin{eqnarray*}
\llangle w_{1}x_{m}\lambda^{n}(x_{c})^{k} \rrangle_{c} &=& \sum_{i=0}^{n} A^{n-2i} \binom{n}{i} \llangle w_{1}x_{m-n+2i}(x_{c})^{k} \rrangle_{c}, \\
\llangle w_{1}\lambda x_{m+1}\lambda^{n}(x_{c})^{k} \rrangle_{c} &=& \sum_{i=0}^{n} A^{n-2i} \binom{n}{i} \llangle w_{1}\lambda x_{m+1-n+2i}(x_{c})^{k} \rrangle_{c}, \\
\llangle w_{1}x_{m+2}\lambda^{n}(x_{c})^{k} \rrangle_{c} &=& \sum_{i=0}^{n} A^{n-2i} \binom{n}{i} \llangle w_{1}x_{m+2-n+2i} (x_{c})^{k} \rrangle_{c}.
\end{eqnarray*}
Since by either part d) or e) in the definition of $\llangle \cdot \rrangle_{c}$,
\begin{equation*}
\llangle w_{1}\lambda x_{m+1-n+2i}(x_{c})^{k} \rrangle_{c} = A\llangle w_{1}x_{m+2-n+2i}(x_{c})^{k} \rrangle_{c} + A^{-1}\llangle w_{1}x_{m-n+2i}(x_{c})^{k} \rrangle_{c}
\end{equation*}
for $i = 0, 1, \ldots, n$, we see that
\begin{equation*}
\llangle w_{1}x_{m}\lambda^{n}(x_{c})^{k} \rrangle_{c} = A\llangle w_{1}\lambda x_{m+1}\lambda^{n}(x_{c})^{k} \rrangle_{c} - A^{2}\llangle w_{1}x_{m+2}\lambda^{n}(x_{c})^{k} \rrangle_{c}.
\end{equation*}
Therefore, \eqref{eqn:xm_rel_1} follows.

If $\varepsilon = 1$, then $w = w_{1}x_{m}\lambda^{n}x_{c+1}(x_{c})^{k}$. Equations \eqref{eqn:xm_rel_1} and \eqref{eqn:xm_rel_2} follow by taking $w_{1}$ instead of $\lambda^{n_{1}}$ in our argument used for case i) with $\varepsilon = 1$.
\end{proof}

\begin{definition}
\label{def:PolynomialsQ_n}
A family of polynomials $\{Q_{n}\}_{n \in \mathbb{Z}}$ in $\mathbb{Z}[\lambda]$ is determined by recursion
\begin{equation*}
Q_{0} = 0, \quad Q_{1} = 1, \quad \text{and} \quad Q_{n+2} = \lambda Q_{n+1} - Q_{n}
\end{equation*}
for $n \geq 0$, and $Q_{n} = -Q_{-n}$ for $n < 0$.
\end{definition}
We note that for $Q_{n}$ with $n \neq 0$, its degree $\deg (Q_{n}) = |n|-1$ and its leading coefficient is $1$ if $n > 0$.

\begin{lemma}
\label{lem:rel_Pn}
For any $n \in \mathbb{Z}$,
\begin{equation}
\label{eqn:rel_Pn}
P_{n} = -A^{n+2} Q_{n+1} + A^{n-2} Q_{n-1}.
\end{equation}
\end{lemma}

\begin{proof}
We see that by the definition of $P_{n}$,
\begin{equation*}
P_{0} = -A^{2}-A^{-2} = -A^{2}Q_{1} + A^{-2}Q_{-1} \quad \text{and} \quad P_{1} = -A^{3}\lambda = -A^{3}Q_{2}+A^{-1}Q_{0}.   
\end{equation*}
Since $P_{n} = AP_{n-1}\lambda -A^{2}P_{n-2}$, \eqref{eqn:rel_Pn} follows by induction on $n$ for both $n > 1$ and $n < 0$ and Definition~\ref{def:PolynomialsQ_n}. 
\end{proof}

\begin{lemma}
\label{lem:rel_xm}
For any $k \in \mathbb{Z}$ and $w_{1}x_{m}w_{2} \in \Gamma$ with $m \in \mathbb{Z}$,
\begin{equation}
\label{eqn:rel_xm_1}
\llangle w_{1}x_{m}w_{2} \rrangle_{c} = -A^{m-k}\llangle w_{1}x_{k}Q_{m-k-1}w_{2} \rrangle_{c} + A^{m-k-1}\llangle w_{1}x_{k+1}Q_{m-k}w_{2} \rrangle_{c}
\end{equation}
and
\begin{equation}
\label{eqn:rel_xm_2}
\llangle w_{1}x_{m}w_{2} \rrangle_{c} = -A^{k-m}\llangle w_{1}Q_{m-k-1}x_{k}w_{2} \rrangle_{c} + A^{k-m+1}\llangle w_{1}Q_{m-k}x_{k+1}w_{2} \rrangle_{c}.
\end{equation}
\end{lemma}

\begin{proof}
Clearly, for any fixed $k$, both \eqref{eqn:rel_xm_1} and \eqref{eqn:rel_xm_2} hold when $m = k, k+1$. Using Lemma~\ref{lem:loc_prop_bracket_c}, one also shows that \eqref{eqn:rel_xm_1} and \eqref{eqn:rel_xm_2} hold for both $m > k+1$ and $m < k$ by induction on $m$ and Definition~\ref{def:PolynomialsQ_n}. 
\end{proof}

\begin{lemma}
\label{lem:annulus_for_any_kn}
Let $\Delta_{t}^{+},\Delta_{t}^{-},\Delta_{x}^{+},\Delta_{x}^{-}$ be finite subsets of $R \times \Gamma \times \Gamma \times \mathbb{Z}$, and define
\begin{equation*}
\Theta_{t}^{+}(k,n) = \sum_{(r,w_{1},w_{2},v) \in \Delta_{t}^{+}} r\llangle w_{1}P_{n+v,k}w_{2} \rrangle_{c}, \quad
\Theta_{t}^{-}(k,n) = \sum_{(r,w_{1},w_{2},v) \in \Delta_{t}^{-}} r\llangle w_{1}P_{-n+v}\lambda^{k}w_{2} \rrangle_{c},
\end{equation*}
\begin{equation*}
\Theta_{x}^{+}(k,n) = \sum_{(r,w_{1},w_{2},v) \in \Delta_{x}^{+}} r\llangle w_{1}\lambda^{k}x_{n+v}w_{2} \rrangle_{c}, \quad
\Theta_{x}^{-}(k,n) = \sum_{(r,w_{1},w_{2},v) \in \Delta_{x}^{-}} r\llangle w_{1}x_{-n+v}\lambda^{k}w_{2} \rrangle_{c},
\end{equation*}
and 
\begin{equation*}
\Theta_{t,x}(k,n) = \Theta_{t}^{+}(k,n) + \Theta_{t}^{-}(k,n) + \Theta_{x}^{+}(k,n)+ \Theta_{x}^{-}(k,n).
\end{equation*}
If either \textup{(1)} $\Theta_{t,x}(0,n) = 0$ for all $n \in \mathbb{Z}$ or \textup{(2)} $\Theta_{t,x}(k,n_{0}) = \Theta_{t,x}(k,n_{0}+1) = 0$ for all $k \geq 0$ and a fixed $n_{0} \in \mathbb{Z}$, then $\Theta_{t,x}(k,n) = 0$ for any $k \geq 0$ and $n \in \mathbb{Z}$.
\end{lemma}

\begin{proof}
We prove 1) by using induction on $k$. Clearly, for all $n$, $\Theta_{t,x}(0,n) = 0$ by assumption. Fix $k-1 \geq 0$ and assume that $\Theta_{t,x}(k-1,n) = 0$ for any $n$. Since by \eqref{eqn:Pnk}, \eqref{eqn:Pnk_1}, \eqref{eqn:xm_rel_1}, and \eqref{eqn:xm_rel_2} respectively,
\begin{eqnarray*}
\Theta_{t}^{+}(k,n) &=& A\Theta_{t}^{+}(k-1,n+1) + A^{-1}\Theta_{t}^{+}(k-1,n-1), \\
\Theta_{t}^{-}(k,n) &=& A\Theta_{t}^{-}(k-1,n+1) + A^{-1}\Theta_{t}^{-}(k-1,n-1), \\
\Theta_{x}^{+}(k,n) &=& A\Theta_{x}^{+}(k-1,n+1) + A^{-1}\Theta_{x}^{+}(k-1,n-1), \\
\Theta_{x}^{-}(k,n) &=& A\Theta_{x}^{-}(k-1,n+1) + A^{-1}\Theta_{x}^{-}(k-1,n-1), 
\end{eqnarray*}
it follows that
\begin{equation*}
\Theta_{t,x}(k,n) = A\Theta_{t,x}(k-1,n+1) + A^{-1}\Theta_{t,x}(k-1,n-1) = 0.
\end{equation*}
We prove 2) by induction on $n$ for both $n > n_{0}+1$ and $n < n_{0}$. Clearly, using the induction assumption $\Theta_{t,x}(k,n_{0}) = \Theta_{t,x}(k,n_{0}+1) = 0$ for all $k$. Fix $n > n_{0}+1$ and assume that $\Theta_{t,x}(k,n-1) = \Theta_{t,x}(k,n-2) = 0$ for any $k$. Since by \eqref{eqn:Pnk}, \eqref{eqn:Pnk_1}, \eqref{eqn:xm_rel_1}, and \eqref{eqn:xm_rel_2} respectively,
\begin{eqnarray*}
\Theta_{t}^{+}(k,n) &=& A^{-1}\Theta_{t}^{+}(k+1,n-1) - A^{-2}\Theta_{t}^{+}(k,n-2), \\
\Theta_{t}^{-}(k,n) &=& A^{-1}\Theta_{t}^{-}(k+1,n-1) - A^{-2}\Theta_{t}^{-}(k,n-2), \\
\Theta_{x}^{+}(k,n) &=& A^{-1}\Theta_{x}^{+}(k+1,n-1) - A^{-2}\Theta_{x}^{+}(k,n-2), \\
\Theta_{x}^{-}(k,n) &=& A^{-1}\Theta_{x}^{-}(k+1,n-1) - A^{-2}\Theta_{x}^{-}(k,n-2), 
\end{eqnarray*}
it follows that
\begin{equation*}
\Theta_{t,x}(k,n) = A^{-1}\Theta_{t,x}(k+1,n-1) - A^{-2}\Theta_{t,x}(k,n-2) = 0.
\end{equation*}
Similar arguments can be used when $n < n_{0}$.
\end{proof}

Given an arrow diagram $D$ in ${\bf A}^{2}$, using $\llangle \cdot \rrangle$ we express $D$ as a $R$-linear combination of arrow diagrams $w$ without crossings. Therefore, using $\llangle \cdot \rrangle_{\Gamma}$ allows us to express each $w$ (hence entire linear combination) as a $R$-linear combination of elements $u\in \Gamma$. Finally, using $\llangle \cdot \rrangle_{c}$ allows us to express each $u$ (hence entire linear combination of $u$'s) as a $R$-linear combination of elements of $\Sigma_{c}$. Therefore, for each arrow diagram $D$ in ${\bf A}^{2}$ we are able to assign a $R$-linear combination of elements of $\Sigma_{c}$, i.e., for each $D$, we define
\begin{equation*}
\psi_{c}(D) =  \llangle \llangle \llangle D \rrangle \rrangle_{\Gamma} \rrangle_{c}\in R\Sigma_{c}.
\end{equation*}
It is clear that, for arrow diagrams $D$ and $D'$ related by a finite number of $\Omega_{1}-\Omega_{4}$-moves
\begin{equation*}
\psi_{c}(D) = \psi_{c}(D')
\end{equation*}
and, for any skein triple $D_{+}$, $D_{0}$, $D_{\infty}$, and disjoint unions $D\sqcup T_{1}$ (see Figure~\ref{fig:SkeinTripleOfDiagrams}),
\begin{equation*}
\psi_{c}(D_{+} - A D_{0} - A^{-1} D_{\infty}) = 0\,\,\text{and}\,\,\psi_{c}( D\sqcup T_{1}  + (A^{2}+A^{-2}) D) = 0,
\end{equation*}
where $T_{1}$ is a trivial circle on ${\bf A}^{2}$. We will show that $\psi_{c}$ defined above gives a surjective homomorphism of free $R$-modules $\psi_{c}: R\mathcal{D}({\bf A}^{2}) \to R\Sigma_{c}$. It suffices to show that for arrow diagrams $D$ and $D'$ related by a $\Omega_{5}$-move, $\psi_{c}(D-D') = 0$.

As a first step, we consider special cases of arrow diagrams in ${\bf A}^{2}$ with one crossing (see Figure~\ref{fig:annulus_type3}, Figure~\ref{fig:annulus_type4}, and Figure~\ref{fig:annulus_type5}) that are related by $\Omega_{5}$-move. Assume that $D = w_{1}Ww_{2}$, $D'= w_{1}W'w_{2}$ with $w_{1},w_{2}\in \Gamma$ and $W$ and $W'$ related by an $\Omega_{5}$-move. For such diagrams
\begin{equation}
\label{eqn:hDD'}
\psi_{c}(D-D') = \psi_{c}(w_{1}(W - W')w_{2}) = \llangle w_{1}\llangle A\langle W_{+}\rangle + A^{-1}\langle W_{-}\rangle - A\langle W'_{+}\rangle-A^{-1}\langle W'_{-}\rangle \rrangle_{\Gamma} w_{2}\rrangle_{c},  
\end{equation}
where $W_{+},W_{-},W'_{+}$, and $W'_{-}$ are arrow diagrams obtained from $W$ and $W'$ by smoothing the crossing of each according to positive and negative markers. 

\begin{figure}[H]
\centering
\includegraphics[scale=0.75]{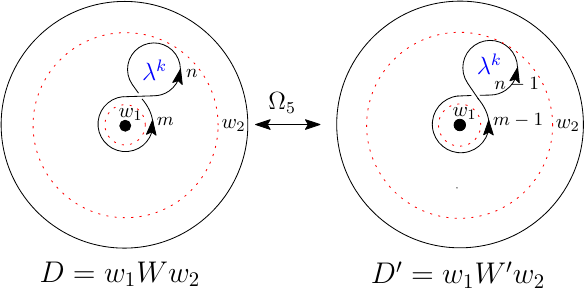}
\caption{Arrow diagrams $D$ and $D'$ in ${\bf A}^{2}$ related by an $\Omega_{5}$-move}
\label{fig:annulus_type3}
\end{figure}

For arrow diagrams $D$, $D'$ in Figure~\ref{fig:annulus_type3}, we see that $W_{+} = \lambda^{k}x_{m+n}$, $W_{-} = x_{m}t_{n,k}$, $W'_{+} = x_{m-1}t_{n-1,k}$, and $W'_{-} = \lambda^{k}x_{m+n-2}$ are obtained by smoothing the crossing of $W$ and $W'$ according to positive and negative markers. Hence, by \eqref{eqn:hDD'}, $\psi_{c}(D-D') = 0$ is equivalent to \eqref{eqn:H3_InvariantOmega5_3}.

\begin{lemma}
\label{lem:annulus_type3}
Let $w_{1},w_{2}\in \Gamma$, then for all $m,n\in \mathbb{Z}$ and $k\geq 0$,
\begin{equation}
\label{eqn:H3_InvariantOmega5_3}
\llangle w_{1}(A\lambda^{k}x_{m+n} +A^{-1}x_{m}P_{n,k} - A x_{m-1}P_{n-1,k} -A^{-1}\lambda^{k}x_{m+n-2})w_{2}\rrangle_{c} = 0.
\end{equation}
\end{lemma}

\begin{proof}
To prove \eqref{eqn:H3_InvariantOmega5_3}, it suffices to show that for any fixed $m$,
\begin{equation}
\label{eqn:case_type_iii}
A\llangle w_{1}\lambda^{k}x_{m+n}w_{2} \rrangle_{c} + A^{-1}\llangle w_{1}x_{m}P_{n,k}w_{2} \rrangle_{c} = A\llangle w_{1}x_{m-1}P_{n-1,k}w_{2} \rrangle_{c} + A^{-1}\llangle w_{1}\lambda^{k}x_{m+n-2}w_{2} \rrangle_{c}.
\end{equation}
When $k = 0$, by \eqref{eqn:rel_xm_1} and \eqref{eqn:rel_Pn}, the left hand side of \eqref{eqn:case_type_iii} becomes
\begin{eqnarray*}
A\llangle w_{1}x_{m+n}w_{2} \rrangle_{c} + A^{-1}\llangle w_{1}x_{m}P_{n}w_{2} \rrangle_{c}  
&=& A(-A^{n+1}\llangle w_{1}x_{m-1}Q_{n}w_{2} \rrangle_{c} + A^{n}\llangle w_{1}x_{m}Q_{n+1}w_{2} \rrangle_{c}) \\ 
&+& A^{-1}\llangle w_{1}x_{m}(-A^{n+2} Q_{n+1} + A^{n-2} Q_{n-1})w_{2} \rrangle_{c} \\
&=& -A^{n+2}\llangle w_{1}x_{m-1}Q_{n}w_{2}\rrangle_{c} + A^{n-3}\llangle w_{1}x_{m}Q_{n-1}w_{2} \rrangle_{c} 
\end{eqnarray*}
and the right hand side of \eqref{eqn:case_type_iii} becomes
\begin{eqnarray*}
A\llangle w_{1}x_{m-1}P_{n-1}w_{2} \rrangle_{c} + A^{-1}\llangle w_{1}x_{m+n-2}w_{2} \rrangle_{c}
&=& A\llangle w_{1}x_{m-1}(-A^{n+1} Q_{n} + A^{n-3} Q_{n-2})w_{2} \rrangle_{c} \\
&+& A^{-1}(-A^{n-1}\llangle w_{1}x_{m-1}Q_{n-2}w_{2} \rrangle_{c} + A^{n-2}\llangle w_{1}x_{m}Q_{n-1}w_{2} \rrangle_{c}) \\
&=& -A^{n+2}\llangle w_{1}x_{m-1}Q_{n}w_{2}\rrangle_{c} + A^{n-3}\llangle w_{1}x_{m}Q_{n-1}w_{2} \rrangle_{c}.
\end{eqnarray*}
Hence \eqref{eqn:case_type_iii} holds for $k = 0$ and, by Lemma~\ref{lem:annulus_for_any_kn}, it also holds for any $k > 0$.
\end{proof}

\begin{figure}[H]
\centering
\includegraphics[scale=0.75]{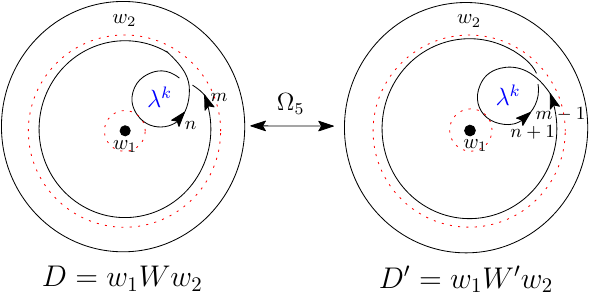}
\caption{Arrow diagrams $D$ and $D'$ in ${\bf A}^{2}$ related by an $\Omega_{5}$-move}
\label{fig:annulus_type4}
\end{figure}

For arrow diagrams $D$, $D'$ in Figure~\ref{fig:annulus_type4}, we see that $W_{+} = t_{n,k}x_{m}$, $W_{-} = x_{m-n}\lambda^{k}$, $W'_{+} = x_{m-n-2}\lambda^{k}$, and $W'_{-} = t_{n+1,k}x_{m-1}$ are obtained by smoothing the crossing of $W$ and $W'$ according to positive and negative markers. Hence, by \eqref{eqn:hDD'}, $\psi_{c}(D-D') = 0$ is equivalent to \eqref{eqn:H3_InvariantOmega5_4}.

\begin{lemma}
\label{lem:annulus_type4}
Let $w_{1},w_{2}\in \Gamma$, then for all $m,n\in \mathbb{Z}$ and $k\geq 0$,
\begin{equation}
\label{eqn:H3_InvariantOmega5_4}
\llangle w_{1}(A P_{n,k}x_{m} +A^{-1} x_{m-n}\lambda^{k} -A x_{m-n-2}\lambda^{k}-A^{-1}P_{n+1,k}x_{m-1} )w_{2}\rrangle_{c} = 0.
\end{equation}
\end{lemma}

\begin{proof}
To prove \eqref{eqn:H3_InvariantOmega5_4}, it suffices to show that for any fixed $m$,
\begin{equation}
\label{eqn:case_type_iv}
A\llangle w_{1}P_{n,k}x_{m}w_{2} \rrangle_{c} + A^{-1}\llangle w_{1}x_{m-n}\lambda^{k}w_{2} \rrangle_{c} = A\llangle w_{1}x_{m-n-2}\lambda^{k}w_{2} \rrangle_{c} + A^{-1}\llangle w_{1}P_{n+1,k}x_{m-1}w_{2} \rrangle_{c}.
\end{equation}    
By Lemma~\ref{lem:annulus_for_any_kn}, it is sufficient to check \eqref{eqn:case_type_iv} for $k = 0$. Indeed, by \eqref{eqn:rel_Pn} and \eqref{eqn:rel_xm_2}, the left hand side of \eqref{eqn:case_type_iv} becomes
\begin{eqnarray*}
A\llangle w_{1}P_{n}x_{m}w_{2} \rrangle_{c} + A^{-1}\llangle w_{1}x_{m-n}w_{2} \rrangle_{c}
&=& A\llangle w_{1}(-A^{n+2} Q_{n+1} + A^{n-2} Q_{n-1})x_{m}w_{2} \rrangle_{c} \\
&+& A^{-1}(-A^{n-1}\llangle w_{1}Q_{-n}x_{m-1}w_{2} \rrangle_{c} + A^{n}\llangle w_{1}Q_{1-n}x_{m}w_{2} \rrangle_{c}) \\
&=& -A^{n+3}\llangle w_{1}Q_{n+1}x_{m}w_{2} \rrangle_{c}+ A^{n-2}\llangle w_{1}Q_{n}x_{m-1}w_{2} \rrangle_{c}
\end{eqnarray*}
and the right hand side of \eqref{eqn:case_type_iv} becomes
\begin{eqnarray*}
A\llangle w_{1}x_{m-n-2}w_{2} \rrangle_{c} + A^{-1}\llangle w_{1}P_{n+1}x_{m-1}w_{2} \rrangle_{c}
&=& A(-A^{n+1}\llangle w_{1}Q_{-n-2}x_{m-1}w_{2} \rrangle_{c} + A^{n+2}\llangle w_{1}Q_{-n-1}x_{m}w_{2} \rrangle_{c}) \\
&+& A^{-1}\llangle w_{1}(-A^{n+3} Q_{n+2} + A^{n-1} Q_{n})x_{m-1}w_{2} \rrangle_{c} \\
&=& -A^{n+3}\llangle w_{1}Q_{n+1}x_{m}w_{2} \rrangle_{c} + A^{n-2}\llangle w_{1}Q_{n}x_{m-1}w_{2} \rrangle_{c}.
\end{eqnarray*}
Hence, \eqref{eqn:case_type_iv} holds when $k = 0$, which completes our proof.
\end{proof}

\begin{figure}[H]
\centering
\includegraphics[scale=0.75]{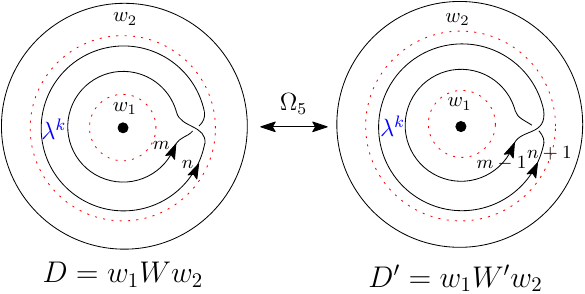}
\caption{Arrow diagrams $D$ and $D'$ in ${\bf A}^{2}$ related by an $\Omega_{5}$-move}
\label{fig:annulus_type5}
\end{figure}

For arrow diagrams $D$ and $D'$ in Figure~\ref{fig:annulus_type5}, we see that $W_{+} = t_{n-m,k}$, $W_{-} = x_{m}\lambda^{k}x_{n}$, $W'_{+} = x_{m-1}\lambda^{k}x_{n+1}$, and $W'_{-} = t_{n-m+2,k}$ are obtained by smoothing the crossing of $W$ and $W'$ according to positive and negative markers. Hence, by \eqref{eqn:hDD'}, $\psi_{c}(D-D') = 0$ is equivalent to \eqref{eqn:H3_InvariantOmega5_5}.

\begin{lemma}
\label{lem:annulus_type5}
Let $w_{1},w_{2}\in \Gamma$, then for all $m,n\in \mathbb{Z}$ and $k\geq 0$,
\begin{equation}
\label{eqn:H3_InvariantOmega5_5}
\llangle w_{1}(AP_{n-m,k} +A^{-1}x_{m}\lambda^{k}x_{n} - A x_{m-1}\lambda^{k}x_{n+1} -A^{-1}P_{n-m+2,k} )w_{2} \rrangle_{c} = 0.
\end{equation}
\end{lemma}

\begin{proof}
We first prove the following identity: For all $w_{1}\in \Gamma$ and all $m, k \in \mathbb{Z}$ with $k\geq 0$,
\begin{equation}
\label{eqn:case_type_v_simplify}
\llangle w_{1}P_{c-m}x_{c+1}(x_{c})^{k} \rrangle_{c} = A^{-2}\llangle w_{1}P_{c-m+1}(x_{c})^{k+1} \rrangle_{c} - A^{-2}\llangle w_{1}x_{m+1}(x_{c})^{k} \rrangle_{c} + \llangle w_{1}x_{m-1}(x_{c})^{k} \rrangle_{c}.
\end{equation}
Indeed, by \eqref{eqn:rel_xm_2} and \eqref{eqn:rel_Pn}, one can see that
\begin{eqnarray*}
& &\llangle w_{1}P_{c-m}x_{c+1}(x_{c})^{k} \rrangle_{c} - A^{-2}\llangle w_{1}P_{c-m+1}(x_{c})^{k+1} \rrangle_{c}\\
&=& \llangle w_{1}(-A^{c-m+2}Q_{c-m+1} + A^{c-m-2}Q_{c-m-1})x_{c+1}(x_{c})^{k} \rrangle_{c}\\
&-& A^{-2}\llangle w_{1}(- A^{c-m+3}Q_{c-m+2} + A^{c-m-1}Q_{c-m})(x_{c})^{k+1} \rrangle_{c} \\
&=& A^{c-m-3}\llangle w_{1}Q_{m-c}(x_{c})^{k+1} \rrangle_{c} - A^{c-m-2}\llangle w_{1}Q_{m-c+1}x_{c+1}(x_{c})^{k} \rrangle_{c} \\ 
&-& A^{c-m+1}\llangle w_{1}Q_{m-c-2}(x_{c})^{k+1} \rrangle_{c} + A^{c-m+2}\llangle w_{1}Q_{m-c-1}x_{c+1}(x_{c})^{k} \rrangle_{c} \\
&=& - A^{-2}\llangle w_{1}x_{m+1}(x_{c})^{k} \rrangle_{c} + \llangle w_{1}x_{m-1}(x_{c})^{k} \rrangle_{c}.
\end{eqnarray*}

To prove \eqref{eqn:H3_InvariantOmega5_5}, it suffices to show that for any fixed $m$,
\begin{equation}
\label{eqn:case_type_v}
A\llangle w_{1}P_{n-m,k}w_{2} \rrangle_{c} + A^{-1}\llangle w_{1}x_{m}\lambda^{k}x_{n}w_{2} \rrangle_{c} = A\llangle w_{1}x_{m-1}\lambda^{k}x_{n+1}w_{2} \rrangle_{c} + A^{-1}\llangle w_{1}P_{n-m+2,k}w_{2} \rrangle_{c}.
\end{equation}   
By Lemma~\ref{lem:annulus_for_any_kn}, it is sufficient to check \eqref{eqn:case_type_v} for $k = 0$. Indeed, we may assume that $w_{2} = \lambda^{n'}(x_{c+1})^{\varepsilon}(x_{c})^{k'} \in \Gamma$ with $\varepsilon \in \{0,1\}$ by using similar arguments as at the beginning of our proof of Lemma~\ref{lem:loc_prop_bracket_c}. We prove that \eqref{eqn:case_type_v} holds for $n' = 0$ and $n = c,c-1$. 

If $\varepsilon = 0$ and $n = c$, then after applying part f) in the definition of $\llangle\cdot\rrangle_{c}$ and \eqref{eqn:Pn}, the right hand side of \eqref{eqn:case_type_v} becomes
\begin{eqnarray*}
A\llangle w_{1}x_{m-1}x_{c+1}(x_{c})^{k'} \rrangle_{c} + A^{-1}\llangle w_{1}P_{c-m+2}(x_{c})^{k'} \rrangle_{c}
&=& A(-A^{-1}\llangle w_{1}\lambda P_{c-m+1}(x_{c})^{k'} \rrangle_{c} + 2\llangle w_{1}P_{c-m}(x_{c})^{k'} \rrangle_{c} \\
&+& A^{-2}\llangle w_{1}x_{m}(x_{c})^{k'+1} \rrangle_{c}) + A^{-1}\llangle w_{1}P_{c-m+2}(x_{c})^{k'} \rrangle_{c} \\
&=& A\llangle w_{1}P_{c-m}(x_{c})^{k'} \rrangle_{c} + A^{-1}\llangle w_{1}x_{m}(x_{c})^{k'+1} \rrangle_{c},
\end{eqnarray*}
which is precisely the left hand side of \eqref{eqn:case_type_v}.

If $\varepsilon = 0$ and $n = c-1$, then after applying \eqref{eqn:xm_rel_1}, \eqref{eqn:xm_rel_2}, part f) in the definition of $\llangle\cdot\rrangle_{c}$, and \eqref{eqn:Pn}, the left hand side of \eqref{eqn:case_type_v} becomes
\begin{eqnarray*}
& & A\llangle w_{1}P_{c-1-m}(x_{c})^{k'} \rrangle_{c} + A^{-1}\llangle w_{1}x_{m}x_{c-1}(x_{c})^{k'} \rrangle_{c} \\
&=& A\llangle w_{1}P_{c-1-m}(x_{c})^{k'} \rrangle_{c} + \llangle w_{1}x_{m}\lambda(x_{c})^{k'+1} \rrangle_{c} - A\llangle w_{1}x_{m}x_{c+1}(x_{c})^{k'} \rrangle_{c} \\
&=& A\llangle w_{1}P_{c-1-m}(x_{c})^{k'} \rrangle_{c} + A\llangle w_{1}x_{m-1}(x_{c})^{k'+1} \rrangle_{c} + A^{-1}\llangle w_{1}x_{m+1}(x_{c})^{k'+1} \rrangle_{c} \\
&+& \llangle w_{1}\lambda P_{c-m}(x_{c})^{k'} \rrangle_{c} - 2A\llangle w_{1}P_{c-1-m}(x_{c})^{k'} \rrangle_{c} - A^{-1}\llangle w_{1}x_{m+1}(x_{c})^{k'+1} \rrangle_{c} \\
&=& A\llangle w_{1}x_{m-1}(x_{c})^{k'+1} \rrangle_{c} + A^{-1}\llangle w_{1}P_{c-m+1}(x_{c})^{k'} \rrangle_{c},
\end{eqnarray*}  
which is precisely the right hand side of \eqref{eqn:case_type_v}.

If $\varepsilon = 1$ and $n = c$, then the left hand side of \eqref{eqn:case_type_v} becomes
\begin{equation}
\label{eqn:case_type_v_LHS}
A\llangle w_{1}P_{c-m}x_{c+1}(x_{c})^{k'} \rrangle_{c} + A^{-1}\llangle w_{1}x_{m}x_{c}x_{c+1}(x_{c})^{k'} \rrangle_{c}   
\end{equation}
and the right hand side of \eqref{eqn:case_type_v} becomes
\begin{equation}
\label{eqn:case_type_v_RHS}
A\llangle w_{1}x_{m-1}x_{c+1}x_{c+1}(x_{c})^{k'} \rrangle_{c} + A^{-1}\llangle w_{1}P_{c-m+2}x_{c+1}(x_{c})^{k'} \rrangle_{c}.
\end{equation}
The second term of \eqref{eqn:case_type_v_LHS} after applying part f) in the definition of $\llangle\cdot\rrangle_{c}$ turns into
\begin{eqnarray*}
& & A^{-1}\llangle w_{1}x_{m}x_{c}x_{c+1}(x_{c})^{k'} \rrangle_{c} \\
&=& -A^{-2}\llangle w_{1}x_{m}\lambda P_{0}(x_{c})^{k'} \rrangle_{c} + 2A^{-1}\llangle w_{1}x_{m}P_{-1}(x_{c})^{k'} \rrangle_{c} 
+ A^{-3}\llangle w_{1}x_{m}x_{c+1}(x_{c})^{k'+1} \rrangle_{c} \\
&=& \llangle w_{1}x_{m}\lambda (x_{c})^{k'} \rrangle_{c} - A^{-4}\llangle w_{1}x_{m}\lambda(x_{c})^{k'} \rrangle_{c} + A^{-3}\llangle w_{1}x_{m}x_{c+1}(x_{c})^{k'+1} \rrangle_{c}. 
\end{eqnarray*} 
Using \eqref{eqn:xm_rel_2}, the first term in the above sum becomes 
\begin{equation*}
\llangle w_{1}x_{m}\lambda (x_{c})^{k'} \rrangle_{c} = A\llangle w_{1}x_{m-1}(x_{c})^{k'} \rrangle_{c} + A^{-1}\llangle w_{1}x_{m+1}(x_{c})^{k'} \rrangle_{c} 
\end{equation*}
and, by part f) in the definition of $\llangle\cdot\rrangle_{c}$ and \eqref{eqn:Pn}, the third term becomes
\begin{eqnarray*}
& & A^{-3}\llangle w_{1}x_{m}x_{c+1}(x_{c})^{k'+1} \rrangle_{c}\\
&=& - A^{-4}\llangle w_{1}\lambda P_{c-m}(x_{c})^{k'+1} \rrangle_{c} + 2A^{-3}\llangle w_{1}P_{c-m-1}(x_{c})^{k'+1} \rrangle_{c} + A^{-5}\llangle w_{1}x_{m+1}(x_{c})^{k'+2} \rrangle_{c} \\
&=& - A^{-5}\llangle w_{1}P_{c-m+1}(x_{c})^{k'+1} \rrangle_{c} + A^{-3}\llangle w_{1}P_{c-m-1}(x_{c})^{k'+1} \rrangle_{c} + A^{-5}\llangle w_{1}x_{m+1}(x_{c})^{k'+2} \rrangle_{c}.
\end{eqnarray*} 
Moreover, by \eqref{eqn:case_type_v_simplify}, the first term of \eqref{eqn:case_type_v_LHS} is
\begin{equation*}
A\llangle w_{1}P_{c-m}x_{c+1}(x_{c})^{k'} \rrangle_{c} = A^{-1}\llangle w_{1}P_{c-m+1}(x_{c})^{k'+1} \rrangle_{c} - A^{-1}\llangle w_{1}x_{m+1}(x_{c})^{k'} \rrangle_{c} + A\llangle w_{1}x_{m-1}(x_{c})^{k'} \rrangle_{c}.
\end{equation*}
Therefore, after combining all results above, we see that  \eqref{eqn:case_type_v_LHS} becomes
\begin{eqnarray*}
& & A\llangle w_{1}P_{c-m}x_{c+1}(x_{c})^{k'} \rrangle_{c} + A^{-1}\llangle w_{1}x_{m}x_{c}x_{c+1}(x_{c})^{k'} \rrangle_{c} \\
&=& A^{-1}\llangle w_{1}P_{c-m+1}(x_{c})^{k'+1} \rrangle_{c} + 2A\llangle w_{1}x_{m-1}(x_{c})^{k'} \rrangle_{c}
- A^{-4}\llangle w_{1}x_{m}\lambda(x_{c})^{k'} \rrangle_{c} \\  
&-& A^{-5}\llangle w_{1}P_{c-m+1}(x_{c})^{k'+1} \rrangle_{c}
+ A^{-3}\llangle w_{1}P_{c-m-1}(x_{c})^{k'+1} \rrangle_{c} + A^{-5}\llangle w_{1}x_{m+1}(x_{c})^{k'+2} \rrangle_{c}.
\end{eqnarray*} 
Similarly, in \eqref{eqn:case_type_v_RHS} using part f) in the definition of $\llangle\cdot\rrangle_{c}$, the first term becomes
\begin{eqnarray*}
& & A\llangle w_{1}x_{m-1}x_{c+1}x_{c+1}(x_{c})^{k'} \rrangle_{c}  \\
&=& -\llangle w_{1}x_{m-1}\lambda P_{-1}(x_{c})^{k'} \rrangle_{c} + 2A\llangle w_{1}x_{m-1}P_{-2}(x_{c})^{k'} \rrangle_{c} 
+ A^{-1}\llangle w_{1}x_{m-1}x_{c+2}(x_{c})^{k'+1} \rrangle_{c} \\
&=& -A^{-3}\llangle w_{1}x_{m-1}\lambda^{2}(x_{c})^{k'} \rrangle_{c} + 2A\llangle w_{1}x_{m-1}(1 + A^{-4})(x_{c})^{k'} \rrangle_{c} + A^{-1}\llangle w_{1}x_{m-1}x_{c+2}(x_{c})^{k'+1} \rrangle_{c}.
\end{eqnarray*}  
Moreover, in the above sum, by \eqref{eqn:xm_rel_2}, the first term becomes
\begin{eqnarray*}
-A^{-3}\llangle w_{1}x_{m-1}\lambda^{2}(x_{c})^{k'} \rrangle_{c}
&=& -A^{-2}\llangle w_{1}x_{m-2}\lambda(x_{c})^{k'} \rrangle_{c} - A^{-4}\llangle w_{1}x_{m}\lambda(x_{c})^{k'} \rrangle_{c} \\
&=& -A^{-1}\llangle w_{1}x_{m-3}(x_{c})^{k'} \rrangle_{c} - A^{-3}\llangle w_{1}x_{m-1}(x_{c})^{k'} \rrangle_{c} -A^{-4}\llangle w_{1}x_{m}\lambda(x_{c})^{k'} \rrangle_{c},
\end{eqnarray*}
and by \eqref{eqn:xm_rel_1} its last term is
\begin{equation*}
A^{-1}\llangle w_{1}x_{m-1}x_{c+2}(x_{c})^{k'+1} \rrangle_{c}
= A^{-2}\llangle w_{1}x_{m-1}\lambda x_{c+1}(x_{c})^{k'+1} \rrangle_{c} - A^{-3}\llangle w_{1}x_{m-1}(x_{c})^{k'+2} \rrangle_{c}. 
\end{equation*}
For the first term in the above equation, using part f) in the definition of $\llangle\cdot\rrangle_{c}$,
\begin{eqnarray*}
A^{-2}\llangle w_{1}x_{m-1}\lambda x_{c+1}(x_{c})^{k'+1} \rrangle_{c}
&=& - A^{-3}\llangle w_{1}\lambda P_{c-m+1,1}(x_{c})^{k'+1} \rrangle_{c} + 2A^{-2} \llangle w_{1}P_{c-m,1}(x_{c})^{k'+1} \rrangle_{c} \\
&+& A^{-4}\llangle w_{1}x_{m}\lambda (x_{c})^{k'+2} \rrangle_{c},
\end{eqnarray*}
and we rewrite each term on the right hand side of the equation just above as follows:
\begin{eqnarray*}
- A^{-3}\llangle w_{1}\lambda P_{c-m+1,1}(x_{c})^{k'+1} \rrangle_{c} 
&=& - A^{-3}\llangle w_{1}\lambda (AP_{c-m+2} + A^{-1}P_{c-m})(x_{c})^{k'+1} \rrangle_{c} \\
&=& - A^{-3}\llangle w_{1}(A^{2}P_{c-m+1} + P_{c-m+3} + P_{c-m-1} + A^{-2}P_{c-m+1})(x_{c})^{k'+1} \rrangle_{c}
\end{eqnarray*}
by \eqref{eqn:Pnk} and \eqref{eqn:Pn},
\begin{equation*}
2A^{-2} \llangle w_{1}P_{c-m,1}(x_{c})^{k'+1} \rrangle_{c} = 2A^{-2} \llangle w_{1}(AP_{c-m+1} + A^{-1}P_{c-m-1})(x_{c})^{k'+1} \rrangle_{c}
\end{equation*}
by \eqref{eqn:Pnk}, and 
\begin{equation*}
A^{-4}\llangle w_{1}x_{m}\lambda (x_{c})^{k'+2} \rrangle_{c} = A^{-3}\llangle w_{1}x_{m-1}(x_{c})^{k'+2} \rrangle_{c} + A^{-5}\llangle w_{1}x_{m+1}(x_{c})^{k'+2} \rrangle_{c}
\end{equation*}
by \eqref{eqn:xm_rel_2}. Moreover, for the second term of \eqref{eqn:case_type_v_RHS}, it follows by \eqref{eqn:case_type_v_simplify} that
\begin{equation*}
A^{-1}\llangle w_{1}P_{c-m+2}x_{c+1}(x_{c})^{k'} \rrangle_{c} = A^{-3}\llangle w_{1}P_{c-m+3}(x_{c})^{k'+1} \rrangle_{c} - A^{-3}\llangle w_{1}x_{m-1} (x_{c})^{k'} \rrangle_{c} + A^{-1}\llangle w_{1}x_{m-3}(x_{c})^{k'} \rrangle_{c}.
\end{equation*}
Therefore, after combining all results above, we see that  \eqref{eqn:case_type_v_RHS} becomes
\begin{eqnarray*}
& & A\llangle w_{1}x_{m-1}x_{c+1}x_{c+1}(x_{c})^{k'} \rrangle_{c} + A^{-1}\llangle w_{1}P_{c-m+2}x_{c+1}(x_{c})^{k'} \rrangle_{c} \\
&=& A^{-1}\llangle w_{1}P_{c-m+1}(x_{c})^{k'+1} \rrangle_{c} + 2A\llangle w_{1}x_{m-1}(x_{c})^{k'} \rrangle_{c}
- A^{-4}\llangle w_{1}x_{m}\lambda(x_{c})^{k'} \rrangle_{c} \\  
&-& A^{-5}\llangle w_{1}P_{c-m+1}(x_{c})^{k'+1} \rrangle_{c}
+ A^{-3}\llangle w_{1}P_{c-m-1}(x_{c})^{k'+1} \rrangle_{c} + A^{-5}\llangle w_{1}x_{m+1}(x_{c})^{k'+2} \rrangle_{c}.
\end{eqnarray*} 
It follows that \eqref{eqn:case_type_v_LHS} and \eqref{eqn:case_type_v_RHS} are same, so \eqref{eqn:case_type_v} holds.

If $\varepsilon = 1$ and $n = c-1$, then the left hand side of \eqref{eqn:case_type_v} becomes
\begin{equation}
\label{eqn:case_type_v_LHS_2}
A\llangle w_{1}P_{c-1-m}x_{c+1}(x_{c})^{k'} \rrangle_{c} + A^{-1}\llangle w_{1}x_{m}x_{c-1}x_{c+1}(x_{c})^{k'} \rrangle_{c}   
\end{equation}
and the right hand side of \eqref{eqn:case_type_v} becomes
\begin{equation}
\label{eqn:case_type_v_RHS_2}
A\llangle w_{1}x_{m-1}x_{c}x_{c+1}(x_{c})^{k'} \rrangle_{c} + A^{-1}\llangle w_{1}P_{c-m+1}x_{c+1}(x_{c})^{k'} \rrangle_{c}.
\end{equation}
After applying \eqref{eqn:case_type_v_simplify}, the first term of \eqref{eqn:case_type_v_LHS_2} becomes
\begin{equation*}
A\llangle w_{1}P_{c-1-m}x_{c+1}(x_{c})^{k'} \rrangle_{c} = A^{-1}\llangle w_{1}P_{c-m}(x_{c})^{k'+1} \rrangle_{c} - A^{-1}\llangle w_{1}x_{m+2}(x_{c})^{k'} \rrangle_{c} + A\llangle w_{1}x_{m}(x_{c})^{k'} \rrangle_{c}.
\end{equation*}
Using part f) in the definition of $\llangle\cdot\rrangle_{c}$, the second term of \eqref{eqn:case_type_v_LHS_2} becomes
\begin{eqnarray*}
& &A^{-1}\llangle w_{1}x_{m}x_{c-1}x_{c+1}(x_{c})^{k'} \rrangle_{c} \\
&=& -A^{-2}\llangle w_{1}x_{m}\lambda P_{1}(x_{c})^{k'} \rrangle_{c} + 2A^{-1}\llangle w_{1}x_{m}P_{0}(x_{c})^{k'} \rrangle_{c} + A^{-3}\llangle w_{1}x_{m}(x_{c})^{k'+2} \rrangle_{c} \\
&=& A\llangle w_{1}x_{m}\lambda^{2} (x_{c})^{k'} \rrangle_{c} - 2A\llangle w_{1}x_{m}(x_{c})^{k'} \rrangle_{c} - 2A^{-3}\llangle w_{1}x_{m}(x_{c})^{k'} \rrangle_{c} + A^{-3}\llangle w_{1}x_{m}(x_{c})^{k'+2} \rrangle_{c}.
\end{eqnarray*}
Moreover, we use \eqref{eqn:xm_rel_2} to rewrite the first term of the above equation as
\begin{eqnarray*}
A\llangle w_{1}x_{m}\lambda^{2} (x_{c})^{k'} \rrangle_{c} 
&=& A^{2}\llangle w_{1}x_{m-1}\lambda (x_{c})^{k'} \rrangle_{c} + \llangle w_{1}x_{m+1}\lambda (x_{c})^{k'} \rrangle_{c} \\
&=& A^{2}\llangle w_{1}x_{m-1}\lambda (x_{c})^{k'} \rrangle_{c} + A\llangle w_{1}x_{m}(x_{c})^{k'} \rrangle_{c} + A^{-1}\llangle w_{1}x_{m+2}(x_{c})^{k'} \rrangle_{c}.
\end{eqnarray*}
Therefore, combining all results above, \eqref{eqn:case_type_v_LHS_2} becomes
\begin{eqnarray*}
& & A\llangle w_{1}P_{c-1-m}x_{c+1}(x_{c})^{k'} \rrangle_{c} + A^{-1}\llangle w_{1}x_{m}x_{c-1}x_{c+1}(x_{c})^{k'} \rrangle_{c} \\
&=& A^{-1}\llangle w_{1}P_{c-m}(x_{c})^{k'+1} \rrangle_{c} + A^{2}\llangle w_{1}x_{m-1}\lambda (x_{c})^{k'} \rrangle_{c} - 2A^{-3}\llangle w_{1}x_{m}(x_{c})^{k'} \rrangle_{c} 
+ A^{-3}\llangle w_{1}x_{m}(x_{c})^{k'+2} \rrangle_{c}.
\end{eqnarray*}
Analogously, after using \eqref{eqn:case_type_v_simplify}, the second term of \eqref{eqn:case_type_v_RHS_2} becomes
\begin{equation*}
A^{-1}\llangle w_{1}P_{c-m+1}x_{c+1}(x_{c})^{k'} \rrangle_{c} = A^{-3}\llangle w_{1}P_{c-m+2}(x_{c})^{k'+1} \rrangle_{c} - A^{-3}\llangle w_{1}x_{m}(x_{c})^{k'} \rrangle_{c} + A^{-1}\llangle w_{1}x_{m-2}(x_{c})^{k'} \rrangle_{c}.
\end{equation*}
For the first term of \eqref{eqn:case_type_v_simplify}, by part f) in the definition of $\llangle\cdot\rrangle_{c}$, we see that
\begin{eqnarray*}
& & A\llangle w_{1}x_{m-1}x_{c}x_{c+1}(x_{c})^{k'} \rrangle_{c} \\
&=& -\llangle w_{1}x_{m-1}\lambda P_{0}(x_{c})^{k'} \rrangle_{c} + 2A\llangle w_{1}x_{m-1}P_{-1}(x_{c})^{k'} \rrangle_{c} + A^{-1}\llangle w_{1}x_{m-1}x_{c+1}(x_{c})^{k'+1} \rrangle_{c} \\
&=& A^{2}\llangle w_{1}x_{m-1}\lambda (x_{c})^{k'} \rrangle_{c} - A^{-2}\llangle w_{1}x_{m-1} \lambda (x_{c})^{k'} \rrangle_{c} + A^{-1}\llangle w_{1}x_{m-1}x_{c+1}(x_{c})^{k'+1} \rrangle_{c}.
\end{eqnarray*}
Moreover, we rewrite the second and last terms of the above equation as
\begin{equation*}
- A^{-2}\llangle w_{1}x_{m-1} \lambda (x_{c})^{k'} \rrangle_{c} = - A^{-1}\llangle w_{1}x_{m-2}(x_{c})^{k'} \rrangle_{c} - A^{-3}\llangle w_{1}x_{m}(x_{c})^{k'} \rrangle_{c}
\end{equation*}
by \eqref{eqn:xm_rel_2} and
\begin{eqnarray*}
& &A^{-1}\llangle w_{1}x_{m-1}x_{c+1}(x_{c})^{k'+1} \rrangle_{c} \\
&=& - A^{-2}\llangle w_{1}\lambda P_{c-m+1}(x_{c})^{k'+1} \rrangle_{c} + 2A^{-1}\llangle w_{1}P_{c-m}(x_{c})^{k'+1} \rrangle_{c} + A^{-3}\llangle w_{1}x_{m}(x_{c})^{k'+2} \rrangle_{c} \\
&=& - A^{-3}\llangle w_{1}P_{c-m+2}(x_{c})^{k'+1} \rrangle_{c} + A^{-1}\llangle w_{1}P_{c-m}(x_{c})^{k'+1} \rrangle_{c} + A^{-3}\llangle w_{1}x_{m}(x_{c})^{k'+2} \rrangle_{c}
\end{eqnarray*}
by part f) in the definition of $\llangle\cdot\rrangle_{c}$ and \eqref{eqn:Pn}.
Therefore, after combining all the results above, we see that \eqref{eqn:case_type_v_RHS_2} becomes
\begin{eqnarray*}
& & A\llangle w_{1}x_{m-1}x_{c}x_{c+1}(x_{c})^{k'} \rrangle_{c} + A^{-1}\llangle w_{1}P_{c-m+1}x_{c+1}(x_{c})^{k'} \rrangle_{c} \\
&=& - 2A^{-3}\llangle w_{1}x_{m}(x_{c})^{k'} \rrangle_{c} + A^{2}\llangle w_{1}x_{m-1}\lambda (x_{c})^{k'} \rrangle_{c} + A^{-1}\llangle w_{1}P_{c-m}(x_{c})^{k'+1} \rrangle_{c} + A^{-3}\llangle w_{1}x_{m}(x_{c})^{k'+2} \rrangle_{c}.
\end{eqnarray*}
It follows that \eqref{eqn:case_type_v_LHS_2} and \eqref{eqn:case_type_v_RHS_2} are same, so \eqref{eqn:case_type_v} holds.

Since \eqref{eqn:case_type_v} holds for $n' = 0$ and $n = c,c-1$, using Lemma~\ref{lem:annulus_for_any_kn}, we conclude that it also holds for any $n' > 0$ and $n \in \mathbb{Z}$.
\end{proof}

\begin{figure}[ht]
\centering
\includegraphics[scale=0.6]{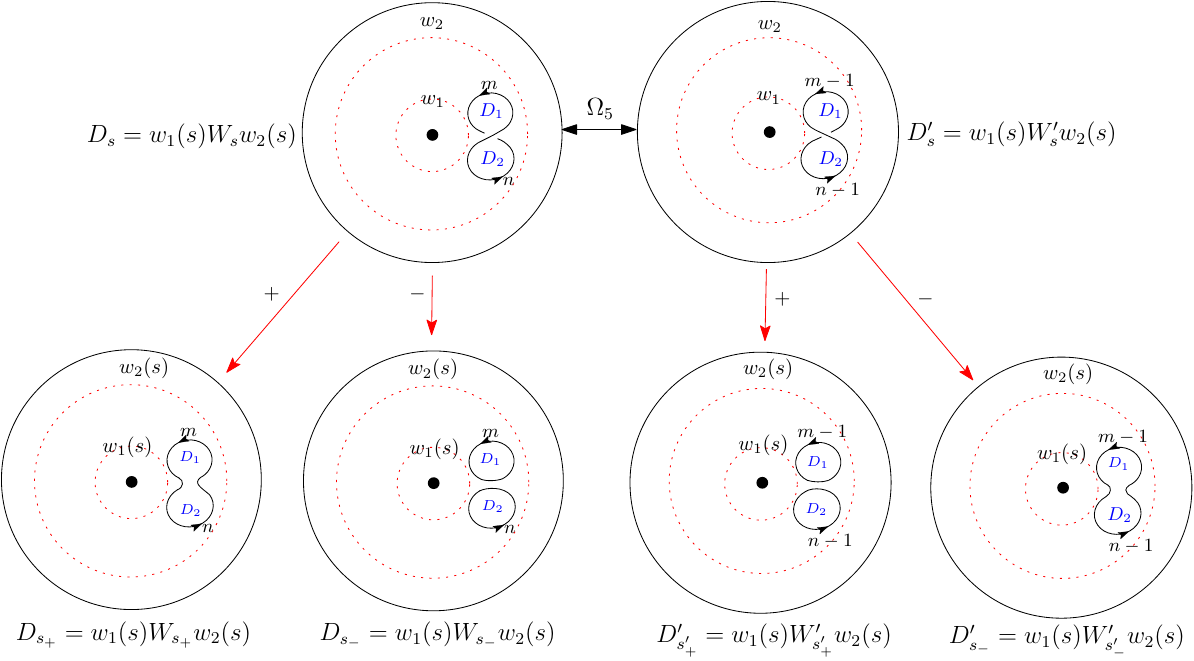}
\caption{$D_{s} = w_{1}(s)W_{s}w_{2}(s)$ and $D'_{s}= w_{1}(s)W'_{s}w_{2}(s)$ in ${\bf A}^{2}$ related by $\Omega_{5}$-move}
\label{fig:Omega5Kauffman_1}
\end{figure}

\begin{figure}[ht]
\centering
\includegraphics[scale=0.6]{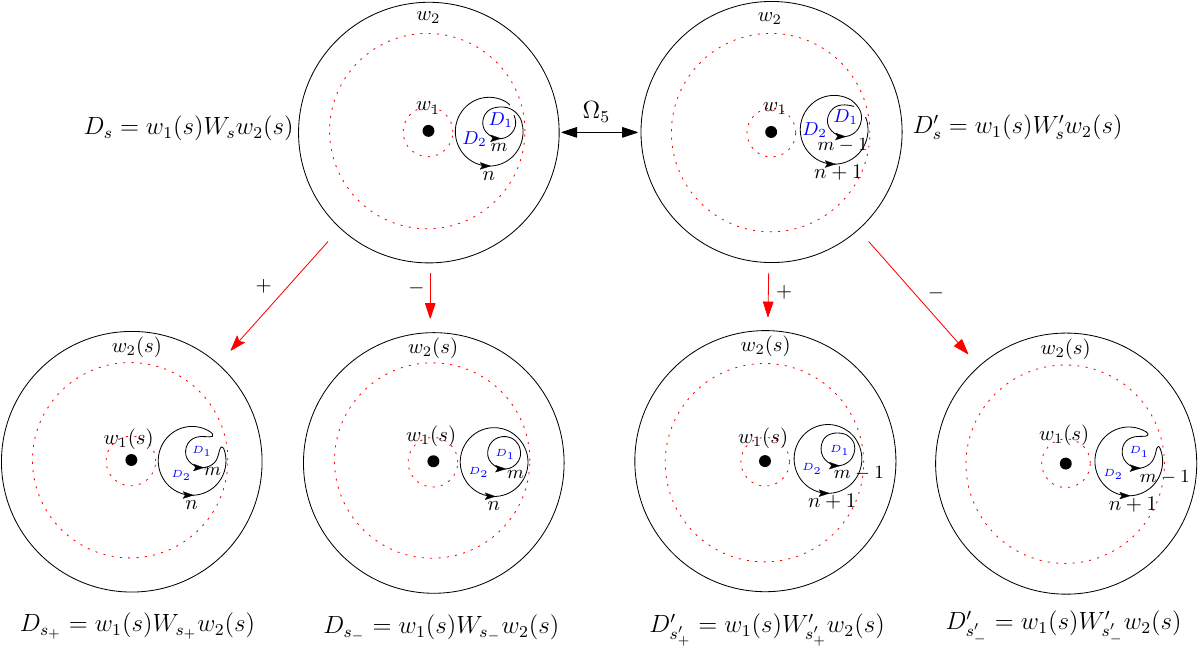}
\caption{$D_{s} = w_{1}(s)W_{s}w_{2}(s)$ and $D'_{s}= w_{1}(s)W'_{s}w_{2}(s)$ in ${\bf A}^{2}$ related by $\Omega_{5}$-move}
\label{fig:Omega5Kauffman_2}
\end{figure}

\begin{lemma}
\label{lem:main_lemma_for_h}
For all arrow diagrams $D,D'$ in ${\bf A}^{2}$ related by $\Omega_{1} - \Omega_{5}$-moves, $\psi_{c}(D-D') = 0$. Therefore, the map $\psi_{c}:R\mathcal{D}({\bf A}^{2}) \to R\Sigma_{c}$ is a well-defined homomorphism of free $R$-modules. 
\end{lemma}

\begin{proof}
If $D$ and $D'$ are related by $\Omega_{i}$-move with $i=1,2,3,4$, then $\psi_{c}(D-D') = 0$. Assume that $D$ and $D'$ are related by an $\Omega_{5}$-move in a $2$-disk ${\bf D}^{2}\subset {\bf A}^{2}$. Since both diagrams have the same number of crossings, we label by $c_{1},c_{2},\ldots,c_{k-1},c_{k}$ and $c_{1},c_{2},\ldots,c_{k-1},c'_{k}$ with $c_{k}$ and $c'_{k}$ assigned to the crossing inside ${\bf D}^{2}$, respectively. Furthermore,
\begin{equation*}
\mathcal{K}(D) = \mathcal{K^{+}}(D)\cup \mathcal{K^{-}}(D) \quad \text{and} \quad \mathcal{K}(D') = \mathcal{K^{+}}(D')\cup \mathcal{K^{-}}(D'),   
\end{equation*}
where $\mathcal{K^{+}}(D),\mathcal{K^{-}}(D)$ and $\mathcal{K^{+}}(D'),\mathcal{K^{-}}(D')$ are all Kauffman states of $D$ and $D'$ with markers $\pm 1$ at $c_{k}$ and $c'_{k}$, respectively. For $s_{+}\in \mathcal{K^{+}}(D)$ and $s'_{+}\in \mathcal{K^{+}}(D')$, $s_{+} = s'_{+} = (s,+1)$ and, for $s_{-}\in \mathcal{K^{-}}(D)$ and $s'_{-}\in \mathcal{K^{-}}(D')$, $s_{-} = s'_{-} = (s,-1)$, where $s\in S_{k-1} = \{-1,1\}^{k-1}$. Let $s\in S_{k-1}$, then smoothing crossings $c_{1},c_{2},\ldots,c_{k-1}$ of $D$ and $D'$ results in arrow diagrams $D_{s}$ and $D'_{s}$ in the top of Figure~\ref{fig:Omega5Kauffman_1}, Figure~\ref{fig:Omega5Kauffman_2}, Figure~\ref{fig:Omega5Kauffman_3}, Figure~\ref{fig:Omega5Kauffman_4}, and Figure~\ref{fig:Omega5Kauffman_5}. Thus, smoothing crossings $c_{k}$ and $c'_{k}$ of $D_{s}$ and $D'_{s}$ results in diagrams $D_{s_{+}}$, $D'_{s_{+}}$, $D_{s_{-}}$ and $D'_{s_{-}}$ on the bottom of Figure~\ref{fig:Omega5Kauffman_1}, Figure~\ref{fig:Omega5Kauffman_2}, Figure~\ref{fig:Omega5Kauffman_3}, Figure~\ref{fig:Omega5Kauffman_4}, and Figure~\ref{fig:Omega5Kauffman_5}. Therefore,
\begin{eqnarray*}
\llangle D - D' \rrangle &=& \sum_{s\in S_{k-1}} A^{p(s)-n(s)} \langle w_{1}(s)(A W_{s_{+}} + A^{-1} W_{s_{-}} - A W'_{s'_{+}} - A^{-1} W'_{s'_{-}})w_{2}(s) \rangle
\end{eqnarray*}
and consequently, it suffices to show that
\begin{equation*}
\llangle \llangle w_{1}(s)(A W_{s_{+}} + A^{-1} W_{s_{-}} - A W'_{s'_{+}} - A^{-1} W'_{s'_{-}})w_{2}(s)\rrangle_{\Gamma} \rrangle_{c} = 0.  
\end{equation*}
Clearly, for $D_{s_{+}}$, $D_{s_{-}}$, $D'_{s'_{+}}$, and $D'_{s'_{-}}$ in Figure~\ref{fig:Omega5Kauffman_1} or Figure~\ref{fig:Omega5Kauffman_2}, the above equation follows from \eqref{eqn:InvarianceOnOmega5_1_2}.

\begin{figure}[ht]
\centering
\includegraphics[scale=0.6]{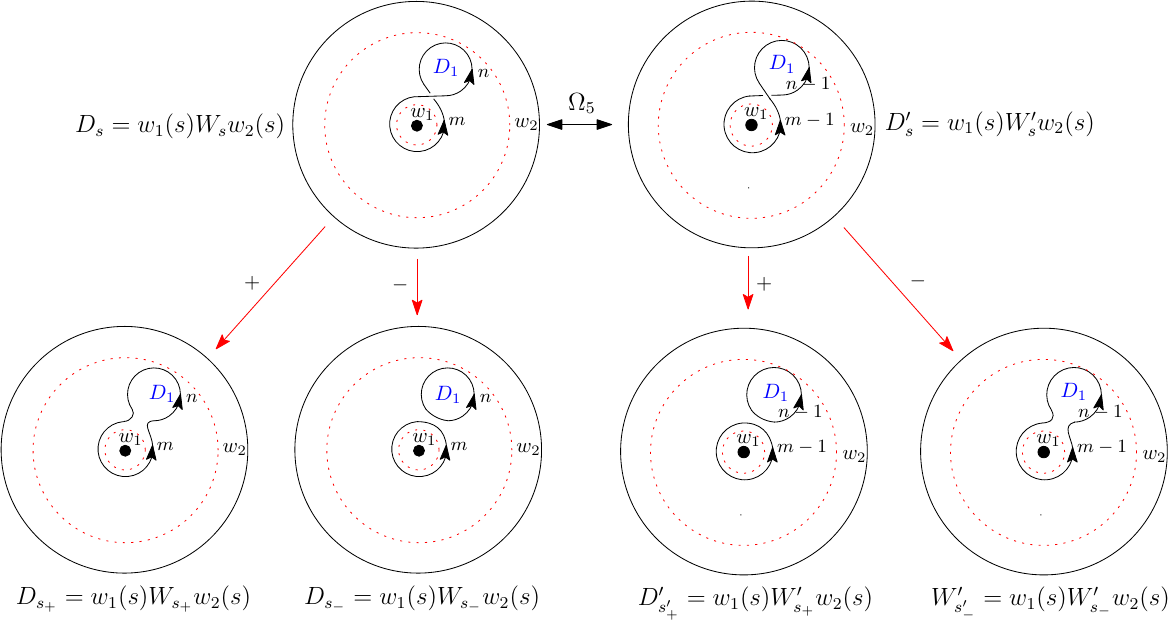}
\caption{$D_{s} = w_{1}(s)W_{s}w_{2}(s)$ and $D'_{s}= w_{1}(s)W'_{s}w_{2}(s)$ in ${\bf A}^{2}$ related by $\Omega_{5}$-move}
\label{fig:Omega5Kauffman_3}
\end{figure}

For $D_{s}$ and $D'_{s}$ in Figure~\ref{fig:Omega5Kauffman_3}, suppose that
\begin{equation*}
\langle D_{1} \rangle_{r} = \sum_{i_{1}=0}^{k_{1}}r_{1,i_{1}}\lambda^{i_{1}}, \,\, \llangle w_{1}(s) \rrangle_{\Gamma} = \sum_{i_{2}=0}^{k_{2}}r_{2,i_{2}}u_{2,i_{2}},\ \text{and}\ \llangle w_{2}(s) \rrangle_{\Gamma} = \sum_{i_{3}=0}^{k_{3}}r_{3,i_{3}}u_{3,i_{3}},
\end{equation*}
then it follows from \eqref{eqn:H3_InvariantOmega5_3} that
\begin{eqnarray*}
&&\llangle \llangle (w_{1}(s)(A W_{s_{+}} + A^{-1} W_{s_{-}} -A W_{s'_{+}} - A W_{s'_{-}})w_{2}(s) \rrangle_{\Gamma}\rrangle_{c}\\
&=&\sum_{i_{1} = 0}^{k_{1}}r_{1,i_{1}}\llangle \llangle w_{1}(s)\rrangle_{\Gamma} (A\lambda^{i_{1}}x_{n+m} +A^{-1}x_{m}P_{n,i_{1}} -Ax_{m-1}P_{n-1,i_{1}} - A^{-1}\lambda^{i_{1}}x_{n+m-2})\llangle w_{2}(s) \rrangle_{\Gamma}\rrangle_{c}\\
&=&\sum_{i_{1}= 0}^{k_{1}}\sum_{i_{2}=0}^{k_{2}}\sum_{i_{3}=0}^{k_{3}}r_{1,i_{1}}r_{2,i_{2}}r_{3,i_{3}}\llangle u_{2,i_{2}}(A\lambda^{i_{1}}x_{n+m} +A^{-1}x_{m}P_{n,i_{1}} -Ax_{m-1}P_{n-1,i_{1}} - A^{-1}\lambda^{i_{1}}x_{n+m-2})u_{3,i_{3}}\rrangle_{c} \\
&=& 0.
\end{eqnarray*}
Analogous argument can be applied for $D_{s}$ and $D'_{s}$ in Figure~\ref{fig:Omega5Kauffman_4} and Figure~\ref{fig:Omega5Kauffman_5} by using \eqref{eqn:H3_InvariantOmega5_4} and \eqref{eqn:H3_InvariantOmega5_5}, respectively. Consequently, $\psi_{c}(D-D') = 0$ for $D$ and $D'$ related by $\Omega_{1}-\Omega_{5}$ moves.

\begin{figure}[ht]
\centering
\includegraphics[scale=0.6]{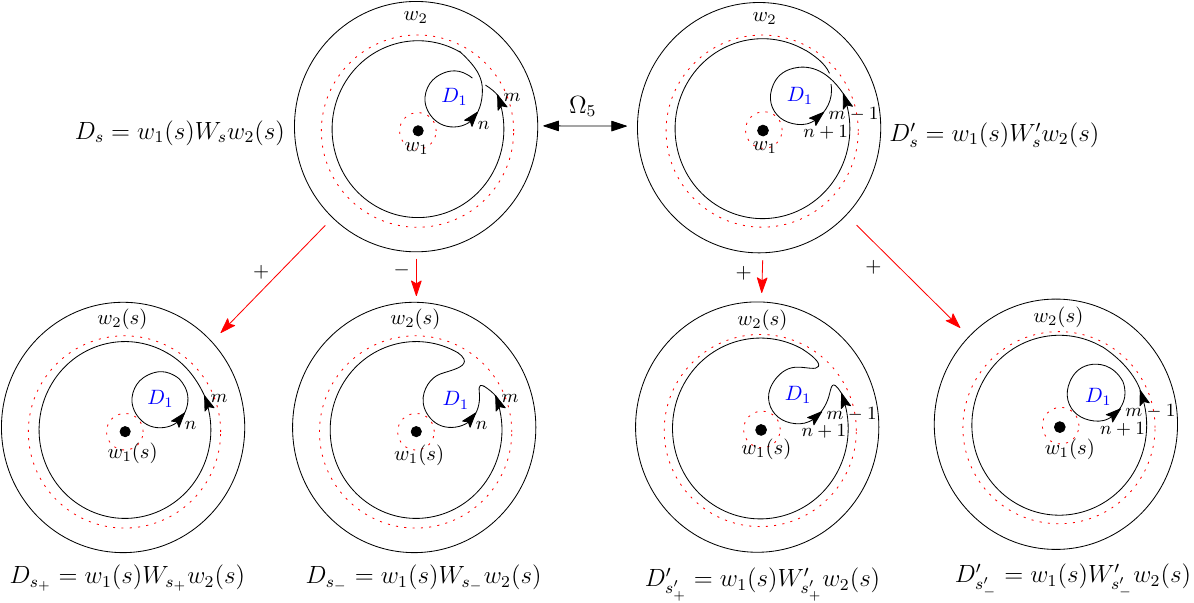}
\caption{$D_{s} = w_{1}(s)W_{s}w_{2}(s)$ and $D'_{s}= w_{1}(s)W'_{s}w_{2}(s)$ in ${\bf A}^{2}$ related by $\Omega_{5}$-move}
\label{fig:Omega5Kauffman_4}
\end{figure}

\begin{figure}[ht]
\centering
\includegraphics[scale=0.6]{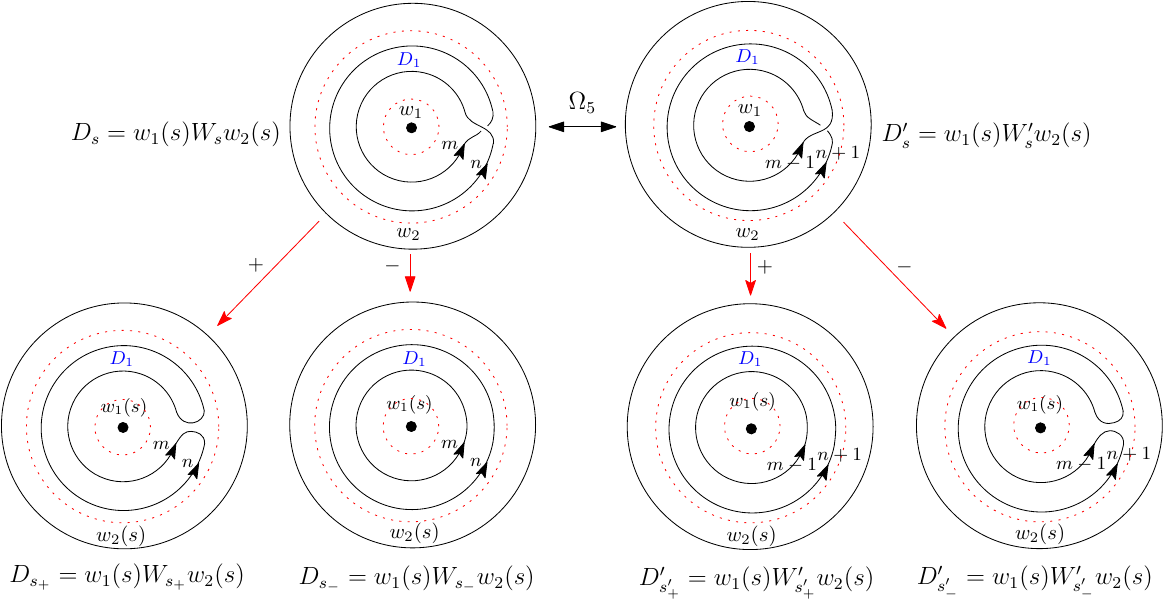}
\caption{$D_{s} = w_{1}(s)W_{s}w_{2}(s)$ and $D'_{s}= w_{1}(s)W'_{s}w_{2}(s)$ in ${\bf A}^{2}$ related by $\Omega_{5}$-move}
\label{fig:Omega5Kauffman_5}
\end{figure}
\end{proof} 

\begin{theorem}
\label{thm:basis_Sigma_c_for_A2S1}
KBSM of ${\bf A}^{2} \times S^{1}$ is a free $R$-module with basis that consists ambient isotopy classes of generic framed links with arrow diagrams in $\Sigma_{c}$, i.e.,
\begin{equation*}
\mathcal{S}_{2,\infty}({\bf A}^{2}\times S^{1};R,A) \cong R\Sigma_{c}.   
\end{equation*}
\end{theorem}

\begin{proof}
Since $\Sigma_{c} \subset \mathcal{D}({\bf A}^{2})$, we define a homomorphism of $R$-modules
\begin{equation*}
\varphi: R\Sigma_{c} \to  S\mathcal{D}({\bf A}^{2}), \,\,\varphi(w) = [w] 
\end{equation*}
which assigns to $w \in \Sigma_{c}$ its equivalence class $[w]$ modulo the submodule $S_{2,\infty}({\bf A}^{2})$. For any arrow diagram $D$ in ${\bf A}^{2}$ and let $W = \psi_{c}(D)$, then $\varphi(W) = [W] = [D]$. Since $S\mathcal{D}({\bf A}^{2})$ is generated by $\mathcal{D}({\bf A}^{2})$, $\varphi$ is surjective.

We show that $\varphi$ is also injective. Since $\Sigma_{c} \subset \mathcal{D}({\bf A}^{2})$, by Lemma~\ref{lem:main_lemma_for_h}, $\psi_{c} : R\mathcal{D}({\bf A}^{2}) \to R\Sigma_{c}$ is an epimorphism of free $R$-modules. Furthermore, as we noted it before
\begin{eqnarray*}
&&\psi_{c}(D_{+} - A D_{0} - D_{\infty}) = 0 \quad \text{and} \quad \psi_{c}(D\sqcup T_{1} +(A^{-2}+A^{2})D) = 0,
\end{eqnarray*}
for any skein triple $D_{+}$, $D_{0}$, $D_{\infty}$ of arrow diagrams and disjoint unions $D\sqcup T_{1}$ (see Figure~\ref{fig:SkeinTripleOfDiagrams}). Consequently, $\psi_{c}$ descends to a surjective homomorphism of $R$-modules
\begin{equation*}
\hat{\psi}_{c}: S\mathcal{D}({\bf A}^{2}) \to R\Sigma_{c},
\end{equation*}
which assigns $\psi_{c}(D)$ for a generator $[D]$. Since, for each $W\in R\Sigma_{c}$, by the definition of $\psi_{c}$, $\psi_{c}(W) = W$, so
\begin{equation*}
(\hat{\psi}_{c} \circ \varphi)(W) = \hat{\psi_{c}}([W]) = \psi_{c}(W) = W.    
\end{equation*}
Therefore, $\varphi$ is also injective. Consequently, both $\varphi$ and $\hat{\psi}_{c}$ are isomorphisms. 

By Theorem~1 of \cite{GM2017}, there is a bijection between ambient isotopy classes of framed links in ${\bf A}^{2}\times S^{1}$ and equivalence classes of arrow diagrams in ${\bf A}^{2}$ modulo $\Omega_{1}-\Omega_{5}$ moves, consequently
\begin{equation*}
\mathcal{S}_{2,\infty}({\bf A}^{2}\times S^{1};R,A) \cong S\mathcal{D}({\bf A}^{2}) \underset{\hat{\psi}_{c}}{\cong} R\Sigma_{c}.  
\end{equation*}
\end{proof}

\section{Basis for KBSM of \texorpdfstring{$V(\beta,2)$}{V(\unichar{0946},2)}}
\label{s:basis_V}

Let $\mathcal{D}({\bf D}^{2}_{\beta})$ be the set of equivalence classes of arrow diagrams in ${\bf D}^{2}_{\beta}$, modulo $\Omega_{1}-\Omega_{5}$ and $S_{\beta}$-moves. We may regard $\mathcal{D}({\bf A}^{2})$ as a subset of $\mathcal{D}({\bf D}^{2}_{\beta})$ and consequently $\Gamma$ and
\begin{equation*}
\Sigma_{\nu} = \{ \lambda^{n}, x_{\nu}\lambda^{n}(x_{\nu})^{k}, x_{\nu+1}\lambda^{n}(x_{\nu})^{k} \mid \ n,k \geq 0 \}, \ \nu = \lfloor \frac{\beta}{2} \rfloor,
\end{equation*}
are also subsets of $\mathcal{D}({\bf D}^{2}_{\beta})$. Let $S_{2,\infty}({\bf D}^{2}_{\beta})$ be a submodule of $R\mathcal{D}({\bf D}^{2}_{\beta})$ generated by 
\begin{equation*}
D_{+} - AD_{0}-A^{-1}D_{\infty} \quad \text{and} \quad D\sqcup T_{1}+(A^{-2}+A^{2})D, 
\end{equation*}
for all skein triples $D_{+},D_{0}$, and $D_{\infty}$ of arrow diagrams in ${\bf D}^{2}_{\beta}$, and all disjoint unions $D\sqcup T_{1}$ of an arrow diagram $D$ and a trivial circle\footnote{A trivial circle in ${\bf D}^{2}_{\beta}$ is a curve equivalent modulo $\Omega_{1}-\Omega_{5}$ and $S_{\beta}$-moves to a curve with no arrows that bounds a $2$-disk which does not contain $\beta$.} (see Figure~\ref{fig:SkeinTripleOfDiagrams}). 
In this section, we show that
\begin{equation*}
R\Sigma^{\prime}_{\nu} \cong S\mathcal{D}({\bf D}^{2}_{\beta}) =  R\mathcal{D}({\bf D}^{2}_{\beta})/S_{2,\infty}({\bf D}^{2}_{\beta}),   
\end{equation*}
where 
\begin{equation*}
\Sigma^{\prime}_{\nu} = \{\lambda^{n}, x_{\nu}\lambda^{n} \mid n \geq 0\}.
\end{equation*}

\begin{lemma}
\label{lem:Sigma_2_beta} 
Let $w\in \Sigma_{\nu}$, then in $S\mathcal{D}({\bf D}^{2}_{\beta})$,
\begin{enumerate}
\item for $w = x_{\nu+1}\lambda^{n}(x_{\nu})^{k}$,
\begin{equation*}
w = -A^{3}x_{\nu}\lambda^{n}(x_{\nu})^{k};
\end{equation*}
\item for $w = x_{\nu}\lambda^{n}(x_{\nu})^{k}$ with $k \geq 1$,
\begin{equation*}
w = A^{-1}t_{-1,n}(x_{\nu})^{k-1} - A^{-2}t_{0,n}(x_{\nu})^{k-1}.
\end{equation*}
\end{enumerate}
\end{lemma}

\begin{figure}[ht]
\centering
\includegraphics[scale=0.7]{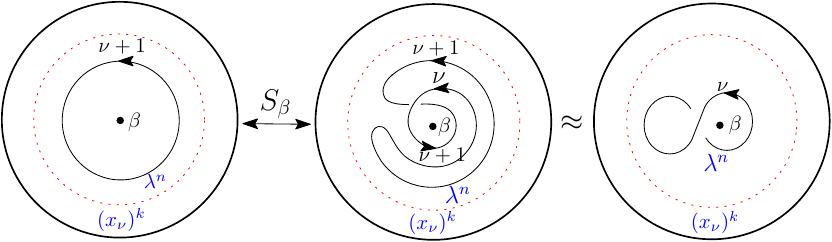}
\caption{Identity $x_{\nu+1}\lambda^{n}(x_{\nu})^{k} = -A^{3}x_{\nu}\lambda^{n}(x_{\nu})^{k}$}
\label{fig:V_beta_2_case2_new}
\end{figure}

\begin{proof}
For 1) we notice that, since arrow diagrams shown on the left and the right of Figure~\ref{fig:V_beta_2_case2_new} are related by an $S_{\beta}$-move on ${\bf D}^{2}_{\beta}$, 
\begin{equation*}
x_{\nu}\lambda^{n}(x_{\nu})^{k} = Ax_{\nu+1}\lambda^{n}(x_{\nu})^{k} + A^{-1}x_{\nu+1}(-A^{-2}-A^{2})\lambda^{n}(x_{\nu})^{k} = -A^{-3}x_{\nu+1}\lambda^{n}(x_{\nu})^{k}.
\end{equation*}

Analogously, for 2), since arrow diagrams shown on the left and the right hand side of Figure~\ref{fig:V_beta_2_case1_new} differ by an $S_{\beta}$-move, 
\begin{equation*}
t_{0,n}(x_{\nu})^{k-1} = At_{-1,n}(x_{\nu})^{k-1} + A^{-1}x_{\nu+1}\lambda^{n}(x_{\nu})^{k}.
\end{equation*}
Therefore, by 1)
\begin{equation*}
x_{\nu}\lambda^{n}(x_{\nu})^{k} = -A^{-3}x_{\nu+1}\lambda^{n}(x_{\nu})^{k} =
A^{-1}t_{-1,n}(x_{\nu})^{k-1} - A^{-2}t_{0,n}(x_{\nu})^{k-1}.
\end{equation*}
\end{proof}

\begin{figure}[ht]
\centering
\includegraphics[scale=0.7]{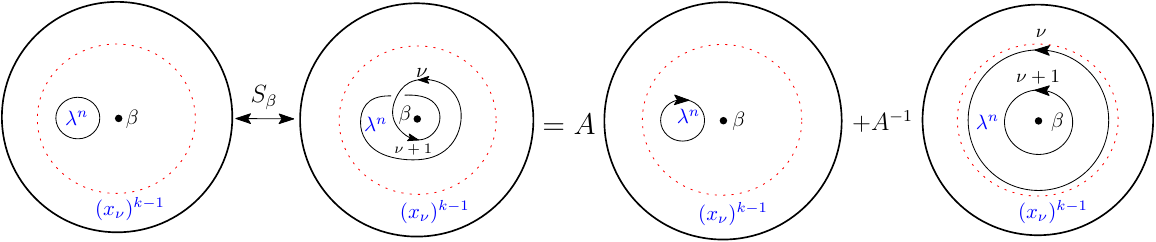}
\caption{Identity $t_{0,n}(x_{\nu})^{k-1} = At_{-1,n}(x_{\nu})^{k-1} + A^{-1}x_{\nu+1}\lambda^{n}(x_{\nu})^{k}$}
\label{fig:V_beta_2_case1_new}
\end{figure}

For curves $t_{-1,n}(x_{\nu})^{k-1}$ and $t_{0,n}(x_{\nu})^{k-1}$ on the right-hand side of 2) in Lemma~\ref{lem:Sigma_2_beta}, we use $\llangle \cdot \rrangle_{\Gamma}$ to express them as $\llangle t_{-1,n}(x_{\nu})^{k-1} \rrangle_{\Gamma} = P_{-1,n}(x_{\nu})^{k-1}$ and $\llangle t_{0,n}(x_{\nu})^{k-1} \rrangle_{\Gamma} = P_{0,n}(x_{\nu})^{k-1}$, respectively. Using $\llangle \cdot \rrangle_{\nu}$ we can further write them as elements of $R\Sigma_{\nu}$ with the number of $x_{\nu}$-curves $\leq k-1$. Consequently, using Lemma~\ref{lem:Sigma_2_beta}, $\llangle \cdot \rrangle_{\Gamma}$, and  $\llangle \cdot \rrangle_{\nu}$, each $w\in \Sigma_{\nu}$ can be represented as a $R$-linear combination of $x_{\nu +\varepsilon}\lambda^{n}$ with $\varepsilon = 0,1$ and $\lambda^{n}$, $n\geq 0$. Furthermore, by using 1) in Lemma~\ref{lem:Sigma_2_beta}, we further reduce each $x_{\nu +1}\lambda^{n}$ into $x_{\nu}\lambda^{n}$. As a result, each $w\in \Sigma_{\nu}$ can be represented as an $R$-linear combination of elements in $\Sigma^{\prime}_{\nu}$. Hence, for $w\in R\Gamma$, we define an extension $\llangle \cdot \rrangle_{\Sigma^{\prime}_{\nu}}$ of $\llangle \cdot \rrangle_{\nu}$ as follows:
\begin{enumerate}[(a)]
\item if $w = \sum_{w' \in S} r_{w'}w'$, $S$ is a finite subset of $\Gamma$ with at least two elements, let 
\begin{equation*}
\llangle w \rrangle_{\Sigma^{\prime}_{\nu}} = \sum_{w'\in S} r_{w'}\llangle w' \rrangle_{\Sigma^{\prime}_{\nu}};
\end{equation*}
\item for $w \in \Sigma^{\prime}_{\nu}$,
\begin{equation*}
\llangle w \rrangle_{\Sigma^{\prime}_{\nu}} = w;
\end{equation*}
\item for $w = w'x_{m}\lambda^{n}(x_{\nu})^{k}$, where $w'$ contains at least one $x$ curve and $n \geq 1$,
\begin{equation*}
\llangle w \rrangle_{\Sigma^{\prime}_{\nu}} = A \llangle w'x_{m-1}\lambda^{n-1}(x_{\nu})^{k} \rrangle_{\Sigma^{\prime}_{\nu}} + A^{-1} \llangle w'x_{m+1}\lambda^{n-1}(x_{\nu})^{k} \rrangle_{\Sigma^{\prime}_{\nu}};
\end{equation*}
\item for $w = w'x_{m}(x_{\nu})^{k}$, where $w'$ contains at least one $x$ curve and $m > \nu + 1$,
\begin{equation*}
\llangle w \rrangle_{\Sigma^{\prime}_{\nu}} = A^{-1} \llangle w'\lambda x_{m-1}(x_{\nu})^{k} \rrangle_{\Sigma^{\prime}_{\nu}} - A^{-2} \llangle w'x_{m-2}(x_{\nu})^{k} \rrangle_{\Sigma^{\prime}_{\nu}};
\end{equation*}
\item for $w = w'x_{m}(x_{\nu})^{k}$, where $w'$ contains at least one $x$ curve and $m < \nu$,
\begin{equation*}
\llangle w \rrangle_{\Sigma^{\prime}_{\nu}} = A \llangle w'\lambda x_{m+1}(x_{\nu})^{k} \rrangle_{\Sigma^{\prime}_{\nu}} - A^{2} \llangle w'x_{m+2}(x_{\nu})^{k} \rrangle_{\Sigma^{\prime}_{\nu}};
\end{equation*}
\item for $w = w'x_{m}\lambda^{n}x_{\nu+1}(x_{\nu})^{k}$,
\begin{equation*}
\llangle w \rrangle_{\Sigma^{\prime}_{\nu}} = -A^{-1} \llangle w'\lambda P_{\nu-m,n}(x_{\nu})^{k} \rrangle_{\Sigma^{\prime}_{\nu}} + 2 \llangle w'P_{\nu-1-m,n}(x_{\nu})^{k} \rrangle_{\Sigma^{\prime}_{\nu}} + A^{-2} \llangle w'x_{m+1}\lambda^{n}(x_{\nu})^{k+1} \rrangle_{\Sigma^{\prime}_{\nu}};
\end{equation*}
\item for $w = \lambda^{n}x_{m}\lambda^{n'}(x_{\nu})^{k}$ with $n \geq 1$,
\begin{equation*}
\llangle w \rrangle_{\Sigma^{\prime}_{\nu}} = A \llangle \lambda^{n-1}x_{m+1}\lambda^{n'}(x_{\nu})^{k} \rrangle_{\Sigma^{\prime}_{\nu}} + A^{-1} \llangle \lambda^{n-1}x_{m-1}\lambda^{n'}(x_{\nu})^{k} \rrangle_{\Sigma^{\prime}_{\nu}};
\end{equation*}
\item for $w = x_{m}\lambda^{n}(x_{\nu})^{k}$ with $m > \nu + 1$,
\begin{equation*}
\llangle w \rrangle_{\Sigma^{\prime}_{\nu}} = A \llangle x_{m-1}\lambda^{n+1}(x_{\nu})^{k} \rrangle_{\Sigma^{\prime}_{\nu}} - A^{2} \llangle x_{m-2}\lambda^{n}(x_{\nu})^{k} \rrangle_{\Sigma^{\prime}_{\nu}};
\end{equation*}
\item for $w = x_{m}\lambda^{n}(x_{\nu})^{k}$ with $m < \nu$,
\begin{equation*}
\llangle w \rrangle_{\Sigma^{\prime}_{\nu}} = A^{-1} \llangle x_{m+1}\lambda^{n+1}(x_{\nu})^{k} \rrangle_{\Sigma^{\prime}_{\nu}} - A^{-2} \llangle x_{m+2}\lambda^{n}(x_{\nu})^{k} \rrangle_{\Sigma^{\prime}_{\nu}};
\end{equation*}
\item for $w = x_{\nu+1}\lambda^{n}(x_{\nu})^{k}$,
\begin{equation*}
\llangle w \rrangle_{\Sigma^{\prime}_{\nu}} = -A^{3}\llangle x_{\nu}\lambda^{n}(x_{\nu})^{k} \rrangle_{\Sigma^{\prime}_{\nu}};
\end{equation*}
\item for $w = x_{\nu}\lambda^{n}(x_{\nu})^{k}$ with $k \geq 1$,
\begin{equation*}
\llangle w \rrangle_{\Sigma^{\prime}_{\nu}} = A^{-1}\llangle P_{-1,n}(x_{\nu})^{k-1} \rrangle_{\Sigma^{\prime}_{\nu}} - A^{-2}\llangle P_{0,n}(x_{\nu})^{k-1} \rrangle_{\Sigma^{\prime}_{\nu}}. 
\end{equation*}
\end{enumerate}

Given an arrow diagram $D$ in ${\bf D}^{2}_{\beta}$, we first apply $\llangle \cdot \rrangle$ to express $D$ as a $R$-linear combination of diagrams without crossings $w$. For each such $w$ (hence also entire linear combination) we are then able to apply $\llangle \cdot \rrangle_{\Gamma}$ to write it as a $R$-linear combination of elements $u$ in $\Gamma$. Since for each such $u$ we can use $\llangle \cdot \rrangle_{\Sigma_{\nu}^{\prime}}$ to express it as a $R$-linear combination of elements $\Sigma_{\nu}^{\prime}$. Therefore, for each arrow diagram $D$ we can assign an element of the free $R$-module $R\Sigma_{\nu}^{\prime}$ given by
\begin{equation*}
\phi_{\beta}(D) = \llangle \llangle \llangle D \rrangle \rrangle_{\Gamma} \rrangle_{\Sigma_{\nu}^{\prime}}.
\end{equation*}
It is clear that, for any skein triple $D_{+}$, $D_{0}$, $D_{\infty}$, and disjoint unions $D\sqcup T_{1}$ (see Figure~\ref{fig:SkeinTripleOfDiagrams}),
\begin{equation*}
\phi_{\beta}(D_{+} - A D_{0} - A^{-1} D_{\infty}) = 0\,\,\text{and}\,\,\phi_{\beta}( D\sqcup T_{1}  + (A^{2}+A^{-2}) D) = 0,
\end{equation*}
where $T_{1}$ is a trivial circle. We will show that if arrow diagrams $D$ and $D'$ are related on ${\bf D}^{2}_{\beta}$ by $\Omega_{1}-\Omega_{5}$ and $S_{\beta}$-moves, then $\phi_{\beta}(D - D') = 0$. 

Since items a) -- i) included in the definition of $\llangle \cdot \rrangle_{\Sigma^{\prime}_{\nu}}$ define $\llangle \cdot \rrangle_{\nu}$, as a consequence of Lemma~\ref{lem:main_lemma_for_h}, we state for a reference the following property of $\phi_{\beta}$.

\begin{lemma}
\label{lem:lemma_for_h4}
If arrow diagrams $D$ and $D'$ in ${\bf D}^{2}_{\beta}$ are related by $\Omega_{1}-\Omega_{5}$ moves then
\begin{equation*}
\phi_{\beta}(D-D') = 0.
\end{equation*}
\end{lemma}

Furthermore, for $\llangle \cdot \rrangle_{\Sigma^{\prime}_{\nu}}$ one can prove results analogous to Lemma~\ref{lem:rel_xm} and Lemma~\ref{lem:annulus_for_any_kn} which were obtained for $\llangle \cdot \rrangle_{\nu}$ in Section~\ref{s:basis_A2TimesS1}.

\begin{lemma}
\label{lem:rel_xm_bracket_2_beta}
For any $k \in \mathbb{Z}$ and $w_{1}x_{m}w_{2} \in \Gamma$ with $m \in \mathbb{Z}$,
\begin{equation}
\label{eqn:rel_xm_bracket_2_beta_1}
\llangle w_{1}x_{m}w_{2} \rrangle_{\Sigma^{\prime}_{\nu}} = -A^{m-k}\llangle w_{1}x_{k}Q_{m-k-1}w_{2} \rrangle_{\Sigma^{\prime}_{\nu}} + A^{m-k-1}\llangle w_{1}x_{k+1}Q_{m-k}w_{2} \rrangle_{\Sigma^{\prime}_{\nu}}
\end{equation}
and
\begin{equation*}
\llangle w_{1}x_{m}w_{2} \rrangle_{\Sigma^{\prime}_{\nu}} = -A^{k-m}\llangle w_{1}Q_{m-k-1}x_{k}w_{2} \rrangle_{\Sigma^{\prime}_{\nu}} + A^{k-m+1}\llangle w_{1}Q_{m-k}x_{k+1}w_{2} \rrangle_{\Sigma^{\prime}_{\nu}}.
\end{equation*}
\end{lemma}

\begin{lemma}
\label{lem:annulus_for_any_kn_bracket_2_beta}
Let $\Delta_{t}^{+},\Delta_{t}^{-},\Delta_{x}^{+},\Delta_{x}^{-}$ be finite subsets of $R \times \Gamma \times \Gamma \times \mathbb{Z}$, and define
\begin{equation*}
\Theta_{t}^{+}(k,n) = \sum_{(r,w_{1},w_{2},v) \in \Delta_{t}^{+}} r\llangle w_{1}P_{n+v,k}w_{2} \rrangle_{\Sigma^{\prime}_{\nu}}, \quad
\Theta_{t}^{-}(k,n) = \sum_{(r,w_{1},w_{2},v) \in \Delta_{t}^{-}} r\llangle w_{1}P_{-n+v}\lambda^{k}w_{2} \rrangle_{\Sigma^{\prime}_{\nu}},
\end{equation*}
\begin{equation*}
\Theta_{x}^{+}(k,n) = \sum_{(r,w_{1},w_{2},v) \in \Delta_{x}^{+}} r\llangle w_{1}\lambda^{k}x_{n+v}w_{2} \rrangle_{\Sigma^{\prime}_{\nu}}, \quad
\Theta_{x}^{-}(k,n) = \sum_{(r,w_{1},w_{2},v) \in \Delta_{x}^{-}} r\llangle w_{1}x_{-n+v}\lambda^{k}w_{2} \rrangle_{\Sigma^{\prime}_{\nu}},
\end{equation*}
and 
\begin{equation*}
\Theta_{t,x}(k,n) = \Theta_{t}^{+}(k,n) + \Theta_{t}^{-}(k,n) + \Theta_{x}^{+}(k,n)+ \Theta_{x}^{-}(k,n).
\end{equation*}
If either \textup{(1)} $\Theta_{t,x}(0,n) = 0$ for all $n \in \mathbb{Z}$ or \textup{(2)} $\Theta_{t,x}(k,n_{0}) = \Theta_{t,x}(k,n_{0}+1) = 0$ for all $k \geq 0$ and a fixed $n_{0} \in \mathbb{Z}$, then $\Theta_{t,x}(k,n) = 0$ for any $k \geq 0$ and $n \in \mathbb{Z}$.
\end{lemma}

We will show that for arrow diagrams $D$ and $D'$ in ${\bf D}^{2}_{\beta}$ related  by an $S_{\beta}$-move, $\phi_{\beta}(D-D') = 0$. As the first step we consider special cases of such arrow diagrams $D = uw$ and $D'= W'w$, where $W'$ is obtained from $u\in \{t_{m,n},\lambda^{n}x_{m}\}$ by an $S_{\beta}$-move and $w\in \Gamma$. For such diagrams
\begin{equation}
\label{eqn:gDD'}
\phi_{\beta}(D-D') = \phi_{\beta}((u - W')w) = \llangle \llangle (u - A W'_{+} - A^{-1} W'_{-})w \rrangle_{\Gamma} \rrangle_{\Sigma_{\nu}^{\prime}},  
\end{equation}
where $W'_{+},W'_{-}$ are arrow diagrams obtained from $W'$ by smoothing its crossing according to positive and negative markers.

\begin{figure}[ht]
\centering
\includegraphics[scale=0.9]{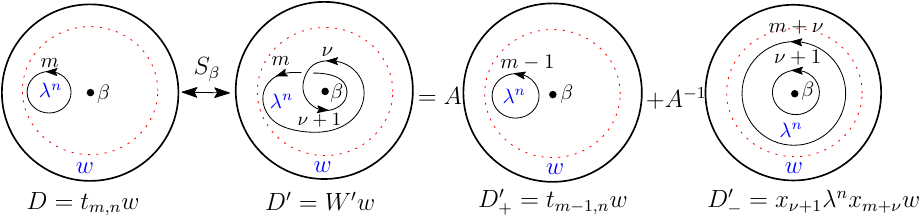}
\caption{Arrow diagrams $D$ and $D'$ in ${\bf D}^{2}_{\beta}$ related by an $S_{\beta}$-move}
\label{fig:V_beta_2_case1}
\end{figure}

For arrow diagrams $D$, $D'$ in Figure~\ref{fig:V_beta_2_case1}, we see that $u = t_{m,n}$, $W'_{+} = t_{m-1,n}$, and $W'_{-} = x_{\nu+1}\lambda^{n}x_{m+\nu}$. Hence, by \eqref{eqn:gDD'}, $\phi_{\beta}(D-D') = 0$ is equivalent to \eqref{eqn:Prop_h4_1} in the following lemma.

\begin{lemma}
\label{lem:V_beta_2_case1}
For $w\in \Gamma$, $m\in \mathbb{Z}$, and $n\geq 0$,
\begin{equation}
\label{eqn:Prop_h4_1}
\llangle P_{m,n}w - A P_{m-1,n}w - A^{-1}x_{\nu+1}\lambda^{n}x_{m+\nu}w \rrangle_{\Sigma^{\prime}_{\nu}} = 0.
\end{equation}
\end{lemma}

\begin{proof}
To prove \eqref{eqn:Prop_h4_1}, it suffices to show that for any fixed $m$,
\begin{equation}
\label{eqn:V_beta_2_case1}
\llangle P_{m,n}w \rrangle_{\Sigma^{\prime}_{\nu}} = A\llangle P_{m-1,n}w \rrangle_{\Sigma^{\prime}_{\nu}} + A^{-1}\llangle x_{\nu+1}\lambda^{n}x_{m+\nu}w \rrangle_{\Sigma^{\prime}_{\nu}}.
\end{equation}
We may assume that $w = \lambda^{n'}$ or $w = \lambda^{n'}x_{m'}(x_{\nu})^{k'}$ by using similar arguments as at the beginning of our proof for Lemma~\ref{lem:loc_prop_bracket_c}. 

When $w = \lambda^{n'}$, \eqref{eqn:V_beta_2_case1} becomes
\begin{equation}
\label{eqn:V_beta_2_case1a}
\llangle P_{m,n}\lambda^{n'} \rrangle_{\Sigma^{\prime}_{\nu}} = A\llangle P_{m-1,n}\lambda^{n'} \rrangle_{\Sigma^{\prime}_{\nu}} + A^{-1}\llangle x_{\nu+1}\lambda^{n}x_{m+\nu}\lambda^{n'} \rrangle_{\Sigma^{\prime}_{\nu}}.
\end{equation}
By Lemma~\ref{lem:annulus_for_any_kn_bracket_2_beta}, it suffices to check \eqref{eqn:V_beta_2_case1a} for $n' = 0$ and $m = 0$ or $m = 1$. Indeed, when $n' = 0$ and $m = 0$, by part k) and j) in the definition of $\llangle \cdot \rrangle_{\Sigma^{\prime}_{\nu}}$,
\begin{equation*}
\llangle P_{0,n} \rrangle_{\Sigma^{\prime}_{\nu}} = A\llangle P_{-1,n} \rrangle_{\Sigma^{\prime}_{\nu}} - A^{2} \llangle x_{\nu}\lambda^{n}x_{\nu} \rrangle_{\Sigma^{\prime}_{\nu}} = A\llangle P_{-1,n} \rrangle_{\Sigma^{\prime}_{\nu}}+ A^{-1}\llangle x_{\nu+1}\lambda^{n}x_{\nu} \rrangle_{\Sigma^{\prime}_{\nu}}.
\end{equation*}
Hence, \eqref{eqn:V_beta_2_case1a} holds in this case. When $n' = 0$ and $m = 1$, by part f) in the definition of $\llangle \cdot \rrangle_{\Sigma^{\prime}_{\nu}}$, we see that
\begin{equation*}
\llangle x_{\nu+1}\lambda^{n}x_{\nu+1} \rrangle_{\Sigma^{\prime}_{\nu}} 
= -A^{-1} \llangle \lambda P_{-1,n} \rrangle_{\Sigma^{\prime}_{\nu}} + 2 \llangle P_{-2,n} \rrangle_{\Sigma^{\prime}_{\nu}} + A^{-2} \llangle x_{\nu+2}\lambda^{n}x_{\nu} \rrangle_{\Sigma^{\prime}_{\nu}}.
\end{equation*}
Moreover, by part h) in the definition of $\llangle \cdot \rrangle_{\Sigma^{\prime}_{\nu}}$,
\begin{equation*}
A^{-2}\llangle x_{\nu+2}\lambda^{n}x_{\nu} \rrangle_{\Sigma^{\prime}_{\nu}} = A^{-1}\llangle x_{\nu+1}\lambda^{n+1}x_{\nu} \rrangle_{\Sigma^{\prime}_{\nu}} - \llangle x_{\nu}\lambda^{n}x_{\nu} \rrangle_{\Sigma^{\prime}_{\nu}}.
\end{equation*}
Since 
\begin{equation*}
A^{-1}\llangle x_{\nu+1}\lambda^{n+1}x_{\nu} \rrangle_{\Sigma^{\prime}_{\nu}} =
-A^{2}\llangle x_{\nu}\lambda^{n+1}x_{\nu} \rrangle_{\Sigma^{\prime}_{\nu}} = -A\llangle P_{-1,n+1} \rrangle_{\Sigma^{\prime}_{\nu}} + \llangle P_{0,n+1} \rrangle_{\Sigma^{\prime}_{\nu}}
\end{equation*}
by part j) and k) in the definition of $\llangle \cdot \rrangle_{\Sigma^{\prime}_{\nu}}$,
and 
\begin{equation*}
\llangle x_{\nu}\lambda^{n}x_{\nu} \rrangle_{\Sigma^{\prime}_{\nu}} = A^{-1}\llangle P_{-1,n} \rrangle_{\Sigma^{\prime}_{\nu}} - A^{-2}\llangle P_{0,n} \rrangle_{\Sigma^{\prime}_{\nu}}
\end{equation*}
by part k) in the definition of $\llangle \cdot \rrangle_{\Sigma^{\prime}_{\nu}}$. Finally, using \eqref{eqn:Pnk} and \eqref{eqn:Pnk_1},
\begin{eqnarray*}
\llangle x_{\nu+1}\lambda^{n}x_{\nu+1} \rrangle_{\Sigma^{\prime}_{\nu}} 
&=& -A^{-1} \llangle \lambda P_{-1,n} \rrangle_{\Sigma^{\prime}_{\nu}} + 2 \llangle P_{-2,n} \rrangle_{\Sigma^{\prime}_{\nu}} - A\llangle P_{-1,n+1} \rrangle_{\Sigma^{\prime}_{\nu}} \\
&+& \llangle P_{0,n+1} \rrangle_{\Sigma^{\prime}_{\nu}} - (A^{-1}\llangle P_{-1,n} \rrangle_{\Sigma^{\prime}_{\nu}} - A^{-2}\llangle P_{0,n} \rrangle_{\Sigma^{\prime}_{\nu}}) \\
&=& -A^{-2} \llangle P_{0,n} \rrangle_{\Sigma^{\prime}_{\nu}} - \llangle P_{-2,n} \rrangle_{\Sigma^{\prime}_{\nu}} + 2 \llangle P_{-2,n} \rrangle_{\Sigma^{\prime}_{\nu}} - A^{2}\llangle P_{0,n} \rrangle_{\Sigma^{\prime}_{\nu}} - \llangle P_{-2,n} \rrangle_{\Sigma^{\prime}_{\nu}} \\
&+& A\llangle P_{1,n} \rrangle_{\Sigma^{\prime}_{\nu}} + A^{-1}\llangle P_{-1,n} \rrangle_{\Sigma^{\prime}_{\nu}} 
- A^{-1}\llangle P_{-1,n} \rrangle_{\Sigma^{\prime}_{\nu}} + A^{-2}\llangle P_{0,n} \rrangle_{\Sigma^{\prime}_{\nu}} \\
&=& A\llangle P_{1,n} \rrangle_{\Sigma^{\prime}_{\nu}} - A^{2}\llangle P_{0,n} \rrangle_{\Sigma^{\prime}_{\nu}},
\end{eqnarray*}
it follows that \eqref{eqn:V_beta_2_case1a} holds in this case.

When $w = \lambda^{n'}x_{m'}(x_{\nu})^{k'}$, then \eqref{eqn:V_beta_2_case1} becomes
\begin{equation}
\label{eqn:V_beta_2_case1b}
\llangle P_{m,n}\lambda^{n'}x_{m'}(x_{\nu})^{k'} \rrangle_{\Sigma^{\prime}_{\nu}} = A\llangle P_{m-1,n}\lambda^{n'}x_{m'}(x_{\nu})^{k'} \rrangle_{\Sigma^{\prime}_{\nu}} + A^{-1}\llangle x_{\nu+1}\lambda^{n}x_{m+\nu}\lambda^{n'}x_{m'}(x_{\nu})^{k'} \rrangle_{\Sigma^{\prime}_{\nu}}.
\end{equation}
By Lemma~\ref{lem:annulus_for_any_kn_bracket_2_beta}, it suffices to check \eqref{eqn:V_beta_2_case1b} for $n' = 0$ and $(m,m') = (0,\nu),(1,\nu),(0,\nu+1),(1,\nu+1)$.

For $n' = 0$ and $(m,m') = (0,\nu)$, \eqref{eqn:V_beta_2_case1b} becomes
\begin{equation}
\label{eqn:V_beta_2_case1b_1}
\llangle P_{0,n}x_{\nu}(x_{\nu})^{k'} \rrangle_{\Sigma^{\prime}_{\nu}} = A\llangle P_{-1,n}x_{\nu}(x_{\nu})^{k'} \rrangle_{\Sigma^{\prime}_{\nu}} + A^{-1}\llangle x_{\nu+1}\lambda^{n}x_{\nu}x_{\nu}(x_{\nu})^{k'} \rrangle_{\Sigma^{\prime}_{\nu}}.
\end{equation}
By part j) and k) in the definition of $\llangle \cdot \rrangle_{\Sigma^{\prime}_{\nu}}$,
\begin{equation*}
\llangle x_{\nu+1}\lambda^{n}x_{\nu}(x_{\nu})^{k'+1} \rrangle_{\Sigma^{\prime}_{\nu}}
= -A^{3}\llangle x_{\nu}\lambda^{n}x_{\nu}(x_{\nu})^{k'+1} \rrangle_{\Sigma^{\prime}_{\nu}} 
= -A^{2}\llangle P_{-1,n}(x_{\nu})^{k'+1} \rrangle_{\Sigma^{\prime}_{\nu}} + A\llangle P_{0,n}(x_{\nu})^{k'+1} \rrangle_{\Sigma^{\prime}_{\nu}},
\end{equation*}
which gives \eqref{eqn:V_beta_2_case1b_1}.

For $n' = 0$ and $(m,m') = (1,\nu)$, \eqref{eqn:V_beta_2_case1b} becomes
\begin{equation}
\label{eqn:V_beta_2_case1b_2}
\llangle P_{1,n}x_{\nu}(x_{\nu})^{k'} \rrangle_{\Sigma^{\prime}_{\nu}} = A\llangle P_{0,n}x_{\nu}(x_{\nu})^{k'} \rrangle_{\Sigma^{\prime}_{\nu}} + A^{-1}\llangle x_{\nu+1}\lambda^{n}x_{\nu+1}x_{\nu}(x_{\nu})^{k'} \rrangle_{\Sigma^{\prime}_{\nu}}.
\end{equation}
By part f) in the definition of $\llangle \cdot \rrangle_{\Sigma^{\prime}_{\nu}}$,
\begin{eqnarray*}
\llangle x_{\nu+1}\lambda^{n}x_{\nu+1}(x_{\nu})^{k'+1} \rrangle_{\Sigma^{\prime}_{\nu}}
&=& -A^{-1} \llangle \lambda P_{-1,n}(x_{\nu})^{k'+1} \rrangle_{\Sigma^{\prime}_{\nu}} + 2 \llangle P_{-2,n}(x_{\nu})^{k'+1} \rrangle_{\Sigma^{\prime}_{\nu}} \\
&+& A^{-2} \llangle x_{\nu+2}\lambda^{n}(x_{\nu})^{k'+2} \rrangle_{\Sigma^{\prime}_{\nu}}.
\end{eqnarray*}
Moreover, by part h) in the definition of $\llangle \cdot \rrangle_{\Sigma^{\prime}_{\nu}}$,
\begin{equation*}
A^{-2} \llangle x_{\nu+2}\lambda^{n}(x_{\nu})^{k'+2} \rrangle_{\Sigma^{\prime}_{\nu}}
= A^{-1} \llangle x_{\nu+1}\lambda^{n+1}(x_{\nu})^{k'+2} \rrangle_{\Sigma^{\prime}_{\nu}} - \llangle x_{\nu}\lambda^{n}(x_{\nu})^{k'+2} \rrangle_{\Sigma^{\prime}_{\nu}}.
\end{equation*}
Since 
\begin{equation*}
A^{-1} \llangle x_{\nu+1}\lambda^{n+1}(x_{\nu})^{k'+2} \rrangle_{\Sigma^{\prime}_{\nu}} = \llangle P_{0,n+1}(x_{\nu})^{k'+1} \rrangle_{\Sigma^{\prime}_{\nu}} - A\llangle P_{-1,n+1}(x_{\nu})^{k'+1} \rrangle_{\Sigma^{\prime}_{\nu}} 
\end{equation*}
by \eqref{eqn:V_beta_2_case1b_1} and
\begin{equation*}
- \llangle x_{\nu}\lambda^{n}(x_{\nu})^{k'+2} \rrangle_{\Sigma^{\prime}_{\nu}} = -A^{-1} \llangle P_{-1,n}(x_{\nu})^{k'+1} \rrangle_{\Sigma^{\prime}_{\nu}} + A^{-2} \llangle P_{0,n}(x_{\nu})^{k'+1} \rrangle_{\Sigma^{\prime}_{\nu}} 
\end{equation*}
by part k) in the definition of $\llangle \cdot \rrangle_{\Sigma^{\prime}_{\nu}}$, it follows that
\begin{eqnarray*}
& &\llangle x_{\nu+1}\lambda^{n}x_{\nu+1}(x_{\nu})^{k'+1} \rrangle_{\Sigma^{\prime}_{\nu}} \\
&=& -A^{-1} \llangle \lambda P_{-1,n}(x_{\nu})^{k'+1} \rrangle_{\Sigma^{\prime}_{\nu}} + 2 \llangle P_{-2,n}(x_{\nu})^{k'+1} \rrangle_{\Sigma^{\prime}_{\nu}} + \llangle P_{0,n+1}(x_{\nu})^{k'+1} \rrangle_{\Sigma^{\prime}_{\nu}} \\
&-& A\llangle P_{-1,n+1}(x_{\nu})^{k'+1} \rrangle_{\Sigma^{\prime}_{\nu}} -A^{-1} \llangle P_{-1,n}(x_{\nu})^{k'+1} \rrangle_{\Sigma^{\prime}_{\nu}} + A^{-2} \llangle P_{0,n}(x_{\nu})^{k'+1} \rrangle_{\Sigma^{\prime}_{\nu}}.
\end{eqnarray*}
Applying \eqref{eqn:Pnk} and \eqref{eqn:Pnk_1},
\begin{equation*}
-A^{-1} \llangle \lambda P_{-1,n}(x_{\nu})^{k'+1} \rrangle_{\Sigma^{\prime}_{\nu}} = -A^{-2} \llangle P_{0,n}(x_{\nu})^{k'+1} \rrangle_{\Sigma^{\prime}_{\nu}} - \llangle P_{-2,n}(x_{\nu})^{k'+1} \rrangle_{\Sigma^{\prime}_{\nu}},
\end{equation*}
\begin{equation*}
\llangle P_{0,n+1}(x_{\nu})^{k'+1} \rrangle_{\Sigma^{\prime}_{\nu}} = A \llangle P_{1,n}(x_{\nu})^{k'+1} \rrangle_{\Sigma^{\prime}_{\nu}} + A^{-1} \llangle P_{-1,n}(x_{\nu})^{k'+1} \rrangle_{\Sigma^{\prime}_{\nu}},
\end{equation*}
and
\begin{equation*}
-A\llangle P_{-1,n+1}(x_{\nu})^{k'+1} \rrangle_{\Sigma^{\prime}_{\nu}} = - A^{2} \llangle P_{0,n}(x_{\nu})^{k'+1} \rrangle_{\Sigma^{\prime}_{\nu}} - \llangle P_{-2,n}(x_{\nu})^{k'+1} \rrangle_{\Sigma^{\prime}_{\nu}}.
\end{equation*}
Therefore, after cancelling alike terms,
\begin{equation*}
\llangle x_{\nu+1}\lambda^{n}x_{\nu+1}(x_{\nu})^{k'+1} \rrangle_{\Sigma^{\prime}_{\nu}} = -A^{2} \llangle P_{0,n}(x_{\nu})^{k'+1} \rrangle_{\Sigma^{\prime}_{\nu}} + A \llangle P_{1,n}(x_{\nu})^{k'+1} \rrangle_{\Sigma^{\prime}_{\nu}}.
\end{equation*}
This finishes our proof for \eqref{eqn:V_beta_2_case1b_2}.

\begin{figure}[ht]
\centering
\includegraphics[scale=0.8]{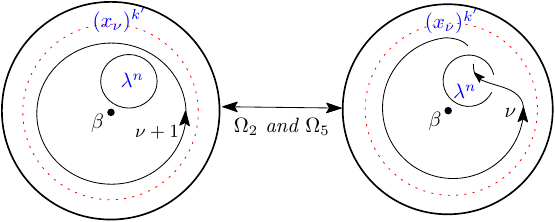}
\caption{$\Omega_{2}$ and $\Omega_{5}$ applied on $D = t_{0,n}x_{\nu+1}(x_{\nu})^{k'}$}
\label{fig:V_beta_2_case1b_31}
\end{figure}

For $n' = 0$ and $(m,m') = (0,\nu+1)$, \eqref{eqn:V_beta_2_case1b} becomes
\begin{equation}
\label{eqn:V_beta_2_case1b_3}
\llangle P_{0,n}x_{\nu+1}(x_{\nu})^{k'} \rrangle_{\Sigma^{\prime}_{\nu}} = A\llangle P_{-1,n}x_{\nu+1}(x_{\nu})^{k'} \rrangle_{\Sigma^{\prime}_{\nu}} + A^{-1}\llangle x_{\nu+1}\lambda^{n}x_{\nu}x_{\nu+1}(x_{\nu})^{k'} \rrangle_{\Sigma^{\prime}_{\nu}}.
\end{equation}
We notice that applying $\Omega_{2}$ and  $\Omega_{5}$ on $D = t_{0,n}x_{\nu+1}(x_{\nu})^{k'}$ yields the diagram $D''$ as shown on the right of Figure~\ref{fig:V_beta_2_case1b_31}. Since by Lemma~\ref{lem:lemma_for_h4},
\begin{eqnarray*}
\llangle P_{0,n}x_{\nu+1}(x_{\nu})^{k'} \rrangle_{\Sigma^{\prime}_{\nu}} &=& \phi_{\beta}(D)=\phi_{\beta}(D'') \\
&=& A^{2} \llangle x_{\nu} \lambda^{n} P_{-1} (x_{\nu})^{k'} \rrangle_{\Sigma^{\prime}_{\nu}} + 2 \llangle x_{\nu-1} \lambda^{n} (x_{\nu})^{k'} \rrangle_{\Sigma^{\prime}_{\nu}} + A^{-2} \llangle P_{1,n} (x_{\nu})^{k'+1} \rrangle_{\Sigma^{\prime}_{\nu}}.
\end{eqnarray*}
Since $P_{-1} = -A^{-3}\lambda$, using part i) in the definition of $\llangle \cdot \rrangle_{\Sigma^{\prime}_{\nu}}$, it follows that
\begin{eqnarray*}
\llangle P_{0,n}x_{\nu+1}(x_{\nu})^{k'} \rrangle_{\Sigma^{\prime}_{\nu}} 
&=& -A^{-1} \llangle x_{\nu} \lambda^{n+1} (x_{\nu})^{k'} \rrangle_{\Sigma^{\prime}_{\nu}} + 2A^{-1} \llangle x_{\nu} \lambda^{n+1} (x_{\nu})^{k'} \rrangle_{\Sigma^{\prime}_{\nu}} \\
&-& 2A^{-2} \llangle x_{\nu+1} \lambda^{n} (x_{\nu})^{k'} \rrangle_{\Sigma^{\prime}_{\nu}} + A^{-2} \llangle P_{1,n} (x_{\nu})^{k'+1} \rrangle_{\Sigma^{\prime}_{\nu}}.
\end{eqnarray*}
Therefore, by part j) in the definition $\llangle \cdot \rrangle_{\Sigma^{\prime}_{\nu}}$,
\begin{equation}
\label{eqn:V_beta_2_case1b_30}
\llangle P_{0,n}x_{\nu+1}(x_{\nu})^{k'} \rrangle_{\Sigma^{\prime}_{\nu}} = A^{-1} \llangle x_{\nu} \lambda^{n+1} (x_{\nu})^{k'} \rrangle_{\Sigma^{\prime}_{\nu}} + 2A \llangle x_{\nu} \lambda^{n} (x_{\nu})^{k'} \rrangle_{\Sigma^{\prime}_{\nu}} + A^{-2} \llangle P_{1,n} (x_{\nu})^{k'+1} \rrangle_{\Sigma^{\prime}_{\nu}}.
\end{equation}

\begin{figure}[ht]
\centering
\includegraphics[scale=0.8]{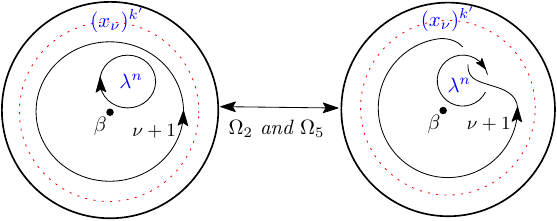}
\caption{$\Omega_{2}$ and $\Omega_{5}$ applied on $D = t_{-1,n}x_{\nu+1}(x_{\nu})^{k'}$}
\label{fig:V_beta_2_case1b_32}
\end{figure}

Applying $\Omega_{2}$ and $\Omega_{5}$ on arrow diagram $D = t_{-1,n}x_{\nu+1}(x_{\nu})^{k'}$  yields the arrow diagram $D''$ on the right of Figure~\ref{fig:V_beta_2_case1b_32}. Since by Lemma~\ref{lem:lemma_for_h4},
\begin{eqnarray*}
\llangle P_{-1,n}x_{\nu+1}(x_{\nu})^{k'} \rrangle_{\Sigma^{\prime}_{\nu}} &=& \phi_{\beta}(D) = \phi_{\beta}(D'') \\
&=& A^{2} \llangle x_{\nu+1} \lambda^{n} P_{-1} (x_{\nu})^{k'} \rrangle_{\Sigma^{\prime}_{\nu}} + 2 \llangle x_{\nu} \lambda^{n} (x_{\nu})^{k'} \rrangle_{\Sigma^{\prime}_{\nu}} + A^{-2} \llangle P_{0,n} (x_{\nu})^{k'+1} \rrangle_{\Sigma^{\prime}_{\nu}}.
\end{eqnarray*}
Since $P_{-1} = -A^{-3}\lambda$, using part j) in the definition of $\llangle \cdot \rrangle_{\Sigma^{\prime}_{\nu}}$, after multiplying by $A$ it follows that
\begin{equation}
\label{eqn:V_beta_2_case1b_31}
A\llangle P_{-1,n}x_{\nu+1}(x_{\nu})^{k'} \rrangle_{\Sigma^{\prime}_{\nu}} = A^{3} \llangle x_{\nu} \lambda^{n+1} (x_{\nu})^{k'} \rrangle_{\Sigma^{\prime}_{\nu}} + 2A \llangle x_{\nu} \lambda^{n} (x_{\nu})^{k'} \rrangle_{\Sigma^{\prime}_{\nu}} + A^{-1} \llangle P_{0,n} (x_{\nu})^{k'+1} \rrangle_{\Sigma^{\prime}_{\nu}}.
\end{equation}
Furthermore, by part f) in the definition of $\llangle \cdot \rrangle_{\Sigma^{\prime}_{\nu}}$ and then after using part j) in the definition of $\llangle \cdot \rrangle_{\Sigma^{\prime}_{\nu}}$,
\begin{eqnarray}
\label{eqn:V_beta_2_case1b_32}
& &A^{-1}\llangle x_{\nu+1}\lambda^{n}x_{\nu}x_{\nu+1}(x_{\nu})^{k'} \rrangle_{\Sigma^{\prime}_{\nu}} \notag\\
&=& -A^{-2}\llangle x_{\nu+1}\lambda^{n+1}P_{0}(x_{\nu})^{k'} \rrangle_{\Sigma^{\prime}_{\nu}} + 2A^{-1} \llangle x_{\nu+1}\lambda^{n}P_{-1}(x_{\nu})^{k'} \rrangle_{\Sigma^{\prime}_{\nu}} 
+ A^{-3} \llangle x_{\nu+1}\lambda^{n}x_{\nu+1}(x_{\nu})^{k'+1} \rrangle_{\Sigma^{\prime}_{\nu}} \notag\\ 
&=& -A^{3}\llangle x_{\nu}\lambda^{n+1}(x_{\nu})^{k'} \rrangle_{\Sigma^{\prime}_{\nu}} + A^{-1} \llangle x_{\nu}\lambda^{n+1}(x_{\nu})^{k'} \rrangle_{\Sigma^{\prime}_{\nu}} 
+ A^{-3} \llangle x_{\nu+1}\lambda^{n}x_{\nu+1}(x_{\nu})^{k'+1} \rrangle_{\Sigma^{\prime}_{\nu}}.
\end{eqnarray}
Therefore, \eqref{eqn:V_beta_2_case1b_3} follows from \eqref{eqn:V_beta_2_case1b_30}--\eqref{eqn:V_beta_2_case1b_32} and \eqref{eqn:V_beta_2_case1b_2}.

For $n' = 0$ and $(m,m') = (1,\nu+1)$, \eqref{eqn:V_beta_2_case1b} becomes
\begin{equation}
\label{eqn:V_beta_2_case1b_4}
\llangle P_{1,n}x_{\nu+1}(x_{\nu})^{k'} \rrangle_{\Sigma^{\prime}_{\nu}} = A\llangle P_{0,n}x_{\nu+1}(x_{\nu})^{k'} \rrangle_{\Sigma^{\prime}_{\nu}} + A^{-1}\llangle x_{\nu+1}\lambda^{n}x_{\nu+1}x_{\nu+1}(x_{\nu})^{k'} \rrangle_{\Sigma^{\prime}_{\nu}}.
\end{equation}

\begin{figure}[ht]
\centering
\includegraphics[scale=0.8]{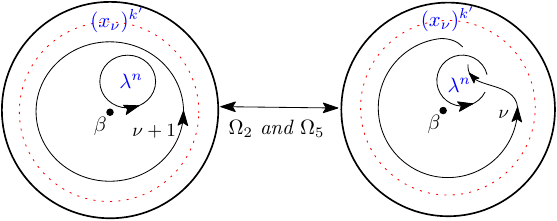}
\caption{$\Omega_{2}$ and $\Omega_{5}$ applied on $D = t_{1,n}x_{\nu+1}(x_{\nu})^{k'}$}
\label{fig:V_beta_2_case1b_41}
\end{figure}

Let $D$ and $D''$ be arrow diagrams in Figure~\ref{fig:V_beta_2_case1b_41}. Since $D$ and $D''$ differ by $\Omega_{2}$ and $\Omega_{5}$-moves, by Lemma~\ref{lem:lemma_for_h4},
\begin{eqnarray*}
\llangle P_{1,n}x_{\nu+1}(x_{\nu})^{k'} \rrangle_{\Sigma^{\prime}_{\nu}} &=& \phi_{\beta}(D) = \phi_{\beta}(D'') \\
&=& A^{2} \llangle x_{\nu-1} \lambda^{n} P_{-1} (x_{\nu})^{k'} \rrangle_{\Sigma^{\prime}_{\nu}} + 2 \llangle x_{\nu-2} \lambda^{n} (x_{\nu})^{k'} \rrangle_{\Sigma^{\prime}_{\nu}} + A^{-2} \llangle P_{2,n} (x_{\nu})^{k'+1} \rrangle_{\Sigma^{\prime}_{\nu}}.
\end{eqnarray*}
Since $P_{-1} = -A^{-3}\lambda$, using part i) in the definition of $\llangle \cdot \rrangle_{\Sigma^{\prime}_{\nu}}$, it follows that
\begin{eqnarray*}
\llangle P_{1,n}x_{\nu+1}(x_{\nu})^{k'} \rrangle_{\Sigma^{\prime}_{\nu}} 
&=& -A^{-1} \llangle x_{\nu-1} \lambda^{n+1} (x_{\nu})^{k'} \rrangle_{\Sigma^{\prime}_{\nu}} + 2A^{-1} \llangle x_{\nu-1} \lambda^{n+1} (x_{\nu})^{k'} \rrangle_{\Sigma^{\prime}_{\nu}} \\
&-& 2A^{-2} \llangle x_{\nu} \lambda^{n} (x_{\nu})^{k'} \rrangle_{\Sigma^{\prime}_{\nu}} + A^{-2} \llangle P_{2,n} (x_{\nu})^{k'+1} \rrangle_{\Sigma^{\prime}_{\nu}} \\ 
&=& A^{-2} \llangle x_{\nu} \lambda^{n+2} (x_{\nu})^{k'} \rrangle_{\Sigma^{\prime}_{\nu}} - A^{-3} \llangle x_{\nu+1} \lambda^{n+1} (x_{\nu})^{k'} \rrangle_{\Sigma^{\prime}_{\nu}}  \\
&-& 2A^{-2} \llangle x_{\nu} \lambda^{n} (x_{\nu})^{k'} \rrangle_{\Sigma^{\prime}_{\nu}} + A^{-2} \llangle P_{2,n} (x_{\nu})^{k'+1} \rrangle_{\Sigma^{\prime}_{\nu}}. 
\end{eqnarray*}
Therefore, by part j) in the definition $\llangle \cdot \rrangle_{\Sigma^{\prime}_{\nu}}$,
\begin{eqnarray}
\label{eqn:V_beta_2_case1b_40}
\llangle P_{1,n}x_{\nu+1}(x_{\nu})^{k'} \rrangle_{\Sigma^{\prime}_{\nu}} &=& A^{-2} \llangle x_{\nu} \lambda^{n+2} (x_{\nu})^{k'} \rrangle_{\Sigma^{\prime}_{\nu}} + \llangle x_{\nu} \lambda^{n+1} (x_{\nu})^{k'} \rrangle_{\Sigma^{\prime}_{\nu}} \notag\\
&-& 2A^{-2} \llangle x_{\nu} \lambda^{n} (x_{\nu})^{k'} \rrangle_{\Sigma^{\prime}_{\nu}} + A^{-2} \llangle P_{2,n} (x_{\nu})^{k'+1} \rrangle_{\Sigma^{\prime}_{\nu}}.
\end{eqnarray}

Furthermore, by part f) in the definition of $\llangle \cdot \rrangle_{\Sigma^{\prime}_{\nu}}$, the second term of the right hand side of \eqref{eqn:V_beta_2_case1b_4} becomes
\begin{eqnarray}
\label{eqn:V_beta_2_case1b_41}
A^{-1}\llangle x_{\nu+1}\lambda^{n}x_{\nu+1}x_{\nu+1}(x_{\nu})^{k'} \rrangle_{\Sigma^{\prime}_{\nu}}
&=& -A^{-2}\llangle x_{\nu+1}\lambda^{n+1}P_{-1}(x_{\nu})^{k'} \rrangle_{\Sigma^{\prime}_{\nu}} + 2A^{-1} \llangle x_{\nu+1}\lambda^{n}P_{-2}(x_{\nu})^{k'} \rrangle_{\Sigma^{\prime}_{\nu}} \notag\\
&+& A^{-3} \llangle x_{\nu+1}\lambda^{n}x_{\nu+2}(x_{\nu})^{k'+1} \rrangle_{\Sigma^{\prime}_{\nu}}. 
\end{eqnarray}
Since $P_{-1} = -A^{-3}\lambda$ and $P_{-2} = - Q_{-1} + A^{-4}Q_{-3} = 1 + A^{-4} - A^{-4}\lambda^{2}$, using part j) in the definition $\llangle \cdot \rrangle_{\Sigma^{\prime}_{\nu}}$,
\begin{equation*}
-A^{-2}\llangle x_{\nu+1}\lambda^{n+1}P_{-1}(x_{\nu})^{k'} \rrangle_{\Sigma^{\prime}_{\nu}} = -A^{-2}\llangle x_{\nu}\lambda^{n+2}(x_{\nu})^{k'} \rrangle_{\Sigma^{\prime}_{\nu}}  
\end{equation*}
and
\begin{eqnarray*}
2A^{-1} \llangle x_{\nu+1}\lambda^{n}P_{-2}(x_{\nu})^{k'} \rrangle_{\Sigma^{\prime}_{\nu}} &=& - 2A^{2} \llangle x_{\nu}\lambda^{n}(x_{\nu})^{k'} \rrangle_{\Sigma^{\prime}_{\nu}} - 2A^{-2} \llangle x_{\nu}\lambda^{n}(x_{\nu})^{k'} \rrangle_{\Sigma^{\prime}_{\nu}} \notag\\
&+& 2A^{-2} \llangle x_{\nu}\lambda^{n+2}(x_{\nu})^{k'} \rrangle_{\Sigma^{\prime}_{\nu}}.
\end{eqnarray*}
For the third term on the right hand side of \eqref{eqn:V_beta_2_case1b_41}, using part d) and j) in the definition of $\llangle \cdot \rrangle_{\Sigma^{\prime}_{\nu}}$,
\begin{eqnarray*}
A^{-3} \llangle x_{\nu+1}\lambda^{n}x_{\nu+2}(x_{\nu})^{k'+1} \rrangle_{\Sigma^{\prime}_{\nu}} 
&=& A^{-4} \llangle x_{\nu+1}\lambda^{n+1}x_{\nu+1}(x_{\nu})^{k'+1} \rrangle_{\Sigma^{\prime}_{\nu}} - A^{-5} \llangle x_{\nu+1}\lambda^{n}x_{\nu}(x_{\nu})^{k'+1} \rrangle_{\Sigma^{\prime}_{\nu}} \notag\\ 
&=& A^{-4} \llangle x_{\nu+1}\lambda^{n+1}x_{\nu+1}(x_{\nu})^{k'+1} \rrangle_{\Sigma^{\prime}_{\nu}} + A^{-2} \llangle x_{\nu}\lambda^{n}x_{\nu}(x_{\nu})^{k'+1} \rrangle_{\Sigma^{\prime}_{\nu}}.
\end{eqnarray*}
Therefore, by the above equations and \eqref{eqn:V_beta_2_case1b_2}, 
\begin{eqnarray}
\label{eqn:V_beta_2_case1b_45}
& &A^{-1}\llangle x_{\nu+1}\lambda^{n}x_{\nu+1}x_{\nu+1}(x_{\nu})^{k'} \rrangle_{\Sigma^{\prime}_{\nu}} \notag\\
&=& - 2A^{2} \llangle x_{\nu}\lambda^{n}(x_{\nu})^{k'} \rrangle_{\Sigma^{\prime}_{\nu}}
- 2A^{-2} \llangle x_{\nu}\lambda^{n}(x_{\nu})^{k'} \rrangle_{\Sigma^{\prime}_{\nu}}
+ A^{-2} \llangle x_{\nu}\lambda^{n+2}(x_{\nu})^{k'} \rrangle_{\Sigma^{\prime}_{\nu}} \notag\\
&+& A^{-3} \llangle P_{1,n+1}(x_{\nu})^{k'+1} \rrangle_{\Sigma^{\prime}_{\nu}} - A^{-2} \llangle P_{0,n+1}(x_{\nu})^{k'+1} \rrangle_{\Sigma^{\prime}_{\nu}} + A^{-2} \llangle x_{\nu}\lambda^{n}(x_{\nu})^{k'+2} \rrangle_{\Sigma^{\prime}_{\nu}}.
\end{eqnarray}
Hence, by \eqref{eqn:V_beta_2_case1b_30}, \eqref{eqn:V_beta_2_case1b_45}, and part k) in the definition $\llangle \cdot \rrangle_{\Sigma^{\prime}_{\nu}}$, the right hand side of \eqref{eqn:V_beta_2_case1b_4} becomes
\begin{eqnarray*}
& &A\llangle P_{0,n}x_{\nu+1}(x_{\nu})^{k'} \rrangle_{\Sigma^{\prime}_{\nu}} + A^{-1}\llangle x_{\nu+1}\lambda^{n}x_{\nu+1}x_{\nu+1}(x_{\nu})^{k'} \rrangle_{\Sigma^{\prime}_{\nu}} \\
&=& \llangle x_{\nu} \lambda^{n+1} (x_{\nu})^{k'} \rrangle_{\Sigma^{\prime}_{\nu}} + 2A^{2} \llangle x_{\nu} \lambda^{n} (x_{\nu})^{k'} \rrangle_{\Sigma^{\prime}_{\nu}} + A^{-1} \llangle P_{1,n} (x_{\nu})^{k'+1} \rrangle_{\Sigma^{\prime}_{\nu}} \\
&-& 2A^{2} \llangle x_{\nu}\lambda^{n}(x_{\nu})^{k'} \rrangle_{\Sigma^{\prime}_{\nu}}
- 2A^{-2} \llangle x_{\nu}\lambda^{n}(x_{\nu})^{k'} \rrangle_{\Sigma^{\prime}_{\nu}}
+ A^{-2} \llangle x_{\nu}\lambda^{n+2}(x_{\nu})^{k'} \rrangle_{\Sigma^{\prime}_{\nu}} \notag\\
&+& A^{-3} \llangle P_{1,n+1}(x_{\nu})^{k'+1} \rrangle_{\Sigma^{\prime}_{\nu}} - A^{-2} \llangle P_{0,n+1}(x_{\nu})^{k'+1} \rrangle_{\Sigma^{\prime}_{\nu}} + A^{-2} \llangle x_{\nu}\lambda^{n}(x_{\nu})^{k'+2} \rrangle_{\Sigma^{\prime}_{\nu}} \\
&=& A^{-2} \llangle x_{\nu}\lambda^{n+2}(x_{\nu})^{k'} \rrangle_{\Sigma^{\prime}_{\nu}} + \llangle x_{\nu} \lambda^{n+1} (x_{\nu})^{k'} \rrangle_{\Sigma^{\prime}_{\nu}} - 2A^{-2} \llangle x_{\nu}\lambda^{n}(x_{\nu})^{k'} \rrangle_{\Sigma^{\prime}_{\nu}} \\
&+& A^{-1} \llangle P_{1,n} (x_{\nu})^{k'+1} \rrangle_{\Sigma^{\prime}_{\nu}} + A^{-3} \llangle P_{1,n+1}(x_{\nu})^{k'+1} \rrangle_{\Sigma^{\prime}_{\nu}} - A^{-2} \llangle P_{0,n+1}(x_{\nu})^{k'+1} \rrangle_{\Sigma^{\prime}_{\nu}} \\
&+& A^{-3} \llangle P_{-1,n}(x_{\nu})^{k'+1} \rrangle_{\Sigma^{\prime}_{\nu}} - A^{-4} \llangle P_{0,n}(x_{\nu})^{k'+1} \rrangle_{\Sigma^{\prime}_{\nu}}.
\end{eqnarray*}
Since by \eqref{eqn:Pnk}
\begin{equation*}
A^{-3} \llangle P_{1,n+1}(x_{\nu})^{k'+1} \rrangle_{\Sigma^{\prime}_{\nu}} = A^{-2} \llangle P_{2,n}(x_{\nu})^{k'+1} \rrangle_{\Sigma^{\prime}_{\nu}} + A^{-4} \llangle P_{0,n}(x_{\nu})^{k'+1} \rrangle_{\Sigma^{\prime}_{\nu}}
\end{equation*}
and
\begin{equation*}
- A^{-2} \llangle P_{0,n+1}(x_{\nu})^{k'+1} \rrangle_{\Sigma^{\prime}_{\nu}} = - A^{-1} \llangle P_{1,n}(x_{\nu})^{k'+1} \rrangle_{\Sigma^{\prime}_{\nu}} - A^{-3} \llangle P_{-1,n}(x_{\nu})^{k'+1} \rrangle_{\Sigma^{\prime}_{\nu}},
\end{equation*}
it follows that the right hand side of \eqref{eqn:V_beta_2_case1b_4} is same as the right hand side of \eqref{eqn:V_beta_2_case1b_40}. This completes the proof of the case $n' = 0$ and $(m,m') = (1,\nu+1)$.
\end{proof}

The following lemma gives us an important relation that will be used in the proof of Lemma~\ref{lem:V_beta_2_case2} to make it significantly shorter.

\begin{lemma}
\label{lem:V_beta_2_case2b_lem}
For any $n,k \geq 0$,
\begin{equation}
\label{eqn:V_beta_2_case2b_lem}
\llangle x_{\nu+1}\lambda^{n}x_{\nu+1}(x_{\nu})^{k} \rrangle_{\Sigma^{\prime}_{\nu}} = -A^{3} \llangle x_{\nu}\lambda^{n}x_{\nu+1}(x_{\nu})^{k} \rrangle_{\Sigma^{\prime}_{\nu}}.
\end{equation}
\end{lemma}

\begin{proof}
By part f) in the definition of $\llangle \cdot \rrangle_{\Sigma^{\prime}_{\nu}}$,
\begin{equation*}
\llangle x_{\nu+1}\lambda^{n}x_{\nu+1}(x_{\nu})^{k} \rrangle_{\Sigma^{\prime}_{\nu}} = -A^{-1} \llangle \lambda P_{-1,n}(x_{\nu})^{k} \rrangle_{\Sigma^{\prime}_{\nu}} + 2 \llangle P_{-2,n}(x_{\nu})^{k} \rrangle_{\Sigma^{\prime}_{\nu}} + A^{-2} \llangle x_{\nu+2}\lambda^{n}(x_{\nu})^{k+1} \rrangle_{\Sigma^{\prime}_{\nu}},
\end{equation*}
and then using \eqref{eqn:Pnk_1} and part h) in the definition of $\llangle \cdot \rrangle_{\Sigma^{\prime}_{\nu}}$,
\begin{eqnarray*}
\llangle x_{\nu+1}\lambda^{n}x_{\nu+1}(x_{\nu})^{k} \rrangle_{\Sigma^{\prime}_{\nu}} 
&=& -A^{-2} \llangle P_{0,n}(x_{\nu})^{k} \rrangle_{\Sigma^{\prime}_{\nu}} + \llangle P_{-2,n}(x_{\nu})^{k} \rrangle_{\Sigma^{\prime}_{\nu}} \\ 
&-& \llangle x_{\nu}\lambda^{n}(x_{\nu})^{k+1} \rrangle_{\Sigma^{\prime}_{\nu}} + A^{-1} \llangle x_{\nu+1}\lambda^{n+1}(x_{\nu})^{k+1} \rrangle_{\Sigma^{\prime}_{\nu}}.
\end{eqnarray*}
Since by part k) and j) in the definition $\llangle \cdot \rrangle_{\Sigma^{\prime}_{\nu}}$, 
\begin{equation*}
-\llangle x_{\nu}\lambda^{n}(x_{\nu})^{k+1} \rrangle_{\Sigma^{\prime}_{\nu}} = -A^{-1} \llangle P_{-1,n}(x_{\nu})^{k} \rrangle_{\Sigma^{\prime}_{\nu}} + A^{-2} \llangle P_{0,n}(x_{\nu})^{k} \rrangle_{\Sigma^{\prime}_{\nu}}
\end{equation*}
and
\begin{equation*}
A^{-1}\llangle x_{\nu+1}\lambda^{n+1}(x_{\nu})^{k+1} \rrangle_{\Sigma^{\prime}_{\nu}} 
= -A^{2}\llangle x_{\nu}\lambda^{n+1}(x_{\nu})^{k+1} \rrangle_{\Sigma^{\prime}_{\nu}}
= -A \llangle P_{-1,n+1}(x_{\nu})^{k} \rrangle_{\Sigma^{\prime}_{\nu}} + \llangle P_{0,n+1}(x_{\nu})^{k} \rrangle_{\Sigma^{\prime}_{\nu}}.
\end{equation*}
Furthermore, using \eqref{eqn:Pnk},
\begin{equation*}
-A \llangle P_{-1,n+1}(x_{\nu})^{k} \rrangle_{\Sigma^{\prime}_{\nu}} = -A^{2} \llangle P_{0,n}(x_{\nu})^{k} \rrangle_{\Sigma^{\prime}_{\nu}} - \llangle P_{-2,n}(x_{\nu})^{k} \rrangle_{\Sigma^{\prime}_{\nu}}
\end{equation*}
and
\begin{equation*}
\llangle P_{0,n+1}(x_{\nu})^{k} \rrangle_{\Sigma^{\prime}_{\nu}} = A \llangle P_{1,n}(x_{\nu})^{k} \rrangle_{\Sigma^{\prime}_{\nu}} + A^{-1} \llangle P_{-1,n}(x_{\nu})^{k} \rrangle_{\Sigma^{\prime}_{\nu}}.
\end{equation*}
Therefore,
\begin{equation}
\label{eqn:V_beta_2_case2b_lem_1}
\llangle x_{\nu+1}\lambda^{n}x_{\nu+1}(x_{\nu})^{k} \rrangle_{\Sigma^{\prime}_{\nu}} = -A^{2} \llangle P_{0,n}(x_{\nu})^{k} \rrangle_{\Sigma^{\prime}_{\nu}} + A \llangle P_{1,n}(x_{\nu})^{k} \rrangle_{\Sigma^{\prime}_{\nu}}.
\end{equation}

Analogously, by part f) in the definition of $\llangle \cdot \rrangle_{\Sigma^{\prime}_{\nu}}$,
\begin{equation*}
-A^{3}\llangle x_{\nu}\lambda^{n}x_{\nu+1}(x_{\nu})^{k} \rrangle_{\Sigma^{\prime}_{\nu}} = A^{2} \llangle \lambda P_{0,n}(x_{\nu})^{k} \rrangle_{\Sigma^{\prime}_{\nu}} - 2A^{3} \llangle P_{-1,n}(x_{\nu})^{k} \rrangle_{\Sigma^{\prime}_{\nu}} - A \llangle x_{\nu+1}\lambda^{n}(x_{\nu})^{k+1} \rrangle_{\Sigma^{\prime}_{\nu}},
\end{equation*}
and then using \eqref{eqn:Pnk_1},
\begin{equation*}
-A^{3}\llangle x_{\nu}\lambda^{n}x_{\nu+1}(x_{\nu})^{k} \rrangle_{\Sigma^{\prime}_{\nu}} = A \llangle P_{1,n}(x_{\nu})^{k} \rrangle_{\Sigma^{\prime}_{\nu}} - A^{3} \llangle P_{-1,n}(x_{\nu})^{k} \rrangle_{\Sigma^{\prime}_{\nu}} - A \llangle x_{\nu+1}\lambda^{n}(x_{\nu})^{k+1} \rrangle_{\Sigma^{\prime}_{\nu}}.
\end{equation*}
Since by part j) and k) in the definition $\llangle \cdot \rrangle_{\Sigma^{\prime}_{\nu}}$, 
\begin{equation*}
- A \llangle x_{\nu+1}\lambda^{n}(x_{\nu})^{k+1} \rrangle_{\Sigma^{\prime}_{\nu}}
= A^{4} \llangle x_{\nu}\lambda^{n}(x_{\nu})^{k+1} \rrangle_{\Sigma^{\prime}_{\nu}}
= A^{3} \llangle P_{-1,n}(x_{\nu})^{k} \rrangle_{\Sigma^{\prime}_{\nu}} - A^{2} \llangle P_{0,n}(x_{\nu})^{k} \rrangle_{\Sigma^{\prime}_{\nu}}.
\end{equation*}
Therefore,
\begin{equation}
\label{eqn:V_beta_2_case2b_lem_2}
-A^{3}\llangle x_{\nu}\lambda^{n}x_{\nu+1}(x_{\nu})^{k} \rrangle_{\Sigma^{\prime}_{\nu}} = -A^{2} \llangle P_{0,n}(x_{\nu})^{k} \rrangle_{\Sigma^{\prime}_{\nu}} + A \llangle P_{1,n}(x_{\nu})^{k} \rrangle_{\Sigma^{\prime}_{\nu}}.
\end{equation}
The equation \eqref{eqn:V_beta_2_case2b_lem} follows from \eqref{eqn:V_beta_2_case2b_lem_1} and \eqref{eqn:V_beta_2_case2b_lem_2}.
\end{proof}

\begin{figure}[ht]
\centering
\includegraphics[scale=0.85]{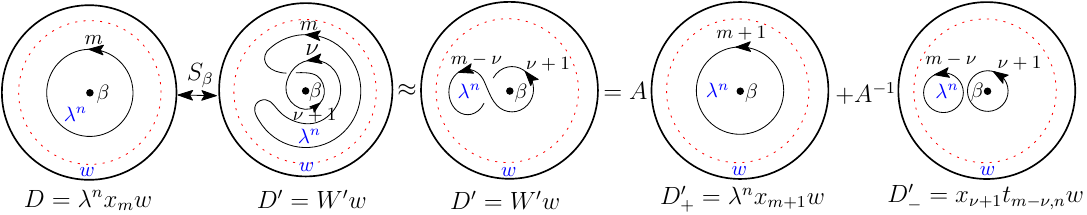}
\caption{Arrow diagrams $D$ and $D'$ in ${\bf D}^{2}_{\beta}$ related by an $S_{\beta}$-move}
\label{fig:V_beta_2_case2}
\end{figure}

For arrow diagrams $D$, $D'$ in Figure~\ref{fig:V_beta_2_case2}, we see that $u = \lambda^{n}x_{m}$, and $W'_{+} = \lambda^{n}x_{m+1}$, $W'_{-} = x_{\nu+1}t_{m-\nu,n}$ are obtained by smoothing crossing of $W'$ according to positive and negative markers. Hence, by \eqref{eqn:gDD'}, $\phi_{\beta}(D-D') = 0$ is equivalent to \eqref{eqn:Prop_h4_2} in the following lemma.

\begin{lemma}
\label{lem:V_beta_2_case2}
For $w\in \Gamma$, $m\in \mathbb{Z}$, and $n\geq 0$,
\begin{equation}
\label{eqn:Prop_h4_2}
\llangle \lambda^{n}x_{m}w - A\lambda^{n}x_{m+1}w - A^{-1}x_{\nu+1}P_{m-\nu,n}w \rrangle_{\Sigma^{\prime}_{\nu}} = 0.
\end{equation}
\end{lemma}

\begin{proof}
To prove \eqref{eqn:Prop_h4_2}, it suffices to show that for any fixed $m$,
\begin{equation}
\label{eqn:V_beta_2_case2}
\llangle \lambda^{n}x_{m}w \rrangle_{\Sigma^{\prime}_{\nu}} = A\llangle \lambda^{n}x_{m+1}w \rrangle_{\Sigma^{\prime}_{\nu}} + A^{-1}\llangle x_{\nu+1}P_{m-\nu,n}w \rrangle_{\Sigma^{\prime}_{\nu}}.
\end{equation}
We may assume that $w = \lambda^{n'}(x_{\nu+1})^{\varepsilon}(x_{\nu})^{k'}$ with $\varepsilon = 0$ or $1$ by using similar arguments as at the beginning of our proof for Lemma~\ref{lem:loc_prop_bracket_c}. By Lemma~\ref{lem:annulus_for_any_kn_bracket_2_beta}, it suffices to show the case $n = n' = 0$. 

When $\varepsilon = 0$, \eqref{eqn:V_beta_2_case2} becomes 
\begin{equation}
\label{eqn:V_beta_2_case2a}
\llangle x_{m}(x_{\nu})^{k'} \rrangle_{\Sigma^{\prime}_{\nu}} = A\llangle x_{m+1}(x_{\nu})^{k'} \rrangle_{\Sigma^{\prime}_{\nu}} + A^{-1}\llangle x_{\nu+1}P_{m-\nu}(x_{\nu})^{k'} \rrangle_{\Sigma^{\prime}_{\nu}}.
\end{equation}
By \eqref{eqn:rel_xm_bracket_2_beta_1} and then by part j) in the definition of $\llangle \cdot \rrangle_{\Sigma^{\prime}_{\nu}}$,
\begin{eqnarray*}
\llangle x_{m}(x_{\nu})^{k'} \rrangle_{\Sigma^{\prime}_{\nu}} &=& -A^{m-\nu}\llangle x_{\nu}Q_{m-\nu-1}(x_{\nu})^{k'} \rrangle_{\Sigma^{\prime}_{\nu}} + A^{m-\nu-1}\llangle x_{\nu+1}Q_{m-\nu}(x_{\nu})^{k'} \rrangle_{\Sigma^{\prime}_{\nu}} \\
&=& A^{m-\nu-3}\llangle x_{\nu+1}Q_{m-\nu-1}(x_{\nu})^{k'} \rrangle_{\Sigma^{\prime}_{\nu}} - A^{m-\nu+2}\llangle x_{\nu}Q_{m-\nu}(x_{\nu})^{k'} \rrangle_{\Sigma^{\prime}_{\nu}}.
\end{eqnarray*}
Moreover, by \eqref{eqn:rel_xm_bracket_2_beta_1} and \eqref{eqn:rel_Pn}, respectively, we see that
\begin{equation*}
A\llangle x_{m+1}(x_{\nu})^{k'} \rrangle_{\Sigma^{\prime}_{\nu}} = -A^{m-\nu+2}\llangle x_{\nu}Q_{m-\nu}(x_{\nu})^{k'} \rrangle_{\Sigma^{\prime}_{\nu}} + A^{m-\nu+1}\llangle x_{\nu+1}Q_{m-\nu+1}(x_{\nu})^{k'} \rrangle_{\Sigma^{\prime}_{\nu}}
\end{equation*}
and
\begin{equation*}
A^{-1}\llangle x_{\nu+1}P_{m-\nu}(x_{\nu})^{k'} \rrangle_{\Sigma^{\prime}_{\nu}} = -A^{m-\nu+1}\llangle x_{\nu+1}Q_{m-\nu+1}(x_{\nu})^{k'} \rrangle_{\Sigma^{\prime}_{\nu}} + A^{m-\nu-3}\llangle x_{\nu+1}Q_{m-\nu-1}(x_{\nu})^{k'} \rrangle_{\Sigma^{\prime}_{\nu}}.
\end{equation*}
Therefore, \eqref{eqn:V_beta_2_case2a} follows from the above equations.

When $\varepsilon = 1$, \eqref{eqn:V_beta_2_case2} becomes 
\begin{equation}
\label{eqn:V_beta_2_case2b}
\llangle x_{m}x_{\nu+1}(x_{\nu})^{k'} \rrangle_{\Sigma^{\prime}_{\nu}} = A\llangle x_{m+1}x_{\nu+1}(x_{\nu})^{k'} \rrangle_{\Sigma^{\prime}_{\nu}} + A^{-1}\llangle x_{\nu+1}P_{m-\nu}x_{\nu+1}(x_{\nu})^{k'} \rrangle_{\Sigma^{\prime}_{\nu}}.
\end{equation}
By \eqref{eqn:rel_xm_bracket_2_beta_1} and then by using Lemma~\ref{lem:V_beta_2_case2b_lem},
\begin{eqnarray*}
\llangle x_{m}x_{\nu+1}(x_{\nu})^{k'} \rrangle_{\Sigma^{\prime}_{\nu}} &=& -A^{m-\nu} \llangle x_{\nu}Q_{m-\nu-1}x_{\nu+1}(x_{\nu})^{k'} \rrangle_{\Sigma^{\prime}_{\nu}} + A^{m-\nu-1} \llangle x_{\nu+1}Q_{m-\nu}x_{\nu+1}(x_{\nu})^{k'} \rrangle_{\Sigma^{\prime}_{\nu}} \\
&=& A^{m-\nu-3} \llangle x_{\nu+1}Q_{m-\nu-1}x_{\nu+1}(x_{\nu})^{k'} \rrangle_{\Sigma^{\prime}_{\nu}} - A^{m-\nu+2} \llangle x_{\nu}Q_{m-\nu}x_{\nu+1}(x_{\nu})^{k'} \rrangle_{\Sigma^{\prime}_{\nu}}.
\end{eqnarray*}
Moreover, by \eqref{eqn:rel_xm_bracket_2_beta_1} and \eqref{eqn:rel_Pn}, respectively, we see that
\begin{equation*}
A\llangle x_{m+1}x_{\nu+1}(x_{\nu})^{k'} \rrangle_{\Sigma^{\prime}_{\nu}} = -A^{m-\nu+2} \llangle x_{\nu}Q_{m-\nu}x_{\nu+1}(x_{\nu})^{k'} \rrangle_{\Sigma^{\prime}_{\nu}} + A^{m-\nu+1} \llangle x_{\nu+1}Q_{m-\nu+1}x_{\nu+1}(x_{\nu})^{k'} \rrangle_{\Sigma^{\prime}_{\nu}}
\end{equation*}
and
\begin{eqnarray*}
A^{-1}\llangle x_{\nu+1}P_{m-\nu}x_{\nu+1}(x_{\nu})^{k'} \rrangle_{\Sigma^{\prime}_{\nu}} 
&=& -A^{m-\nu+1} \llangle x_{\nu+1}Q_{m-\nu+1}x_{\nu+1}(x_{\nu})^{k'} \rrangle_{\Sigma^{\prime}_{\nu}} \\
&+& A^{m-\nu-3} \llangle x_{\nu+1}Q_{m-\nu-1}x_{\nu+1}(x_{\nu})^{k'} \rrangle_{\Sigma^{\prime}_{\nu}}.
\end{eqnarray*}
Therefore, \eqref{eqn:V_beta_2_case2b} follows from the above equations.
\end{proof}

\begin{figure}[ht]
\centering
\includegraphics[scale=0.8]{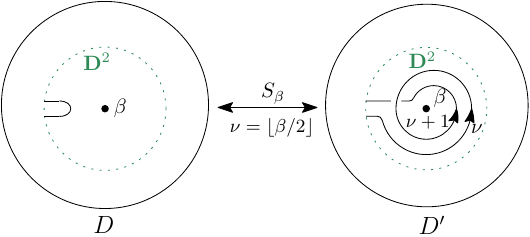}
\caption{$D$ and $D'$ in on ${\bf D}^{2}_{\beta}$ related by $S_{\beta}$-move.}
\label{fig:DiagramsBySBetaMove}
\end{figure}

\begin{lemma}
\label{lem:Main_Lemma_For_VB2}
If arrow diagrams $D$ and $D'$ in ${\bf D}^{2}_{\beta}$ are related by $S_{\beta}$-move then
\begin{equation*}
\phi_{\beta}(D-D') = 0.
\end{equation*}
\end{lemma}

\begin{proof}
Let $D$ and $D'$ be arrow diagrams in ${\bf D}^{2}_{\beta}$ related by $S_{\beta}$-move on a $2$-disk ${\bf D}^{2}$ centered at $\beta$ (see Figure~\ref{fig:DiagramsBySBetaMove}). We denote by $\mathcal{K}(D)$ and $\mathcal{K}(D')$ their corresponding sets of Kauffman states. As shown in Figure~\ref{fig:CasesDiagramsBySBetaMove} Kauffman states $s\in \mathcal{K}(D)$ are in bijection with pairs of Kauffman states $s_{+},s_{-}\in \mathcal{K}(D')$. Moreover, $s$ and $s_{+},s_{-}$ are related as follows
\begin{equation*}
p(s_{+})-n(s_{+}) = p(s)-n(s)+1 \quad \text{and} \quad p(s_{-})-n(s_{-}) = p(s)-n(s)-1,
\end{equation*}
and we denote by $D_{s}$, $D_{s_{+}}$, and $D_{s_{-}}$ the arrow diagrams corresponding $s$ and $s_{+},s_{-}$, respectively. Then,
\begin{equation*}
\llangle D-D'\rrangle = \sum_{s\in \mathcal{K}(D)}A^{p(s)-n(s)}\langle D_{s} - A D'_{s+} - A^{-1} D'_{s-}\rangle. 
\end{equation*}

\begin{figure}[ht]
\centering
\includegraphics[scale=0.8]{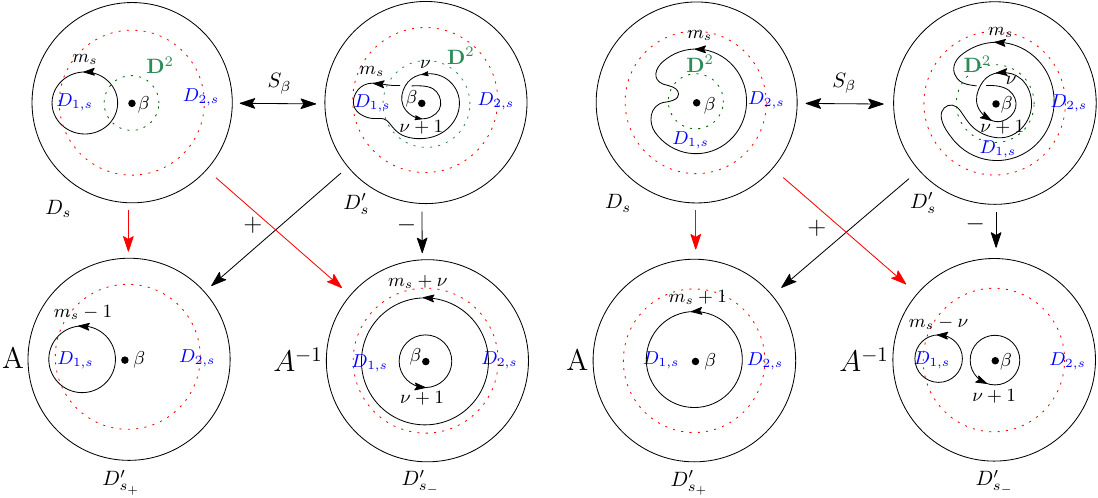}
\caption{$D_{s}$ and $D'_{s}$ related by an $S_{\beta}$-move on ${\bf D}^{2}_{\beta}$}
\label{fig:CasesDiagramsBySBetaMove}
\end{figure}

For $D_{1,s}$ and $D_{2,s}$ in Figure~\ref{fig:CasesDiagramsBySBetaMove}, suppose that
\begin{equation*}
\langle D_{1,s} \rangle_{r} = \sum_{i=0}^{n_{1}}r_{1,i}\lambda^{i} \quad\text{and}\quad \llangle \llangle D_{2,s} \rrangle \rrangle_{\Gamma} =  \sum_{j = 0}^{n_{2}}r_{2,j} w_{j}(s),    
\end{equation*}
then for the arrow diagrams in the left of Figure~\ref{fig:CasesDiagramsBySBetaMove},
\begin{equation*}
\llangle D_{s} - A D'_{s_{+}} - A^{-1} D'_{s_{-}} \rrangle_{\Gamma} = \sum_{i = 0}^{n_{1}} \sum_{j = 0}^{n_{2}} r_{1,i} r_{2,j} (P_{m_{s},i} - AP_{m_{s}-1,i}-A^{-1}x_{\nu+1}\lambda^{i}x_{m_{s}+\nu}) w_{j}(s)
\end{equation*}
and for the arrow diagrams in the right of Figure~\ref{fig:CasesDiagramsBySBetaMove}
\begin{equation*}
\llangle D_{s} - A D'_{s_{+}} - A^{-1} D'_{s_{-}} \rrangle_{\Gamma} = \sum_{i = 0}^{n_{1}} \sum_{j = 0}^{n_{2}} r_{1,i} r_{2,j} (\lambda^{i}x_{m_{s}} - A\lambda^{i}x_{m_{s}+1}-A^{-1}x_{\nu+1}P_{m_{s}-\nu,i}) w_{j}(s).
\end{equation*}
Therefore, it suffices to show that
\begin{eqnarray*}
&&\llangle P_{m_{s},i}w_{j}(s) - AP_{m_{s}-1,i}w_{j}(s)-A^{-1}x_{\nu+1}\lambda^{i}x_{m_{s}+\nu}w_{j}(s) \rrangle_{\Sigma^{\prime}_{\nu}} = 0 \quad \text{and} \\
&&\llangle \lambda^{i}x_{m_{s}}w_{j}(s) - A\lambda^{i}x_{m_{s}+1}w_{j}(s)-A^{-1}x_{\nu+1}P_{m_{s}-\nu,i}w_{j}(s) \rrangle_{\Sigma^{\prime}_{\nu}} = 0.
\end{eqnarray*}
However, the above identities follow from \eqref{eqn:Prop_h4_1} and \eqref{eqn:Prop_h4_2}, respectively.
\end{proof}

\begin{theorem}
\label{thm:basis_for_V(beta,2)}    
KBSM of $V(\beta,2)$ is a free $R$-module with basis consisting ambient isotopy classes of generic framed links $V(\beta,2)$ whose arrow diagrams are in $\Sigma^{\prime}_{\nu}$, i.e.,
\begin{equation*}
\mathcal{S}_{2,\infty}(V(\beta,2);R,A) \cong R\Sigma^{\prime}_{\nu}.
\end{equation*}
\end{theorem}

\begin{proof}
The statement follows by arguments analogous to those in our proof of Theorem~\ref{thm:basis_Sigma_c_for_A2S1}. Specifically, we notice that by Lemma~\ref{lem:lemma_for_h4} and Lemma~\ref{lem:Main_Lemma_For_VB2} the map $\phi_{\beta}$ yields a well-defined homomorphism of free $R$-modules
\begin{equation*}
\phi_{\beta}: R\mathcal{D}({\bf D}^{2}_{\beta}) \to  R\Sigma^{\prime}_{\nu},
\end{equation*}
which descends to an isomorphism of free $R$-modules
\begin{equation*}
\hat{\phi}_{\beta}: S\mathcal{D}({\bf D}^{2}_{\beta}) \to  R\Sigma^{\prime}_{\nu},\, \hat{\phi}_{\beta}([D]) = \phi_{\beta}(D) 
\end{equation*}
and then we apply Theorem~\ref{thm:AmbientIsotopiesInVBeta2}.
\end{proof}

\begin{remark}
\label{rem:Final_Rem_Fib_Torus}
Fibered torus $V(\beta,2)$ is homeomorphic to solid torus $V^{3}$ and a basis for the KBSM of $V^{3}$ is well-known (see \cite{Prz1999}). Therefore, in Theorem~\ref{thm:basis_for_V(beta,2)} we just constructed a family of new basis for $V^{3}$. In follow-up papers, we will use the basis $\Sigma^{\prime}_{\nu}$ to compute the KBSM of certain families of small Seifert fibered $3$-manifolds.
\end{remark}

\section*{Acknowledgement}  Authors would like to thank Professor J\'{o}zef H. Przytycki for all valuable discussions and suggestions.

\bibliography{mybibfile}

\begin{bibdiv}
\begin{biblist}

\bib{GM2017}{inproceedings}{
      author={Gabrov{\v{s}}ek, Bo{\v{s}}tjan},
      author={Mroczkowski, Maciej},
       title={Link {D}iagrams in {S}eifert {M}anifolds and {A}pplications to
  {S}kein {M}odules},
        date={2017},
   booktitle={Algebraic {M}odeling of {T}opological and {C}omputational
  {S}tructures and {A}pplications},
      editor={Lambropoulou, Sofia},
      editor={Theodorou, Doros},
      editor={Stefaneas, Petros},
      editor={Kauffman, Louis~H.},
   publisher={Springer {I}nternational {P}ublishing},
     address={Cham},
       pages={117\ndash 141},
  note={E-print:\href{https://arxiv.org/abs/1802.03645}{arXiv:1802.03645}[math.GT]},
}

\bib{Hud1969}{book}{
      author={Hudson, John~F.P.},
       title={Piecewise {L}inear {T}opology},
   publisher={W.A. Benjamin, Inc., New York-Amsterdam},
        date={1969},
}

\bib{KLH1987}{article}{
      author={Kauffman, Louis~H.},
       title={State models and the jones polynomial},
        date={1987},
     journal={Topology},
      volume={26},
      number={3},
       pages={395\ndash 407},
}

\bib{MD2009}{article}{
      author={Mroczkowski, Maciej},
      author={Dabkowski, Mieczyslaw~K.},
       title={{KBSM} of the product of a disk with two holes and {$S^{1}$}},
        date={2011},
     journal={Topology and its Applications},
      volume={156},
      number={10},
       pages={1831\ndash 1849},
        note={\href{https://arxiv.org/abs/0808.3782}{arXiv:0808.3782}
  [math.GT]},
}

\bib{Prz1991}{article}{
      author={Przytycki, J\'{o}zef~H.},
       title={Skein modules of $3$-manifolds},
        date={1991},
     journal={Bull. Polish Acad. Sci. Math.},
      volume={39},
      number={1-2},
       pages={91\ndash 100},
        note={E-print:
  \href{https://arxiv.org/abs/math/0611797}{arXiv:math/0611797} [math.GT]},
}

\bib{Prz1999}{article}{
      author={Przytycki, J\'{o}zef~H.},
       title={Fundamentals of {K}auffman bracket skein modules},
        date={1999},
     journal={Kobe J. Math.},
      volume={16},
      number={1},
       pages={45\ndash 66},
        note={E-print:
  \href{https://arxiv.org/abs/math/9809113}{arXiv:math/9809113} [math.GT]},
}

\bib{Tur1990}{article}{
      author={Turaev, Vladimir~G.},
       title={The {C}onway and {K}auffman modules of a solid torus},
        date={1990},
     journal={J. Soviet Math.},
      volume={52},
      number={1},
       pages={2799\ndash 2805},
}

\bib{Tur1992}{article}{
      author={Turaev, Vladimir~G.},
       title={Shadow links and face models of statistical mechanics},
        date={1992},
     journal={J. Differ. Geom.},
      volume={36},
      number={1},
       pages={35\ndash 74},
}

\end{biblist}
\end{bibdiv}

\end{document}